\documentclass[12pt,leqno]{article}

\usepackage[T1]{fontenc}
\usepackage[utf8]{inputenc}

\usepackage[a4paper]{geometry}
\usepackage{amssymb}
\usepackage{amsmath}
\usepackage{amsthm}

\usepackage{enumitem}

\usepackage[all]{xy}
\usepackage[usenames,dvipsnames]{color}
\usepackage{tikz}

\usepackage[final,linktocpage=true,colorlinks=true,linkcolor=OliveGreen,citecolor=OliveGreen,urlcolor=blue]{hyperref}

\title{Gromov boundaries as Markov compacta}

\author{Dominika Pawlik} 
\theoremstyle{definition}
\newtheorem{df}{Definition}[section]
\newtheorem{uwaga}[df]{Remark}

\theoremstyle{plain}
\newtheorem{tw}[df]{Theorem}
\newtheorem*{tw*}{Theorem}
\newtheorem*{twsm}{Theorem~\ref*{tw-semi-markow-0}}
\newtheorem{fakt}[df]{Lemma}
\newtheorem{wn}[df]{Corollary}
\newtheorem{lem}[df]{Proposition}
\newtheorem{ozn}[df]{Denotation}

\numberwithin{equation}{section}

\DeclareMathOperator{\diam}{diam}

\DeclareMathOperator{\rank}{rank}
\DeclareMathOperator{\sppan}{span}

\DeclareMathOperator{\innt}{int}
\DeclareMathOperator{\intr}{int}

\newcommand{\dzm}[1]{#1}

\newcommand{\qedthm}{\pushQED{\qed}\qedhere\popQED}

\newcommand{\qhitem}[3]{\expandafter\xdef\csname qhlabel#1\endcsname{#3}\item[(#2)]\hypertarget{qh-#1}{}}
\newcommand{\qhlink}[1]{(\hyperlink{qh-#1}{\csname qhlabel#1\endcsname})}

\newcommand{\divs}{\mathrel{|}}
\newcommand{\liminv}{\mathop{\lim}\limits_{\longleftarrow}}

\let\oldUparrow\Uparrow
\renewcommand{\Uparrow}{{\,\oldUparrow}}

\begin{document}

\maketitle

\begin{abstract}
We prove that the Gromov boundary of every hyperbolic group is homeomorphic to some Markov compactum. Our reasoning is based on constructing a~sequence of covers of~$\partial G$, which is \textit{quasi-$G$-invariant} wrt. the \textit{ball $N$-type} (defined by Cannon) for $N$ sufficiently large.
We also ensure certain additional properties for the inverse system representing~$\partial G$, leading to a~finite description which defines it uniquely. 

By defining a~natural metric on the inverse limit $\liminv K_n$ and proving it to be bi-Lipschitz equivalent to an accordingly chosen visual metric on~$\partial G$, we prove that our construction enables providing a~simplicial description of the natural quasi-conformal structure on~$\partial G$. We also point out that the initial system of covers can be modified so that all the simplices in the resulting inverse system are of dimension less than or equal to $\dim \partial G$. 

We also generalize --- from the torsion-free case to all finitely generated hyperbolic groups --- a~theorem guaranteeing the existence of~a~finite representation of~$\partial G$ of another kind, namely a~\textit{semi-Markovian structure} (which can be understood as an analogue of the well-known automatic structure of $G$ itself).
\end{abstract}

In this paper, we consider Gromov boundaries of hyperbolic groups and their presentations of a~combinatorial nature, describing the topology of~$\partial G$ in a ``recursively finite'' manner.

Our main goal will be to present the space $\partial G$ as a \textit{Markov compactum} (see Definition~\ref{def-kompakt-markowa}). This notion, introduced in~\cite{Dra} by Dranishnikov (who claims that its basic idea comes from Gromov), refers to inverse limits of systems of polyhedra with certain finiteness conditions for the bounding maps (which we will call \textit{Markov systems}). The existence of such presentation is generally known (\cite{Kap-prob}, \cite{JS-trees}); however, it seems that no explicit proof of this fact has been given so far. 

In fact, as we explain in Remark~\ref{uwaga-sk-opis}, a Markov system is \textit{finitely describable}, i.e. it can be uniquely determined out of a~finite set of data (using an appropriate algorithmic procedure), provided that it satisfies several additional conditions, which we call \textit{barycentricity} and \textit{distinct types property} (see Definitions~\ref{def-kompakt-barycentryczny} and~\ref{def-kompakt-wlasciwy}). This feature of Markov compacta seems to be important for potential applications.

In our construction of the Markov system, we will also ensure the \textit{mesh property} (see Definition~\ref{def-mesh}), considered already in~\cite{Dra}, which may be interpreted as a reasonable connection between the topology of a Markov compactum and the simplicial structure of the underlying Markov system.

The main result of the paper can be summarised as follows:


\begin{tw}
\label{tw-kompakt}
Let $G$ be a~hyperbolic group. Then, $\partial G$ is homeomorphic to a Markov compactum $\liminv K_i$ defined by a Markov system $(K_i, f_i)_{i \geq 0}$. Moreover, we can require (simultaneously) that:
\begin{itemize}
 \item the system $(K_i, f_i)_{i \geq 0}$ is barycentric and satisfies distinct types property and mesh property;
 \item the dimensions of the complexes $K_i$ are bounded from above by the topological dimension $\dim \partial G$.
\end{itemize}
\end{tw}

Apart from the above topological and combinatorial properties, we also claim that the above presentation is well-behaved with respect to the metric. More precisely, we associate to every Markov system a (rather natural) family of \textit{simplicial metrics} $d_a^M$ (for all $a > 1$) on the limit compactum (see Definition~\ref{def-metryka-komp}). Then, we show that these metrics correspond to the Gromov \textit{visual metrics} $d_v^{(a)}$ on the boundary $\partial G$ (see Section~\ref{sec-def-hip}), as follows:


\begin{tw}
\label{tw-bi-lip-0}
 Let $G$ be a~hyperbolic group and $(K_i, f_i)_{i \geq 0}$ be the Markov system representing $\partial G$, constructed as in the proof of Theorem~\ref{tw-kompakt}. Then, the quasi-conformal structures on $\partial G$ and $\liminv K_i$ defined respectively by the visual metrics $\{ d_v^{(a)} \}$ and the simplicial metrics $\{ d_a^M \}$ (for sufficiently small values of~$a$) coincide.
More precisely, for $a > 1$ sufficiently close to $1$, the metric spaces $(\partial G, d_v^{(a)})$ and $(\liminv K_i, d_a^M)$ are bi-Lipschitz equivalent. 
\end{tw}

Another kind of ``finitely recursive'' description of $\partial G$ is the structure of a \textit{semi-Markovian space} (see Definition~\ref{def-sm-ps}), a~notion introduced in~\cite{zolta} (with its main idea coming from Gromov), intuitively being a~strengthened analogue of infinite-word automata. In this topic, we prove the following theorem:

\begin{tw}
\label{tw-semi-markow-0}
The boundary of any hyperbolic group $G$ is a semi-Markovian space.
\end{tw}

This generalises the main result of~\cite{zolta}, where such presentation has been constructed for torsion-free hyperbolic groups. The outline of our reasoning is based on the proof from~\cite{zolta}; our crucial improvement comes from a \textit{type strengthening} technique described in Section~\ref{sec-abc} (which is also used in the proof of Theorem~\ref{tw-kompakt}). We also rectify a mistake in the proof from~\cite{zolta} (which is precisely indicated in Remark~\ref{uwaga-zolta}).

\subsubsection*{Organisation of the paper}

The paper is organised as follows. In Section~\ref{sec-def}, we briefly recall the basic properties of hyperbolic groups which we will need, and we introduce Markov compacta. Section~\ref{sec-typy} contains auxiliary facts regarding mainly the \textit{conical} and \textit{ball types} in hyperbolic groups (defined in~\cite{CDP}), which will be the key tool in the proof of Theorem~\ref{tw-kompakt}.

The main claim of Theorem~\ref{tw-kompakt} is obtained by constructing an appropriate family of covers of $\partial G$ in Section~\ref{sec-konstr}, considering the corresponding inverse sequence of nerves in Section~\ref{sec-engelking-top} and finally verifying the Markov property in Section~\ref{sec-markow}. An outline of this reasoning is given in the introduction to Section~\ref{sec-konstr}, and its summary appears in Section~\ref{sec-markow-podsum}. Meanwhile, we give the proof of Theorem~\ref{tw-bi-lip-0} (as a~corollary of Theorem~\ref{tw-bi-lip}) in Section~\ref{sec-bi-lip}, mostly by referring to the content of Section~\ref{sec-engelking-top}.

While the Markov system obtained at the end of Section~\ref{sec-markow} will be already barycentric and have mesh property, in Sections~\ref{sec-abc} and \ref{sec-wymd} we focus respectively on ensuring distinct types property and bounding the dimensions of involved complexes. This will lead to a complete proof of Theorem~\ref{tw-kompakt}, summarised in Section~\ref{sec-wymd-podsum}. 

Finally, Section~\ref{sec-sm} contains the proof of Theorem~\ref{tw-semi-markow-0}; its content is basically unrelated to Sections~\ref{sec-konstr}--\ref{sec-wymd}, except for that we re-use the construction of \textit{$B$-type} from Section~\ref{sec-sm-abc-b}.


\subsubsection*{Acknowledgements}

I would like to thank my supervisors Jacek Świątkowski and Damian Osajda for recommending this topic, for helpful advice and inspiring conversations, and Aleksander Zabłocki for all help and careful correcting.

\section{Introduction}
\label{sec-def}

\subsection{Hyperbolic groups and their boundaries}
\label{sec-def-hip}

Throughout the whole paper, we assume that $G$ is a~hyperbolic group in the sense of~Gromov~\cite{G}. We implicitly assume that $G$ is equipped with a~fixed, finite generating set~$S$, and we identify~$G$ with its Cayley graph~$\Gamma(G, S)$. As a~result, we will often speak about ``distance in~$G$'' or ``geodesics in~$G$'', referring in fact to the Cayley graph. Similarly, the term ``dependent only on~$G$'' shall be understood so that dependence on~$S$ is also allowed. By~$\delta$ we denote some fixed constant such that $\Gamma(G, S)$ is a $\delta$-hyperbolic metric space; we assume w.l.o.g. that $\delta \geq 1$.

We denote by~$e$ the identity element of~$G$, and by~$d(x, y)$ the distance of elements $x, y \in G$. The distance $d(x, e)$ will be called the \textit{length} of~$x$ and denoted by~$|x|$. 
We use a notational convention that $[x, y]$ denotes a geodesic segment between the points $x, y \in G$, that is, an isometric embedding $\alpha : [0, n] \cap \mathbb{Z} \rightarrow G$ such that $\alpha(0) = x$ and $\alpha(n) = y$, where $n$ denotes $d(x, y)$. In the sequel, geodesic segments as well as geodesic rays and bi-infinite geodesic paths (i.e. isometric embeddings resp. of~$\mathbb{N}$ and $\mathbb{Z}$) will be all refered to as ``geodesics in $G$''; to specify which kind of geodesic is meant (when unclear from context), we will use adjectives \textit{finite}, \textit{infinite} and \textit{bi-infinite}.

We denote by~$\partial G$ the Gromov boundary of~$G$, defined as in~\cite{Kap}. We recall after \cite[Chapter 1.3]{zolta} that, as a~set, it is the~quotient of the set of all infinite geodesic rays in~$G$ by the relation of being close:
\[ (x_n) \sim (y_n) \qquad \Leftrightarrow \qquad \exists_{C > 0} \ \forall_{n \geq 0} \ d(x_n, y_n) < C; \]
moreover, in the above definition one can equivalently assume that $C = 4\delta$. It is also known that the topology defined on $\partial G$ is compact, preserved by the natural action of~$G$, and compatible with a~family of \textit{visual metrics}, defined depending on a~parameter~$a > 1$ with values sufficiently close to~$1$. Although we will not refer directly to the definition and properties of these metrics, we will use an estimate stated as (P2) in~\cite[Chapter~1.4]{zolta} which guarantees that, for every sufficiently small~$a > 1$, the visual metric with parameter~$a$ (which we occasionally denote by $d_v^{(a)}$) is bi-Lipschitz equivalent to the following \textit{distance function}:
\[ d_a \big( p, q \big) = a^{-l} \qquad \textrm{ for } p, q \in \partial G, \]
where $l$ is the largest possible distance between $e$ and any bi-infinite geodesic in~$G$ joining $p$ with $q$. As we will usually work with a~fixed value of~$a$, we will drop it in the notation.

For $x, y \in G \cup \partial G$, the symbol $[x, y]$ will denote \textit{any} geodesic in~$G$ joining~$x$ with~$y$. We will use the following fact from~\cite[Chapter~1.3]{zolta}:

\begin{fakt}
\label{fakt-waskie-trojkaty}
 Let $\alpha, \beta, \gamma$ be the sides of a~geodesic triangle in~$G$ with vertices in~$G \cup \partial G$. Then, $\alpha$ is contained in the $4(p+1)\delta$-neighbourhood of $\beta \cup \gamma$, where $p$ is the number of vertices of the triangle which lie in~$\partial G$.
\end{fakt}

\subsection{Markov compacta}

\label{sec-def-markow}

\begin{df}[{\cite[Definition~1.1]{Dra}}]
\label{def-kompakt-markowa}
Let $(K_i, f_i)_{i \geq 0}$ be an inverse system consisting of the spaces $K_i$ and maps $f_i: K_{i + 1} \rightarrow K_i$ for $i \geq 0$.
Such system will be called \textit{Markov} (or said to satisfy \textit{Markov property}) if the following conditions hold:
\begin{itemize}
 \item[(i)] $K_i$ are finite simplicial complexes which satisfy the inequality $\sup \dim K_i < \infty$;
 \item[(ii)] for every simplex $\sigma$, in $K_{i+1}$ its image $f_i(\sigma)$ is contained in some simplex belonging to $K_i$ and the restriction $f_i|_\sigma$ is an affine map;
  \item[(iii)] simplexes in $\amalg_i K_i$ can be assigned finitely many \textit{types} so that for any simplexes $s \in K_i$ and~$s' \in K_j$ of the same type there exist isomorphisms of subcomplexes $i_k : (f^{i+k}_i)^{-1}(s) \rightarrow (f^{j+k}_j)^{-1}(s')$ for $k \geq 0$ such that the following diagram commutes:
\begin{align}
\label{eq-markow-drabinka}
\xymatrix@+3ex{
 s \ar[d]_{i_0} & \ar[l]_{f_i} f_i^{-1}(s) \ar[d]_{i_1} & \ar[l] \ldots & \ar[l] (f^{i+k}_i)^{-1}(s) \ar[d]_{i_k} & \ar[l]_{f_{i+k}} (f^{i+k+1}_i)^{-1}(s) \ar[d]_{i_{k+1}} & \ar[l] \ldots \\
 s' & \ar[l]_{f_j} f_j^{-1}(s') & \ar[l] \ldots & \ar[l] (f^{j+k}_j)^{-1}(s') & \ar[l]_{f_{j+k}} (f^{j+k+1}_j)^{-1}(s') & \ar[l] \ldots
}
\end{align}
where $f^a_b$ (for $a \geq b$) means the composition $f_b \circ f_{b+1} \circ \ldots \circ f_{a-1} : K_a \rightarrow K_b$.
\end{itemize}
\end{df}

\begin{df}[{\cite[Definition~1.1]{Dra}}]
\label{def-kompakt-markowa2}
A topological space $X$ is a \textit{Markov compactum} if it is the inverse limit of a Markov system.
\end{df}

\begin{df}[{cf.~\cite[Lemma~2.3]{Dra}}]
\label{def-mesh}
\ 
\begin{itemize}
 \item[\textbf{(a)}] A sequence $(\mathcal{A}_n)_{n \geq 0}$  of families of subsets in a compact metric space has \textit{mesh property} if  \[ \lim_{n \rightarrow \infty} \, \max_{A \in \mathcal{A}_n} \diam A = 0. \]
 \item[\textbf{(b)}] An inverse system of polyhedra $(K_n, f_n)$ has \textit{mesh property} if, for any $i \geq 0$, the sequence $(\mathcal{F}_n)_{n \geq i}$ of families of subsets in $K_i$ has mesh property, where
 \[ \mathcal{F}_n = \big\{ f^n_i(\sigma) \ \big|\ \sigma \textrm{ is a simplex in }K_n \big\}. \] 
\end{itemize}
\end{df}

\begin{uwaga}
\label{uwaga-mesh-bez-metryki}
We can formulate Definition \ref{def-mesh}a in an equivalent way (regarding only the topology): for any open cover $\mathcal{U}$ of $X$ there exists $n \geq 0$ such that, for every $m \geq n$, every set $A \in \mathcal{A}_m$ is contained in some $U \in \mathcal{U}$. 
In particular, this means that the sense of Definition \ref{def-mesh}b does not depend on the choice of a metric (compatible with the topology) in $K_i$.
\end{uwaga}

\begin{df}
\label{def-kompakt-barycentryczny}
A Markov system $(K_i, f_i)$ is called \textit{barycentric} if, for any $i \geq 0$, the vertices of $K_{i+1}$ are mapped by $f_i$ to the vertices of the first barycentric subdivision of $K_i$.
\end{df}

\begin{df}
\label{def-kompakt-wlasciwy}
A Markov system $(K_i, f_i)$ has \textit{distinct types property} if for any $i \geq 0$ and any simplex $s \in K_i$ all simplexes in the pre-image $f_i^{-1}(s)$ have pairwise distinct types.
\end{df}

\begin{uwaga}
\label{uwaga-sk-opis}
A motivation for the above two definitions is the observation that barycentric Markov systems with distinct types property are \textit{finitely describable}. In more detail, if the system $(K_i, f_i)_{i \geq 0}$ satisfies the conditions from Definitions ~\ref{def-kompakt-markowa}, \ref{def-kompakt-barycentryczny} and~\ref{def-kompakt-wlasciwy}, and if $N$ is so large that complexes $K_0, \ldots, K_N$ contain simplexes of all possible types, then the full system $(K_i, f_i)_{i = 0}^\infty$ can by rebuilt on the base of the initial part of the system (which is finitely describable because of being barycentric).   
\[
\xymatrix@C+3ex{
K_0 & \ar[l]_{f_0} K_1 & \ar[l]_{f_1} \ldots & \ar[l]_{f_{N-1}} K_N & \ar[l]_{f_N} K_{N+1}.
}
\]
The proof is inductive: for any $n \geq N+1$ the complex $K_{n+1}$ with the map $f_n : K_{n+1} \rightarrow K_n$ is given uniquely by the subsystem $K_0 \longleftarrow \ldots \longleftarrow K_n$. This results from the following:
\begin{itemize}

 \item for any simplex $s \in K_n$ there exists a model simplex $\sigma \in K_m$ of the same type, where $m < n$, and then the pre-image $f_n^{-1}(s)$ together with the types of its simplexes and the restriction $f_n \big|_{f_n^{-1}(s)}$ is determined by the pre-image $f_{m}^{-1}(\sigma)$ and the restriction $f_m \big|_{f_m^{-1}(\sigma)}$ (which follows from Definition \ref{def-kompakt-markowa});
 \item for any pair of simplexes $s' \subseteq s \in K_n$ the choice of a type preserving injection $f_n^{-1}(s') \rightarrow f_n^{-1}(s)$ is uniquely determined by the fact that vertices in $f_n^{-1}(s)$ have pairwise distinct types (by Definition \ref{def-kompakt-wlasciwy});
 \item since $K_{n+1}$ is the union of the family of pre-images of the form $f_n^{-1}(s)$ for $s \in K_n$, which is closed with respect to intersecting, the knowledge of these pre-images and the type preserving injections between them is sufficient to recover $K_{n+1}$; obviously we can reconstruct~$f_n$ too, by taking the union of the maps $f_n^{-1}(s) \rightarrow s$ determined so far.
\end{itemize}
\end{uwaga}

\section{Types of elements of~$G$}

\label{sec-typy}

The goal of this section is to introduce the main properties of the \textit{cone types} (Definition \ref{def-typ-stozkowy}) and \textit{ball types} (Definition \ref{def-typ-kulowy}) for elements of a hyperbolic group $G$. These classical results will be used in the whole paper.

The connection between cone types (which describe the natural structure of the group and its boundary) and ball types (which are obviously only finite in number) in the group $G$ was described for the first time by Cannon in \cite{Cannon} and used to prove properties of the growth function for the group. This result turns out to be an important tool in obtaining various finite presentations of Gromov boundary: it is used  in \cite{CDP} to build an automatic structure on $\partial G$ and in \cite{zolta} to present $\partial G$ as a semi-Markovian space in torsion-free case (the goal of Section \ref{sec-sm} is to generalise this result to all groups). Therefore it is not surprising that we will use this method to build the structure of Markov compactum for the space $\partial G$.

\subsection{Properties of geodesics in~$G$}

\begin{fakt}
\label{fakt-geodezyjne-pozostaja-bliskie}
Let $\alpha = [e, x]$ and $\beta = [e, y]$, where $|x| = |y| = n$ and $d(x, y) = k$. Then, for $0 \leq m \leq n$, the following inequality holds:
\[ d \big( \alpha(m), \beta(m) \big) \ \leq \ 8\delta + \max \big( k + 8\delta - 2(n-m), \, 0 \big). \]
In particular, for $0 \leq m \leq n - \tfrac{k}{2} - 4\delta$, we have $d(\alpha(m), \beta(m)) \leq 8\delta$.
\end{fakt}

\begin{proof}
Let us consider the points $\alpha(m), \beta(m)$ lying on the sides of a $4\delta$-narrow geodesic triangle $[e, x, y]$. We will consider three cases.

If $\alpha(m)$ lies at distance at most $4\delta$ from $\beta$, then we have $d(\alpha(m), \beta(m')) \leq 4\delta$ for some $m'$, so from the triangle inequality in the triangle $[e, \alpha(m), \beta(m')]$ we obtain $|m' - m| \leq 4\delta$, so
\[ d(\alpha(m), \beta(m)) \leq d(\alpha(m), \beta(m')) + |m' - m| \leq 8\delta, \]
which gives the claim.

If $\beta(m)$ lies at distance at most $4\delta$ from $\alpha$, the reasoning is analogous.

It remains to consider the case when $\alpha(m)$, $\beta(m)$ are at distance at most $4\delta$ respectively from $a, b \in [x, y]$. Then, $|a|, |b| \leq m + 4\delta$, so $a$, $b$ are at distance at least $D = n - m - 4\delta$ from the both endpoints of $[x, y]$. Therefore, $d(a, b) \leq k - 2D$, and so
\[ d \big( \alpha(m), \beta(m) \big) \leq d \big( \alpha(m), a \big) + d(a, b) + d \big( b, \beta(m) \big) \leq 8\delta + k - 2D = 16\delta + k - 2(n - m). \qedhere \]
\end{proof}

\begin{wn}
\label{wn-krzywe-geodezyjne-pozostaja-bliskie}
Let $\alpha = [e, x]$ and $\beta = [e, y]$, with $|x| =  n$ and $d(x, y) = k$. Then, for $0 \leq m \leq \min(n, |y|)$, we have:
\[ d \big( \alpha(m), \beta(m) \big) \ \leq \ 8\delta + \max \big( 2k + 8\delta - 2(n-m), \, 0 \big). \]
\end{wn}

\begin{proof}
From the triangle inequality we have $\big| n - |y| \big| = \big| |x| - |y| \big| \leq k$. Let $n' = \min(n, |y|)$; we claim that $d(\alpha(n'), \beta(n')) \leq 2k$. Indeed: if $n' = n$, we have
\[ d(\alpha(n'), \beta(n')) \leq d(x, y) + d(y, \beta(n)) \leq k + \big| |y| - n \big| \leq 2k; \]
otherwise $n' = |y|$ and so
\[ d(\alpha(n'), \beta(n')) \leq d(\alpha(|y|), x) + d(x, y) \leq \big| |y| - n \big| + k \leq 2k. \]
It remains to use Lemma~\ref{fakt-geodezyjne-pozostaja-bliskie} for geodesics $\alpha, \beta$ restricted to the interval $[0, n']$ and the doubled value of~$k$.
\end{proof}

\begin{fakt}
\label{fakt-geodezyjne-przekatniowo}
Let $(\alpha_k)_{k \geq 0}$ be a sequence of geodesic rays in~$G$ which start at~$e$. Denote $x_k = \lim_{n \rightarrow \infty} \alpha_k(n)$. Then, there exists a subsequence $(\alpha_{k_i})_{i \geq 0}$ and a geodesic $\alpha_{\infty}$ such that $\alpha_{k_i}$ coincides with~$\alpha_{\infty}$ on the segment~$[0, i]$. Moreover, the point $x_{\infty} = \lim_{n \rightarrow \infty} \alpha_{\infty}(n)$ is the limit of $(x_{k_i})$.
\end{fakt}

\begin{proof}
The first part of the claim is obtained from an easy diagonal argument: since, for every $n \geq 0$, the set $\{ x \in G \,|\, |x| \leq n \}$ is finite, the set of possible restrictions $\{ \alpha_k \big|_{[0, n]} \,|\, k \geq 0 \}$ must be finite too. This allows to define inductively $\alpha_{\infty}$: we take $\alpha_\infty(0) = e$, and for the consecutive $n > 0$ we choose $\alpha_\infty(n)$ so that $\alpha_\infty$ coincides on $[0, n]$ with infinitely many among the~$\alpha_k$'s. Such choice is always possible and guarantees the existence of a subsequence $(\alpha_{k_i})$.

The obtained sequence $\alpha_\infty$ is a geodesic because every its initial segment $\alpha_\infty \big|_{[0, i]}$ coincides with an initial segment $\alpha_{k_i} \big|_{[0, i]}$ of a geodesic. (We note that we can obtain an increasing sequence $(k_i)$). In this situation, from Lemma~5.2.1 in~\cite{zolta} and the definition of the topology in $G \cup \partial G$ it follows that $x_{\infty} = \lim_{i \rightarrow \infty} x_{k_i}$ holds in~$\partial G$. 
On the other hand, we have $\gamma_\infty(k) = g$, and so $x = [\gamma_\infty] \in \sppan(g)$.
\end{proof}

\subsection{Cone types and their analogues in~$\partial G$}

\begin{df}[cf.~\cite{CDP}]
\label{def-typ-stozkowy}
We define the \textit{cone type} $T^c(x)$ of $x \in G$ as the set of all $y \in G$ such that there exists a geodesic connecting $e$ to $xy$ and passing through $x$.
\end{df}

Elements of the set $xT^c(x)$ will be called \textit{descendants} of~$x$.

\begin{fakt}[{\cite[Chapter~12.3]{CDP}}]
\label{fakt-przechodniosc-potomkow}
The relation of being a descendant is transitive: if $y \in T^c(x)$ and $w \in T^c(xy)$, then $yw \in T^c(x)$.
\end{fakt}

\begin{fakt}
\label{fakt-synowie-typy-stozkowe}
If $y \in T^c(x)$, then the cone type $T^c(xy)$ is determined by $T^c(x)$ and $y$. 
\end{fakt}

\begin{proof}
This results from multiple application of Lemma 12.4.3 in~\cite{CDP}.
\end{proof}

\begin{df}
The \textit{span} of an element $g \in G$ (denoted $\sppan(g)$) is the set of all $x \in \partial G$ such that there exists a geodesic from $e$ to $x$ passing through $g$. 
\end{df}

\begin{fakt}
The set $\sppan(g)$ is closed for every $g \in G$. 
\end{fakt}

\begin{proof}
Denote $|g| = k$ and let $x_i$ be a sequence in $\sppan(g)$ converging to $x \in \partial G$. We will show that $x$ also belongs to $\sppan(g)$. Let $\gamma_i$ be a geodesic in~$G$ starting in~$e$, converging to~$x_i$ and such that $\gamma_i(k) = g$. By Lemma~\ref{fakt-geodezyjne-przekatniowo}, there is a subsequence $(\gamma_{i_j})$ which is increasingly coincident with some geodesic~$\gamma_\infty$; in particular, we have $\gamma_\infty(0) = e$ and $\gamma_\infty(k) = g$. Moreover, Lemma~\ref{fakt-geodezyjne-przekatniowo} ensures that~$[\gamma_\infty] \in \partial G$ is the limit of~$x_{i_j}$, so it is equal to~$x$. This means that $x \in \sppan(g)$.
\end{proof}

\begin{fakt}
\label{fakt-stozek-a-span}
For any $g \in G$, $\sppan(g)$ is the set of limits in $\partial G$ of all geodesic rays in $G$ starting at $g$ and contained in $gT^c(g)$.
\end{fakt}

\begin{proof}
Denote $|g| = k$. Let $\alpha$ be a geodesic starting at $g$ and contained in $gT^c(g)$. From the definition of the set $T^c(g)$ it follows that for any $n > 0$ we have $|\alpha(n)| = n + k$. This shows that for any geodesic $\beta$ connecting $e$ with $g$ the curve $\beta \cup \alpha$ is geodesic, because for any $m > k$ its restriction to $[0, m]$ connects the points $e$ and $\alpha(m - k)$ which have distance exactly $m$ from each other. Therefore $\lim_{n \rightarrow \infty} \alpha(n) = \lim_{m \rightarrow \infty} \beta(m)$ belongs to $\sppan(g)$.

The opposite inclusion is obvious.
\end{proof}

Let us fix a constant $a > 1$ (depending on the group $G$) used in the definition of the visual metric on $\partial G$.

\begin{fakt}
\label{fakt-spany-male}
Let $g \in G$. If $|g| = n$, then $\diam \sppan(g) \leq C \cdot a^{-n}$, where $C$ is a constant depending only on $G$.
\end{fakt}

\begin{proof}
Let $x, y \in \sppan(g)$ and $\alpha, \beta$ be geodesics following from $e$ through $g$ correspondingly to $x$ and $y$. By Lemma 12.3.1 in~\cite{CDP}, the path $\overline{\beta}$ built by joining the restrictions $\alpha \big|_{[0, n]}$ and $\beta \big|_{[n, \infty)}$ is a geodesic converging to~$y$. On the other hand, $\overline{\beta}$ coincides with~$\alpha$ on the interval~$[0, n]$. Then, if $\gamma$ is a bi-infinite geodesic connecting~$x$ with~$y$, from Lemma~5.2.1 in~\cite{zolta} we obtain $d(e, \gamma) \geq n - 12\delta$, which finishes the proof.
\end{proof}

\subsection{Ball $N$-types}

\label{sec-typy-kulowe}

\begin{ozn}
For any $x \in G$ and $r > 0$, we denote by $B_r(x)$ the set $\{ y \in G \,|\, d(x, y) \leq r \}$. 
\end{ozn}

\begin{df}[{\cite[Chapter~12]{CDP}}]
\label{def-typ-kulowy}
Let $x \in G$ and $N > 0$. We define the \textit{ball $N$-type} of an element $x$ (denoted $T^b_N(x)$) as the function $f^b_{x,\,N} : B_N(e) \rightarrow \mathbb{Z}$, given by formula
\begin{align} 
\label{eq-def-n-typu}
 f^b_{x,\,N}(y) = |xy| - |x|. 
\end{align}
\end{df}

\begin{lem}[{\cite[Lemma~12.3.3]{CDP}}]
\label{lem-kulowy-wyznacza-stozkowy}
There exists a constant $N_0$, depending only on $G$, such that for any $N \geq N_0$ and $x, y \in G$, the equality $T^b_N(x) = T^b_N(y)$ implies that $T^c(x) = T^c(y)$.
\end{lem}

\begin{fakt}
\label{fakt-kulowy-duzy-wyznacza-maly}
Let $x, y \in G$, $N, k > 0$ and $|y| \leq k$. Then, $T^b_N(xy)$ depends only on $T^b_{N + k}(x)$, $y$ and~$N$. \qed
\end{fakt}

\begin{proof}
Let $f$, $f'$ denote the functions of $(N+k)$-type for $x$ and $N$-type for $xy$, respectively. Let $z \in B_N(e)$. 
Then, $yz$ and $y$ both belong to $B_{N + k}(e)$, which is the domain of~$f$, and moreover
\[ f'(z) = |xyz| - |xy| = |xyz| - |x| - (|xy| - |x|) = f(yz) - f(y). \qedhere \]
\end{proof}

\begin{lem}
\label{lem-potomkowie-dla-kulowych}
Let $N_0$ be the constant from Proposition~\ref{lem-kulowy-wyznacza-stozkowy}. Let $N > N_0 + 8\delta$, $M \geq 0$, $x \in G$ and $y \in T^c(x)$, where $|y| \geq M + 4\delta$. Then, $T^b_M(xy)$ depends only on $T^b_N(x)$, $y$ and $N$, $M$.
\end{lem}

Note that the value $M \geq 0$ in this proposition can be chosen arbitrarily.

\begin{proof}
Let $x, x' \in G$ be such that $T^b_N(x) = T^b_N(x')$. Denote $n = |x|$.

Let $z \in B_M(e)$. We need to prove that
\begin{align} 
\label{eq-potomkowie-do-spr}
|xyz| - |xy| = |x'yz| - |x'y|. 
\end{align}
Let $\alpha, \beta$ be geodesics connecting $e$ respectively with $xy$ and $xyz$; we can assume that $\alpha$ passes through $x$. Denote $w = x^{-1}\beta(n)$. Since $n \leq |xy| - \tfrac{2M}{2} - 4\delta$, by applying Corollary \ref{wn-krzywe-geodezyjne-pozostaja-bliskie} for geodesics $\alpha$, $\beta$, we obtain
\[ |w| = d(x, xw) = d(\alpha(n), \beta(n)) \leq 8\delta. \]
Then, by the equality $T^b_N(x) = T^b_N(x')$, we deduce from Lemma~\ref{fakt-kulowy-duzy-wyznacza-maly} that $T^b_{N - 8\delta}(xw) = T^b_{N - 8\delta}(x'w)$. 
By Proposition \ref{lem-kulowy-wyznacza-stozkowy}, we obtain
\[ T^c(x) = T^c(x'), \qquad T^c(xw) = T^c(x'w), \]
where $y$ belongs to the first set and $w^{-1}yz$ to the second one. This gives \eqref{eq-potomkowie-do-spr} because
\[ |xyz| - |xy| = |xw| + |w^{-1}yz| - (|x| + |y|) = |w^{-1}yz| - |y| = |x'w| + |w^{-1}yz| - (|x'| + |y|) = |x'yz| - |x'y|. \qedhere \]
\end{proof}

\begin{fakt}
\label{fakt-kuzyni-lub-torsje}
For any $r > 0$ there is $N_r > 0$ such that for any $N \geq N_r$ and $g, h \in G$, the conditions
\[ |h| \leq r, \qquad |gh| = |g|, \qquad T^b_N(gh) = T^b_N(g) \]
imply that $h$ is a torsion element.
\end{fakt}

\begin{proof}
If $h$ is not a torsion element, then the remark following Proposition~1.7.3 in~\cite{zolta} states that it must be of hyperbolic type (which means that the sequence $(h^n)_{n \in \mathbb{Z}}$ is a bi-infinite quasi-geodesic in $G$). In this situation, a contradiction follows from the proof of Proposition~7.3.1 in~\cite{zolta}, provided that we replace in this proof the constant $4\delta$ by $r$. (This change may increase the value of~$N_r$ obtained from the proof but the argument does not require any other modification).
\end{proof}

\section{Quasi-invariant systems}

\label{sec-konstr}

The presentation of $\partial G$ as a Markov compactum will be obtained in the following steps:
\begin{itemize}[nolistsep]
 \item[(i)] choose a suitable system $\mathcal{U}$ of open covers of~$\partial G$;
 \item[(ii)] build an inverse system of nerves of these covers (and appropriate maps between them);
 \item[(iii)] prove that $\partial G$ is the inverse limit of this system;
 \item[(iv)] verify the Markov property (see~Definition~\ref{def-kompakt-markowa}).
\end{itemize}
The steps (ii-iii) and (iv) will be discussed in Sections~\ref{sec-engelking} and \ref{sec-markow} respectively. In this section, we focus on step~(i). We begin with introducing in Section~\ref{sec-konstr-quasi-niezm} the notion of a \textit{quasi-$G$-invariant system of covers} in $\partial G$ (or, more generally, in a compact metric $G$-space), which summarises the conditions under which we will be able to execute steps (ii-iv). Section~\ref{sec-konstr-gwiazda} contains proof an additional \textit{star property} for such systems; we will need it in Section~\ref{sec-engelking}. Finally, in Section~\ref{sec-konstr-pokrycia} we construct an example quasi-$G$-invariant system in~$\partial G$, which will serve as the basis for the construction of the Markov system representing $\partial G$.

\subsection{Definitions}
\label{sec-konstr-quasi-niezm}

Let $(X, d)$ be a metric space equipped with a homeomorphic action of a hyperbolic group $G$ (recall that we assume that $G$ is equipped with a fixed set of generators).

Definitions \ref{def-quasi-niezm} and~\ref{def-quasi-niezm-pokrycia} summarise conditions which --- as we will prove in Sections \ref{sec-konstr-gwiazda} and~\ref{sec-markow} --- are sufficient to make the construction of Section \ref{sec-engelking-top} (and in particular Theorem~\ref{tw-konstr}) applicable to the sequence $(\mathcal{U}_n)_{n \geq 0}$, and to guarantee that the constructed inverse system has Markov property (in the sense of Definition~\ref{def-kompakt-markowa}). In the next subsection, we will construct, for a given hyperbolic group $G$, a~sequence of covers of $\partial G$ with all properties introduced in this subsection.

\begin{ozn}
\label{ozn-suma-rodz}
For any family $\mathcal{C} = \{ C_x \}_{x \in G}$ of subsets of a~space $X$, we denote:
\[ \mathcal{C}_n = \big\{ C_x \ \big|\ x \in G, \  |x| = n \big\}, \qquad |\mathcal{C}|_n = \bigcup_{C \in \mathcal{C}_n} C. \]
We will usually identify the family $\mathcal{C}$ with the sequence of subfamilies $(\mathcal{C}_n)_{n \geq 0}$.
\end{ozn}

\dzm{
\begin{df}
\label{def-funkcja-typu}
By a \textit{type function} on~$G$ 
we will mean any function~$T$ on~$G$ 
with values in a~finite set. For $x \in G$, the value $T(x)$ will be called the \textit{($T$-)type} of~$x$. 

Analogously, by a \textit{type function} on a system $(K_n)_{n \geq 0}$ of simplicial complexes we will mean any function $T$ mapping simplexes of all $K_n$ to a finite set; the value $T(\sigma)$ will be called the \textit{($T$-)type} of $\sigma$.

For two type functions $T_1, T_2$ on $G$ (resp. on a system $(K_n)_{n \geq 0}$), we will call $T_1$ \textit{stronger} than $T_2$ if the $T_2$-type of any element (resp. simplex)
can be determined out of its $T_1$-type.
\end{df}
}

\begin{df}
\label{def-quasi-niezm}
A family $\mathcal{C} = \{ C_x \}_{x \in G}$ of subsets of a $G$-space $X$ is a~\textit{quasi-$G$-invariant system} \dzm{(with respect to a type function $T : G \rightarrow \mathcal{T}$)} if there exists a \textit{neighbourhood constant} $D > 0$ and a~\textit{jump constant} $J > 0$ such that:
\begin{itemize}[leftmargin=1.5cm]
 \qhitem{c}{QI1}{QI1} the sequence of subfamilies $(\mathcal{C}_n)_{n \geq 0}$, where $\mathcal{C}_n = \big\{ C_x \ \big|\ x \in G, \  |x| = n \big\}$, has mesh property (in the sense of Definition \ref{def-mesh}a);
 \qhitem{d}{QI2}{QI2} for every $n$ and $x, y \in G$, the following implication holds:
 \[ |x| = |y| = n, \quad C_x \cap C_y \neq \emptyset \qquad \Rightarrow \qquad d(x, y) \leq D; \]
 \qhitem{e}{QI3}{QI3} for every $x \in G$ and $0 < k \leq \tfrac{|x|}{J}$, there exists $y \in G$ such that $|y| = |x| - kJ$ and $C_y \supseteq C_x$;
 \qhitem{f}{QI4}{QI4} whenever $T(x) = T(gx)$ for $g, x \in G$, we have:
 \begin{itemize}
  \qhitem{f1}{a}{QI4a} $C_{gx} = g \cdot C_x$;
  \qhitem{f2}{b}{QI4b} for every $y \in G$ such that $|y| = |x|$ and $C_x \cap C_y \neq \emptyset$, we have
   \[ C_{gy} = g \cdot C_y, \qquad |gy| = |gx|; \]
  \qhitem{f3}{c}{QI4c} 
  for every $y \in G$ such that $|y| = |x| + kJ$ for some $k > 0$ and $\emptyset \neq C_y \subseteq C_x$, we have
   \[   |gy| = |gx| + kJ, \qquad T(gy) = T(y), \qquad \textrm{and \ so} \qquad C_{gy} = g \cdot C_y. \]
 \end{itemize}
\end{itemize}
\end{df}

\begin{uwaga}
\label{uwaga-quasi-niezm-jeden-skok}
Let us note that if \qhlink{e} is satisfied for $k = 1$, then by induction it must hold for all $k > 0$, and that the same applies to \qhlink{f3}.
\end{uwaga}

\begin{uwaga}
\label{uwaga-quasi-niezm-rozne-poziomy}
From now on, we adopt the convention that the sets belonging to $\mathcal{C}_n$ are implicitly equipped with the value of~$n$; this would matter only if some subsets $C_1 \in \mathcal{C}_{n_1}$, $C_2 \in \mathcal{C}_{n_2}$ with $n_1 \neq n_2$ happen to consist of the same elements. In this case, we will treat $C_1$, $C_2$ as \textit{not} equal; in particular, any condition of the form $C_1 = C_g$ will implicitly imply $|g| = n_1$. This should not lead to confusion since, although we will often consider an inclusion between an element of $\mathcal{C}_{n_1}$ and an element of $\mathcal{C}_{n_2}$ with $n_1 \neq n_2$, we will be never interested whether set-equality holds between these objects.
\end{uwaga}

\begin{df}
\label{def-quasi-niezm-pokrycia}
A system $\mathcal{C} = \{ C_x \}_{x \in G}$ of subsets of $X$ will be called \textit{a system of covers} if $\mathcal{C}_n$ is an open cover of $X$ for every $n \geq 0$.
\end{df}

\dzm{
\begin{df}
\label{def-quasi-niezm-system-wpisany}
Let $\mathcal{C} = \{ C_x \}_{x \in G}$, $\mathcal{D} = \{ D_x \}_{x \in G}$ be two quasi-$G$-invariant systems of subsets of~$X$. We will say that $\mathcal{C}$ is \textit{inscribed} in~$\mathcal{D}$ if $C_x \subseteq D_x$ for every $x \in G$, and if the type function associated to~$\mathcal{C}$ is stronger than the one associated to~$\mathcal{D}$.
\end{df}
}

\subsection{The star property}

\label{sec-konstr-gwiazda}

\begin{df}
Let $\mathcal{U}$ be an open cover of~$X$, and $U \in \mathcal{U}$. Then, the \textit{star} of $U$ in~$\mathcal{U}$ is the union $\bigcup\{U_i \, | \, U_i \in \mathcal{U}, \ U_i \cap U \neq \emptyset\}$.
\end{df}

\begin{df}
\label{def-wl-gwiazdy}
Let $(\mathcal{U}_n)$ be a family of open covers of~$X$. We say that $(\mathcal{U}_n)$ has \textit{star property} if, for every $n > 0$, every star in the cover $\mathcal{U}_n$ is contained in some element of the cover $\mathcal{U}_{n-1}$; more formally:
 \[ \forall_{n > 0} \ \forall_{U \in \mathcal{U}_n} \ \exists_{V \in \mathcal{U}_{n-1}} \ \bigcup_{U' \in \mathcal{U}_n; \, U \cap U' \neq \emptyset} U' \subseteq V. \]
\end{df}

\begin{lem}
\label{lem-gwiazda}
Let $(\mathcal{U}_n)$ be a quasi-$G$-invariant system of covers of a compact metric $G$-space $X$ and let $J$ denote its jump constant. Then, there exists a constant $L_0$ such that, for any $L \geq L_0$ divisible by~$J$, the sequence of covers $(\mathcal{U}_{Ln})_{n \in \mathbb{N}}$ has star property.
\end{lem}

\begin{proof}
Let $L_{(i)}$ be constant such that, for every $j \geq i + L_{(i)}$, every element of~$\mathcal{U}_j$ together with its star is contained in some set from $\mathcal{U}_i$. Its existence is an immediate result of the existence of a Lebesgue number for $\mathcal{U}_i$, and from the mesh condition for the system~$(\mathcal{U}_n)$.

Since there exist only finite many $N$-types in~$G$, there exists $S > 0$ such that for any $g \in G$ there is $g' \in G$ such that $|g'| < S$ and $T(g) = T(g')$.
We will show the claim of the proposition is satisfied by
\[ L_0 = 1 + \max \{ L_{(i)} \,|\, i < S \}. \]

Let $|g| = L(k + 1)$ and $L \geq L_0$ be divisible by $J$; we want to prove that there exists $\tilde{f} \in G$ of length $Lk$ such that $U_{\tilde{f}}$ contains $U_g$ together with all its neighbours in $\mathcal{U}_{L(k+1)}$. If $Lk < S$, this holds by the inequality $L \geq L_0 \geq L_{(Lk)}$ and the definition of the constant $L_{(Lk)}$. 

Otherwise, by the property \qhlink{e} there exists $f$ of length $Lk$ such that $U_g \subseteq U_{f}$. Let $f' \in G$ of length $j < S$ satisfy $T(f') = T(f)$.
Denote $h = f' f^{-1}$. Then, since $J \divs L$, by~\qhlink{f3} we have 
\begin{align}
\label{eq-gwiazda-sukces-na-malej}
 U_{hg} = h \cdot U_g \subseteq h \cdot U_{f} = U_{f'}, \qquad T(hg) = T(g), \qquad |hg| = j + L.
\end{align}
Therefore, since $j < S$, there exists some $\tilde{f}'$ of length $j$ such that $U_{\tilde{f}'}$ contains $U_{hg}$ together with its whole star. Then, by \eqref{eq-gwiazda-sukces-na-malej}, we have $U_{\tilde{f}'} \cap U_{f'} \neq \emptyset$, and so from \qhlink{f2} we obtain
\[ \qquad U_{h^{-1}\tilde{f}'} = h^{-1} \cdot U_{\tilde{f}'}, \qquad |h^{-1} \widetilde{f}'| = |h^{-1} f'| = Lk. \]

Now, let $|x| = |g|$ and $U_x \cap U_g \neq \emptyset$. Then, from \eqref{eq-gwiazda-sukces-na-malej} and~\qhlink{f2} we have $U_{hx} = h \cdot U_x$; in particular, $U_{hx}$ is contained in the star of the set $U_{hg} = h \cdot U_g$, and so it is contained in $U_{\tilde{f}'}$. Then, by \qhlink{f3}:
\[ U_x = h^{-1} \cdot U_{hx} \subseteq h^{-1} \cdot U_{\tilde{f}'} = U_{h^{-1} \tilde{f}'}. \]
This means that the element $\tilde{f} := h^{-1}\tilde{f}'$ has the desired property.
\end{proof}

\subsection{The system of span-star interiors}

\label{sec-konstr-pokrycia}

\begin{df}
\label{def-konstr-towarzysze}
For every element $g \in G$ and $r > 0$, we denote
\[ P(x) = \big\{ y \in G \,\big|\, |xy| = |x| \big\}, \qquad P_r(x) = P(x) \cap B_r(e). \]
If $y \in P(x)$ (resp. $P_r(x)$), we call $xy$ a \textit{fellow} (resp. \textit{$r$-fellow}) of $x$.
\end{df}

From the definition of the ball type, we obtain the following property.

\begin{fakt}
\label{fakt-kulowy-wyznacza-towarzyszy}
If $N \geq r > 0$, then the set $P_r(x)$ depends only on $T^b_N(x)$ and $r$, $N$. \qed
\end{fakt}

\begin{df}
\label{def-span-star}
We define the set $S_g$ as the interior of span-star in $\partial G$ around $\sppan(g)$:
\[ S_g = \innt \Big( \bigcup_{h \in I(g)} \sppan(gh) \Big), \qquad \textrm{ gdzie } \quad I(g) = \big\{ h \in P(g) \,\big|\, \sppan(gh) \cap \sppan(g) \neq \emptyset \big\}. \] 

For any $k > 0$, we define the family
\[ \mathcal{S}_k = \{ S_g \ |\  g \in G, \, |g| = k, \, S_g \neq \emptyset \}. \]
\end{df}

\begin{fakt}
\label{fakt-span-w-pokryciu}
For every $g \in G$, we have $\sppan(g) \subseteq S_g$.
\end{fakt}

\begin{proof}
Let us consider the equality
\[ \partial G = \bigcup_{h \in P(g)} \sppan(gh) = \Big( \bigcup_{h \in I(g)} \sppan(gh) \Big) \cup \Big( \bigcup_{h \in P(g) \setminus I(g)} \sppan(gh) \Big). \]
The second summand is disjoint with $\sppan(g)$, and moreover closed (as a finite union of closed sets), which means that $\sppan(g)$ must be contained in the interior of the first summand, which is exactly $S_g$.
\end{proof}

\begin{wn}
\label{wn-konstr-pokrycie}
For every $k > 0$, the family $\mathcal{S}_k$ is a cover of $\partial G$.
\end{wn}

\begin{proof}
This is an easy application of the above lemma and of the equality $\partial G = \bigcup_{g \in G \,:\, |g| = k} \sppan(g)$.
\end{proof}

\begin{fakt}
\label{fakt-pokrycie-male}
Under the notation of Lemma \ref{fakt-spany-male}, for every $k > 0$ and $U \in \mathcal{S}_k$, we have $\diam U \leq 3 C \cdot a^{-k}$.
\end{fakt}

\begin{proof}
Let $U = S_g$ for some $g \in G$, where $|g|=k$, and let $x, y \in S_g$. Then, $x \in \sppan(gh_1)$ and $y \in \sppan(gh_2)$ for some $h_1, h_2 \in I(g)$. By Lemma \ref{fakt-spany-male}, we obtain
\[ d(x, y) \leq \diam \sppan(gh_1) + \diam \sppan(g) + \diam \sppan(gh_2) \leq 3 C \cdot a^{-k}. \qedhere \]
\end{proof}

\begin{fakt}
\label{fakt-sasiedzi-blisko}
Let $h \in P(g)$. Then:
\begin{itemize}
 \item[\textbf{(a)}] If $\sppan(g) \cap \sppan(gh) \neq \emptyset$, then $|h| \leq 4\delta$ (so: $I(g) \subseteq P_{4\delta}(g)$);
 \item[\textbf{(b)}] If $S_g \cap S_{gh} \neq \emptyset$, then $|h| \leq 12\delta$.
\end{itemize}
\end{fakt}

\begin{proof}
\textbf{(a)} Let $|g| = |gh| = k$ and $x \in \sppan(g) \cap \sppan(gh)$. Then, there exist geodesics $\alpha, \beta$ stating at $e$ and converging to $x$ such that $\alpha(k) = g$, $\beta(k) = gh$.

By inequality (1.3.4.1) in~\cite{zolta}, this implies that $d(g, gh) \leq 4\delta$.

\textbf{(b)} Let $x \in S_g \cap S_{gh}$. Then, by definition, we have $x \in \sppan(gu) \cap \sppan(ghv)$ for some $u \in I(g)$, $v \in I(gh)$. Using part~\textbf{a)}, we obtain 
\[ |h| \leq |u| + |u^{-1}hv| + |v^{-1}| \leq 4\delta + 4\delta + 4\delta = 12\delta. \qedhere \]
\end{proof}

\begin{fakt}
\label{fakt-wlasnosc-gwiazdy-bez-gwiazdy}
Let $g \in G$ and $k < |g|$. Then:
\begin{itemize}
\item[\textbf{(a)}] there exists $f \in G$ of length $k$ such that $g \in fT^c(f)$;
\item[\textbf{(b)}] for any $f \in G$ with the properties from part~\textbf{(a)}, we have $\sppan(g) \subseteq \sppan(f)$;
\item[\textbf{(c)}] for any $f \in G$ with the properties from part~\textbf{(a)}, we have $S_g \subseteq S_f$.
\end{itemize}
\end{fakt}

\begin{proof}
\textbf{(a)} Let $\alpha$ be a geodesic from $e$ to $g$. Then, $f = \alpha(k)$ has the desired properties.

\textbf{(b)} If $f$ has the properties from part~\textbf{(a)}, then, by Lemma \ref{fakt-przechodniosc-potomkow}, we have $gT^c(g) \subseteq fT^c(f)$, so it remains to apply Lemma \ref{fakt-stozek-a-span}. 

\textbf{(c)} By the parts~\textbf{(a)} and~\textbf{(b)}, for any $h \in I(g)$ there exists some element $f_h$ of length $k$ such that $\sppan(gh) \subseteq \sppan(f_h)$; here $f_e$ can be chosen to be $f$. In particular, we have:
\[ \emptyset \neq \sppan(g) \cap \sppan(gh) \subseteq \sppan(f) \cap \sppan(f_h), \]
so $f^{-1} f_h \in I(f)$. Since $h \in I(g)$ is arbitrary, we obtain
\[ \bigcup_{h \in I(g)} \sppan(gh) \subseteq \bigcup_{h \in I(g)} \sppan(f_h) \subseteq \bigcup_{x \in I(f)} \sppan(f x). \]
By taking the interiors of both sides of this containment, we get the claim.
\end{proof}

\begin{lem}
\label{lem-typy-pokrycia-niezmiennicze}
Let $N_0$ denote the constant from Proposition~\ref{lem-kulowy-wyznacza-stozkowy}. Assume that $N, r \geq 0$ and $g, x \in G$ satisfy $T^b_N(gx) = T^b_N(x)$. Then:
\begin{itemize}
 \item[\textbf{(a)}] if $N \geq N_0$, then $\sppan(gx) = g \cdot \sppan(x)$;
 \item[\textbf{(b)}] if $N \geq N_0 + r$, then $\sppan(gxy) = g \cdot \sppan(xy)$ for $y \in P_r(x)$;
 \item[\textbf{(c)}] if $N \geq N_1 := N_0 + 4\delta$, then $S_{gx} = g \cdot S_x$;
 \item[\textbf{(d)}] if $N \geq N_1 + r$, then $S_{gxy} = g \cdot S_{xy}$ for $y \in P_r(x)$;
 \item[\textbf{(e)}] if $N \geq N_2 := N_0 + 16\delta$ and $y \in G$ satisfy $|y| = |x|$ and $S_x \cap S_y \neq \emptyset$, then  
 \[ S_{gy} = g \cdot S_y, \qquad \textrm{ and moreover } \quad |gy| = |gx|; \]
\item[\textbf{(f)}] if $N \geq N_3 := N_0 + 21\delta$, $k \geq 0$, $L > N + k + 4\delta$ and $y \in G$ satisfy $|y| = |x| + L$ and $\emptyset \neq S_y \subseteq S_x$, then:
 \[ S_{gy} = g \cdot S_y, \qquad \textrm{ and moreover } \quad |gy| = |gx| + L \quad \textrm{and} \quad T^b_{N + k}(gy) = T^b_{N + k}(y). \]
 \end{itemize}
\end{lem}

\begin{proof}
\textbf{(a)} If $N \geq N_0$, by Proposition \ref{lem-kulowy-wyznacza-stozkowy} we have $T^c(gx) = T^c(x)$ and so $gxT^c(gx) = g \cdot xT^c(x)$. In particular, the left action by $g$, which is an isometry, gives a unique correspondence between geodesics in $G$ starting at $x$ and contained in $xT^c(x)$ and geodesics in $G$ starting at $gx$ and contained in~$gxT^c(gx)$. Then, the claim holds by Lemma~\ref{fakt-stozek-a-span} and by continuity of the action of $g$ on~$G \cup \partial G$.

\textbf{(b)} If $N \geq N_0 + r$, then, by Lemma~\ref{fakt-kulowy-duzy-wyznacza-maly}, we have $T^b_{N_0}(gxy) = T^b_{N_0}(xy)$ for every $y \in P_r(x)$; it remains to apply~\textbf{(a)}.

\textbf{(c)} Let $y \in I(x)$. By Lemma~\ref{fakt-sasiedzi-blisko}a, we have $y \in P_{4\delta}(x)$. Since $N \geq N_0 + 4\delta$, from~\textbf{(b)} and Lemma~\ref{fakt-kulowy-wyznacza-towarzyszy} we obtain that
\[ \sppan(gx) = g \cdot \sppan(x), \qquad \sppan(gxy) = g \cdot \sppan(xy), \qquad y \in P_{4\delta}(gx). \]
Since $\sppan(x) \cap \sppan(xy) \neq \emptyset$, by acting with~$g$ we obtain $\sppan(gx) \cap \sppan(gxy) \neq \emptyset$, so $y \in I(gx)$. Then, we have
\[ g \cdot \bigcup_{y \in I(x)} \sppan(xy) = \bigcup_{y \in I(x)} \sppan(gxy) \subseteq \bigcup_{y \in I(gx)} \sppan(gxy). \]
By an analogous reasoning for the inverse element $g^{-1}$, we prove that the above containment is in fact an equality. Moreover, since the left action of~$g$ is a homeomorphism, it must map the interior of the left-hand side sum (which is~$S_x$) exactly onto the interior of the right-hand side sum (resp.~$S_{gx}$).

\textbf{(d)} This follows from~\textbf{(c)} in the same way as \textbf{(b)} was obtained from~\textbf{(a)}.

\textbf{(e)} By Lemma~\ref{fakt-sasiedzi-blisko}b, we have $x^{-1} y \in P_{12\delta}(x)$. Then, the first part of the claim follows from~\textbf{(d)}. For the second part, note that from $N \geq 16\delta$ we obtain that $x^{-1} y \in P_N(x)$ which is contained in the domain of $T^b_N(x)$ (as a~function); hence, the assumption that $T^b_N(x) = T^b_N(gx)$ implies that $|gxy| - |gx| = |xy| - |x| = 0$, as desired.

\textbf{(f)} Let $|y| \geq |x|$ and $S_y \subseteq S_x$. By Lemma~\ref{fakt-wlasnosc-gwiazdy-bez-gwiazdy}, there exists $z \in G$ such that
\[ |xz| = |x|, \qquad y \in xzT^c(xz), \qquad S_y \subseteq S_{xz}. \]
In particular, $S_{xz} \cap S_x \neq \emptyset$, and so by Lemma~\ref{fakt-sasiedzi-blisko} we have $z \in P_{12\delta}(x)$. By Lemmas~\ref{fakt-kulowy-duzy-wyznacza-maly} and~\ref{fakt-kulowy-wyznacza-towarzyszy}, we obtain
\begin{align} 
\label{eq-niezm-z-tow-gx}
T^b_{N - 12\delta}(gxz) = T^b_{N - 12\delta}(xz), \qquad z \in P_{12\delta}(gx).
\end{align}
From the first of these properties and from Proposition \ref{lem-kulowy-wyznacza-stozkowy}, we have $T^c(gxz) = T^c(xz)$. Since $(xz)^{-1}y$ belongs to $T^c(xz)$ and is of length
\begin{align} 
\label{eq-niezm-dlugosci}
|(xz)^{-1}y| = |y| - |xz| = |y| - |x| = L > N + k + 4\delta, 
\end{align}
by applying Proposition~\ref{lem-potomkowie-dla-kulowych} for \eqref{eq-niezm-z-tow-gx} and the action of $(xz)^{-1}y$ (with parameters $N - 12\delta > N_0 + 8\delta$, $N + k$) we obtain
\[ T^b_{N + k}(gy) = T^b_{N + k}(y), \]
and then from \textbf{(c)} 
\[ S_{gy} = g \cdot S_y. \]
Moreover, the conditions $(xz)^{-1}y \in T^c(gxz)$, \eqref{eq-niezm-dlugosci} and \eqref{eq-niezm-z-tow-gx} imply that
\[ |gy| = |gxz| + |(xz)^{-1}y| = |gxz| + L = |gx| + L, \]
which finishes the proof.
\end{proof}

\begin{wn}
 \label{wn-spanstary-quasi-niezm}
 For $N \geq N_3$, the sequence of covers $(\mathcal{S}_n)$ together with the type function $T = T^b_N$ is a quasi-$G$-invariant system of covers.
\end{wn}

\begin{proof}
 We have checked in Corollary \ref{wn-konstr-pokrycie} that every $\mathcal{S}_n$ is a~cover of $\partial G$; obviously it is open. The subsequent conditions from Definition \ref{def-quasi-niezm} hold correspondingly by \ref{fakt-pokrycie-male}, \ref{fakt-sasiedzi-blisko}b and \ref{fakt-wlasnosc-gwiazdy-bez-gwiazdy} and Proposition~\ref{lem-typy-pokrycia-niezmiennicze}c,e,f (for $k = 0$). Here, we take the following constants:
 \[ D = 12\delta, \qquad J_0 = 0, \qquad J = N_3 + 4\delta. \qedhere \]
\end{proof}

\section{Inverse limit construction}

\label{sec-engelking}

\dzm{
In this section, we present a classical construction (see Theorem~\ref{tw-konstr} below) which presents --- up to a~homeomorphism --- every compact metric space~$X$ as the inverse limit of the sequence of nerves of an appropriate system of covers of~$X$ (which we will call \textit{admissible}; see Definition~\ref{def-konstr-admissible}). We will also show (in Lemma~\ref{fakt-konstr-sp-zal}) that admissible systems can be easily obtained from any quasi-$G$-invariant systems of covers~$\mathcal{U}$.

In Section~\ref{sec-bi-lip}, we investigate this construction in the particular case when~$X = \partial G$ and $\mathcal{U}$ is inscribed in the system~$\mathcal{S}$ from Section~\ref{sec-konstr-pokrycia}. As we will show in Theorem~\ref{tw-bi-lip}, in such case the construction allows as well to describe certain metric properties of $\partial G$. (See the introduction to Section~\ref{sec-bi-lip} for more details).
}

\dzm{
\subsection{A topological description by limit of nerves}

\label{sec-engelking-top}
}

Let $X$ be a~compact metric space.

\begin{df}
 Recall that the \textit{rank} of a family $\mathcal{U}$ of subspaces of a space $X$ is the maximal number of elements of $\mathcal{U}$ which have non-empty intersection.
\end{df}

\begin{df}
\label{def-konstr-admissible}
A sequence $(\mathcal{U}_i)_{i \geq 0}$ of open covers $X$ will be called an \textit{admissible system} if the following holds:
\begin{itemize}
 \item[(i)] for every $i \geq 0$, the cover $\mathcal{U}_i$ is finite and does not contain empty sets;
 \item[(ii)] there exists $n \geq 0$ such that $\rank \mathcal{U}_i \leq n$ for every $i \geq 0$;
 \item[(iii)] the sequence $(\mathcal{U}_i)_{i \geq 0}$ has mesh property (in the sense of Definition \ref{def-mesh}a);
 \item[(iv)] the sequence $(\mathcal{U}_i)_{i \geq 0}$ has star property (see Definition \ref{def-wl-gwiazdy}).
\end{itemize}
\end{df}

There is an easy connection between this notion and the contents of the previous section:

\begin{fakt}
\label{fakt-konstr-sp-zal}
Let $(\mathcal{U}_n)$ be a quasi-$G$-invariant system of covers of a $G$-space $X$. Define
\[ \widetilde{\mathcal{U}}_n = \{ U \in \mathcal{U}_n \,|\, U \neq \emptyset \}. \]
Let $L_0$ denote the constant obtained for the system $(\mathcal{U}_n)$ from Proposition~\ref{lem-gwiazda}. Then, for any $L \geq L_0$, the sequence of the covers $(\widetilde{\mathcal{U}}_{nL})_{n \geq 0}$ is admissible.
\end{fakt}

\begin{proof}
Clearly, for every $n \geq 0$ the family $\widetilde{\mathcal{U}}_n$ is an open cover of $X$. The condition (i) follows from the definition of~$\widetilde{\mathcal{U}}_n$. The mesh and star properties result correspondingly from the property \qhlink{c} and Proposition~\ref{lem-gwiazda}.

Finally, the condition (ii) follows from \qhlink{d}: whenever $U_x \cap U_y \neq \emptyset$, we have $d(x, y) \leq D$, so $x^{-1}y$ belongs to the ball in $G$ centred in $e$ of radius $D$. This means that the rank of the cover $\mathcal{U}_n$ (and thus also of $\widetilde{\mathcal{U}}_n$) does not exceed the number of elements in this ball, which is finite and independent from $n$.
\end{proof}

\begin{ozn}
Let $\mathcal{U}$ be an open cover of~$X$. For $U \in \mathcal{U}$, we denote by~$v_U$ the vertex in the nerve of~$\mathcal{U}$ corresponding to~$U$. We also denote by~$[v_1, \ldots, v_n]$ the simplex in~this nerve spanned by vertices $v_1, \ldots, v_n$.
\end{ozn}

\begin{df}
\label{def-konstr-nerwy}
For an admissible system $(\mathcal{U}_i)_{i \geq 0}$ in~$X$, we define the \textit{associated system of nerves} $(K_i, f_i)_{i \geq 0}$, where $f_i : K_{i+1} \rightarrow K_i$ for $i \geq 0$, as follows:
\begin{itemize}
 \item[(i)] for $i \geq 0$, $K_i$ is the nerve of the cover~$\mathcal{U}_i$;
 \item[(ii)] for $U \in K_{i+1}$, $f_i(v_U)$ is the barycentre of the simplex spanned by $\{ v_V \,|\, V \in K_i, \, V \supseteq U \}$;
 \item[(iii)] for other elements of $K_{i+1}$, we extend $f_i$ so that it is affine on every simplex.
\end{itemize}
For any $j \geq 0$, we denote by $\pi_j$ the natural projection from the inverse limit $\liminv K_i$ to $K_j$.
\end{df}

\begin{uwaga}
If $v_{U_1}, \ldots, v_{U_n}$ span a~simplex in $K_{i+1}$, then $U_1 \cap \ldots \cap U_n \neq \emptyset$; this implies that the family $\mathcal{A} = \{ V \in \mathcal{U}_i \,|\, V \supseteq U_1 \cap \ldots \cap U_n \}$ has a non-empty intersection and therefore the vertices $\{ v_V \,|\, V \in \mathcal{A} \}$ span a simplex in~$K_i$ which contains all the images $f_i(v_{U_j})$ for $1 \leq j \leq n$. This ensures that the affine extension described in condition (iii) of Definition~\ref{def-konstr-nerwy} is indeed possible.
\end{uwaga}

The following theorem is essentially an adjustment of Theorem~1.13.2 in~\cite{E} to our needs (see the discussion below).

\begin{tw}
\label{tw-konstr}
Let $(\mathcal{U}_i)_{i \geq 0}$ be an admissible system in~$X$, and $(K_i, f_i)$ be its associated nerve system. For any $x \in X$ and $i \geq 0$, denote by $K_i(x)$ the simplex in $K_i$ spanned by the set $\{ v_U \,|\, U \in \mathcal{U}_i \, x \in U \}$. Then:
\begin{itemize}
 \item[\textbf{(a)}] The system $(K_i, f_i)$ has mesh property;
 \item[\textbf{(b)}] For every $x \in X$, the space $\liminv K_i(x) \subseteq \liminv K_i$ has a~unique element, which we will denote by~$\varphi(x)$;
 \item[\textbf{(c)}] The map $\varphi : X \rightarrow \liminv K_i$ defined above is a~homeomorphism.
\end{itemize}

\end{tw}

A proof of Theorem~\ref{tw-konstr} can be obtained from the proof of Theorem~1.13.2 given in~\cite{E} as follows:

\begin{itemize}

 \item Although our assumptions are different than those in~\cite{E}, they still imply all statements in the proof given there, except for the condition labelled as~(2). However, this condition is used there only to ensure the mesh and star properties of~$(\mathcal{U}_i)$ which we have assumed anyway.

 \item The theorem from~\cite{E} does not state the mesh property for the nerve system.
 However, an inductive application of the inequality labelled as~(6) in its proof gives (in our notation) that:
 \begin{align} 
 \label{eq-konstr-szacowanie-obrazow}
 \diam f^j_i(\sigma) \leq \left( \tfrac{n}{n + 1} \right)^{j - i} \qquad \textrm{ for every simplex $\sigma$ in~$K_j$}, 
 \end{align}

 where $n$ denotes the upper bound for the rank of covers required by Definition~\ref{def-konstr-admissible}. The right-hand side of~\eqref{eq-konstr-szacowanie-obrazow} does not depend on~$\sigma$, but only on~$i$, and tends to zero as $i \rightarrow \infty$, which proves mesh property for the nerve system. (Although the above estimate holds only for the particular metric on~$K_i$ used in~\cite{E}, this suffices to deduce the mesh property in view of Remark~\ref{uwaga-mesh-bez-metryki}).
\end{itemize}

\subsection{A metric description for systems inscribed in~$\mathcal{S}$}
\label{sec-bi-lip}

\dzm{
Let~$G$ by a hyperbolic group, and let~$\mathcal{U}$ be a quasi-$G$-invariant system of covers of~$\partial G$, inscribed in the system~$\mathcal{S}$ defined in Section~\ref{sec-konstr-pokrycia}. We will now prove that, under such assumptions, the homeomorphism $\varphi : \partial G \rightarrow \liminv K_i$ obtained from Theorem~\ref{tw-konstr} on the basis of~$\mathcal{U}$ (through Lemma~\ref{fakt-konstr-sp-zal}) is a bi-Lipschitz equivalence --- when $\partial G$ is considered with the visual metric $d_v^{(a)}$ for sufficiently small value of~$a$, and $\liminv K_i$ with the natural \textit{simplicial metric} (see Definition~\ref{def-metryka-komp} below) for the same value of~$a$.
}

To put this in a context, let us recall the known properties of visual metrics on~$\partial G$. The definition of the visual metric given in~\cite{zolta} depends not only on the choice of~$a$, but also on the choice of a basepoint in the group (in this paper, we always set it to be~$e$) and a set of its generators. It is known that the visual metrics obtained for different choices of these parameters do not have to be bi-Lipschitz equivalent, however, they all determine the same quasi-conformal structure (\cite[Theorems~2.18 and~3.2]{Kap}). In this situation, Theorem~\ref{tw-bi-lip} shows that this natural quasi-conformal structure on~$\partial G$ can be as well described by means of \dzm{the inverse limit of polyhedra which we have built so far. This will enable us, in view of Theorems~\ref{tw-kompakt-ogolnie} and~\ref{tw-sk-opis} (to be shown in the next sections), to} give (indirectly) a description of quasi-conformal structures on the boundaries of hyperbolic groups in terms of appropriate Markov systems.

\subsubsection{The simplicial metric}

Let us recall the definition of the metric on simplicial complexes used in the proof of Theorem 1.13.2 in~\cite{E} (which serves as the base for Theorem \ref{tw-konstr}). For any $n \geq 0$, we denote
\[ e_i = (\underbrace{0, \ldots, 0}_{i-1}, 1, 0, \ldots, 0) \in \mathbb{R}^n. \]
\begin{df}
\label{def-metryka-l1}
Let $K$ be a simplicial complex with $n$ vertices. Let $m \geq n$ and $f : K \rightarrow \mathbb{R}^m$ be an injective affine map sending vertices of $k$ to points of the form $e_i$ (for $1 \leq i \leq m$). We define the \textit{$l^1$ metric} on $K$ by the formula:
\[ d_K(x, y) = \| f(x) - f(y) \|_1 \qquad \textrm{ for } x, y \in K. \]
\end{df}

\begin{uwaga}
\label{uwaga-metryka-l1-sens}
The metric given by Definition \ref{def-metryka-l1} does not depend on the choice of $m$ and~$f$ because any other affine inclusion $f' : K \rightarrow \mathbb{R}^{m'}$ must be (after restriction to $K$) a composition of $f$ with a linear coordinate change which is an isometry with respect to the norm $\| \cdot \|_1$.
\end{uwaga}

\begin{uwaga}
\label{uwaga-metryka-l1-ogr}
Since in Definition~\ref{def-metryka-l1} we have $f(K) \subseteq \{ (x_i) \,|\, x_i \geq 0 \textrm{ for } 1 \leq i \leq m, \ \sum_{i=1}^m x_i = 1 \}$, it can be easily deduced that any complex $K$ has diameter at most $2$ in the $l^1$ metric.
\end{uwaga}

\begin{df}
\label{def-metryka-komp}
Let $(K_i, f_i)_{i \geq 0}$ be an inverse system of simplicial complexes. For any real $a > 1$, we define \dzm{the \textit{simplicial metric} (with parameter~$a$)} $d^M_a$ on~$\liminv K_i$ by the formula
\[ d^M_a \big( (x_i)_{i \geq 0}, (y_i)_{i \geq 0} \big) = \sum_{i = 0}^\infty a^{-i} \cdot d_{K_i}(x_i, y_i). \]
\end{df}

\begin{uwaga}
In the case when $a = 2$, Definition \ref{def-metryka-komp} gives the classical metric used in countable products of metric spaces (and hence also in the limits of inverse systems); in particular, it is known that the metric~$d^M_2$ is compatible with the natural topology on the inverse limit (i.e. the restricted Tichonov's product topology). However, this fact holds, with an analogous proof, for any other value of $a>1$ (see~\cite[the remark following Theorem~4.2.2]{ET})
\end{uwaga}

\subsubsection{Bi-Lipschitz equivalence of both metrics}

In the following theorem, we use the notions \textit{quasi-$G$-invariant}, \textit{system of covers}, \textit{inscribed} defined respectively in Definitions~\ref{def-quasi-niezm}, \ref{def-quasi-niezm-pokrycia} and \ref{def-quasi-niezm-system-wpisany}, as well as the system $\mathcal{S}$ defined in Definition~\ref{def-span-star}.

\begin{tw}
\label{tw-bi-lip}
Let $G$ be a hyperbolic group. \dzm{Let~$\mathcal{U}$ be a quasi-$G$-invariant system of covers of~$\partial G$, inscribed in the system~$\mathcal{S}$ (see Section~\ref{sec-konstr-pokrycia}), and let $\varphi : \partial G \rightarrow \liminv K_i$ be the homeomorphism obtained for~$\mathcal{U}$ from Theorem~\ref{tw-konstr}.}

Then, there exists a constant $a_1 > 1$ (depending only on~$G$) such that, for any $a \in (1, a_1)$, 
$\varphi$ is a bi-Lipschitz equivalence between the visual metric on~$G$ with parameter~$a$ (see Section~\ref{sec-def-hip}) and the simplicial metric $d^M_a$ on~$\liminv K_i$.
\end{tw}

\begin{uwaga}
Theorem~\ref{tw-bi-lip} re-states the second claim of Theorem~\ref{tw-bi-lip-0}, which is sufficient to deduce the first claim in view of the introduction to Section~\ref{sec-bi-lip}.
\end{uwaga}

\begin{uwaga}
\label{uwaga-bi-lip-wystarczy-d}
To prove the above theorem, it is clearly sufficient to check a bi-Lipschitz equivalence between the simplicial metric $d_a^M$ and the \textit{distance function}~$d_a$ which has been introduced in Section~\ref{sec-def-hip} as a bi-Lipschitz approximation of the visual metric.
\end{uwaga}

\begin{fakt}
\label{fakt-rozlaczne-symp-daleko}
If $s_1,s_2$ are two disjoint simplexes in a complex $K$, then for any $z_1 \in s_1, z_2 \in s_2$, we have $d_K(z_1, z_2) = 2$.
\end{fakt}

\begin{proof}
Let $f : K \rightarrow \mathbb{R}^m$ satisfy the conditions from Definition \ref{def-metryka-l1}. For $j = 1, 2$, let~$A_j$ denote the set of indexes $1 \leq i \leq m$ for which $e_i = f(v)$ for some vertex~$v \in s_j$. Then we have
\[ f(s_j) = \Big\{ (x_i) \in \mathbb{R}^m \ \Big|\ \ x_i \geq 0 \textrm{ for } 1 \leq i \leq m, \ \ x_i = 0 \textrm{ for } i \in A_j, \ \ \sum_{i \in A_j} x_i = 1 \Big\}. \]
However, since $f$ is an inclusion, the sets $A_1$, $A_2$ are disjoint, from which it results that, for any $p_j \in f(s_j)$ (for $j = 1, 2$), we have $\| p_1 - p_2 \|_1 = 2$.
\end{proof}

\begin{lem}
\label{lem-bi-lip-geodezyjne}
There exist constants $E_1, N_4$ (depending only on~$G$) such that if $k, l \geq 0$, $N > N_4$, $g, x \in G$ and $p, q \in \partial G$ satisfy the conditions:
\[ |x| = k, \qquad |gx| = l, \qquad T^b_N(x) = T^b_N(gx), \qquad p \in \sppan(x), \qquad d(p, q) \leq a^{-(k+E_1)}, \]
then in $\partial G$ we have
\begin{align} 
\label{eq-bi-lip-geodezyjne-teza}
d(g \cdot p, \, g \cdot q) \leq a^{-(l-k)} \cdot d(p, q). 
\end{align}
\end{lem}

\begin{uwaga}
As soon as we prove the inequality \eqref{eq-bi-lip-geodezyjne-teza} in general, it will follow that it can be strengthened to an equality. This is because if the elements $g, x, p, q$ satisfy the assumptions of the proposition, then its claim implies that the elements~$g^{-1}, gx, g \cdot p, g \cdot q$ also satisfy these assumptions. By using the proposition to these elements, we will then obtain that $d(p, q) \leq a^{-(k-l)} \cdot d(g \cdot p, g \cdot q)$, ensuring that an equality in \eqref{eq-bi-lip-geodezyjne-teza} holds. We do not include this result it in the claim of the proposition because it is not used in this article.
\end{uwaga}

\begin{proof}[Proof of Proposition~\ref{lem-bi-lip-geodezyjne}]
We set
\[ E_1 = 13\delta, \qquad N_4 = N_0 + 64\delta, \]
where $N_0$ denotes the constant from Proposition \ref{lem-kulowy-wyznacza-stozkowy}.

\textbf{1. }Let $\gamma$ be some geodesic connecting $p$ with~$q$ for which the distance $d(e, \gamma)$ is maximal. Note that then
\[ d(e, \gamma) = - \log_a d(p, q) \geq k + 13\delta. \]
Since the left shift by $g$ is an isometry in $G$, the sequence $g \cdot \gamma$ determines a bi-infinite geodesic which, by definition, connects the points $\gamma \cdot p, \gamma \cdot q$ in~$\partial G$. Then, to finish the proof it suffices to estimate from below the distance $d(e, g \cdot \gamma)$.

\textbf{2. }Let $\alpha, \beta$ be some geodesics connecting $e$ correspondingly with~$p$ and~$q$; we can require in addition that $\alpha(k) = x$. Denote $y = \beta(k)$.
Since $\alpha, \beta, \gamma$ form a geodesic triangle (with two vertices in infinity), by Lemma~\ref{fakt-waskie-trojkaty}, there exists an element $s \in \beta \cup \gamma$ in a distance $\leq 12\delta$ from~$x$. 
Then, $\big| |s| - k \big| \leq 12\delta$, so in particular $s \notin \gamma$, and so $s \in \beta$. In this situation, we have
\[ d(x, y) \leq d(x, s) + d(s, y) \leq 12\delta + \big| |s| - k \big| \leq 24\delta. \]

\textbf{3. }For any $i \in \mathbb{Z}$, we choose a geodesic $\eta_i$ connecting $e$ with~$\gamma(i)$. Since $\gamma(i)$ must lie in a distance $\leq 12\delta$ from some element of~$\alpha$ or~$\beta$, by using Corollary \ref{wn-krzywe-geodezyjne-pozostaja-bliskie} for the geodesic $\eta_i$ and correspondingly $\alpha$ or~$\beta$, we obtain that the point $z_i = \eta(k)$ lies in a distance $\leq 40\delta$ correspondingly from $x$ or~$y$. Therefore, in any case we have 
\[ d(x, z_i) \leq 64 \delta. \]

\textbf{4. }We still consider any value of $i \in \mathbb{Z}$. Since $N > N_4$ and $T^b_N(x) = T^b_N(gx)$, as well as $|x| = |z_i| = k$, from Lemmas \ref{fakt-kulowy-duzy-wyznacza-maly} and~\ref{fakt-kulowy-wyznacza-towarzyszy} we obtain that
\[ T^b_{N_0}(z_i) = T^b_{N_0}(gz_i), \qquad |gx| = |gz_i| = l. \]
Then, $z_i$ and~$gz_i$ have the same cone types by Proposition \ref{lem-kulowy-wyznacza-stozkowy}, so from $\gamma(i) \in z_iT^c(z_i)$ we deduce that $g\gamma(i) \in gz_iT^c(gz_i)$, and then
\[ |g\gamma(i)| = |gz_i| + |z_i^{-1}\gamma(i)| = |gz_i| + |\gamma(i)| - |z_i| = |\gamma(i)| + (l - k). \]
By taking the minimum over all $i \in \mathbb{Z}$, we obtain that
\[ d(e, g \cdot \gamma) = d(e, \gamma) + (l - k), \]
and then
\[ d(g \cdot p, g \cdot q) \leq d(p, q) \cdot a^{-(l-k)}. \qedhere \]
\end{proof}

\begin{fakt}
\label{fakt-duze-gwiazdy}
\dzm{Under the assumptions of Theorem~\ref{tw-bi-lip}, }there exists a constant $E$ (depending only on $G$ \dzm{and~$\mathcal{U}$}) such that, for any $k \geq 0$, the Lebesgue number of the cover $\mathcal{U}_k$ is at least $E \cdot a^{-k}$.
\end{fakt}

\begin{proof}
Let $N > N_4 \dzm{+ D}$, where $N_4$ is the constant from Proposition~\ref{lem-bi-lip-geodezyjne} \dzm{and $D$ is the neighbourhood constant of the system~$\mathcal{S}$. Denote by~$T$ the type function associated with~$\mathcal{U}$.}
Let $M>0$ be chosen so that for any $g \in G$ there exists $h \in G$ such that \dzm{$T(g) = T(h)$} and $|h| < M$. Let $L_j$ denote the Lebesgue constant for the cover $\mathcal{U}_j$ for $j < M$. We will prove that the claim of the lemma is satisfied by the number
\[ E = a^{-E_1} \cdot \min_{j<M} \, (a^j L_j), \]
where $E_1$ is the constant from Proposition~\ref{lem-bi-lip-geodezyjne}.

Let $k \geq 0$ and~$B \subset \partial G$ be a non-empty subset with diameter at most $E \cdot a^{-k}$. Let $x$ be any element of $B$, then 
\dzm{
there exist elements $g, \widetilde{g} \in G$ of length~$k$ such that
\[ x \in \sppan(g), \qquad x \in U_{\widetilde{g}}. \]
Then, we have $x \in \sppan(g) \cap U_{\widetilde{g}} \subseteq S_g \cap S_{\widetilde{g}}$, so by~\qhlink{d} it follows that $d(g, \widetilde{g}) \leq D$. 
}
By the definition of $M$, there exists $h \in G$ such that 
\[ |h| < M, \qquad \dzm{T(g) = T(h)}. \]

Denote $\gamma = hg^{-1}$ 
\dzm{
and $\widetilde{h} = \gamma \widetilde{g}$.
By Definition~\ref{def-quasi-niezm-system-wpisany}, the type function~$T$ is stronger than the ball type~$T^b_N$ (in the sense of Definition~\ref{def-funkcja-typu}). Therefore, $T^b_N(g) = T^b_N(h)$, which together with Lemma~\ref{fakt-kulowy-duzy-wyznacza-maly} and $d(g, \widetilde{g}) \leq D$ implies that $T^b_{N-D}(\widetilde{g}) = T^b_{N-D}(\widetilde{h})$. 
Since $N - D > N_4$, we obtain
}
from Proposition \ref{lem-bi-lip-geodezyjne} that
\[ \diam (\gamma \cdot B)  \leq E \cdot a^{-k} \cdot a^{-(|h| - k)} \leq a^{-|h|} \cdot \min_{j < M} (a^j L_j) \leq L_{|h|}, \]
which means that there exists $h' \in G$ such that $|h'| = |h|$ and
$\gamma \cdot B \subseteq U_{h'}$.

Let us note that it follows from \qhlink{f1} that
\[ \gamma \cdot x \in \dzm{\gamma \cdot U_{\widetilde{g}} = U_{\widetilde{h}}}, \]
so $\gamma \cdot x$ is a common element of \dzm{$U_{\widetilde{h}}$ and~$U_{h'}$}. Then, from \qhlink{f2} we obtain that
\[ U_{\gamma^{-1} h'} = \gamma^{-1} \cdot U_{h'} \supseteq B. \qedhere \]
\end{proof}

\begin{proof}[{\normalfont \textbf{Proof of Theorem~\ref{tw-bi-lip}}}]
Denote by $n$ the \dzm{maximal rank of all the covers~$\mathcal{S}_t$, for $t \geq 0$. Then, for every $t \geq 0$ we have $\rank \, \mathcal{U}_t \leq n$ and hence $\dim K_t \leq n$.} Denote also by $a_0$ a (constant) number such that the visual metric, considered for values $1 < a < a_0$, has all the properties described in Section~\ref{sec-def}. We define
\[ a_1 = \min \big( a_0, \tfrac{n+1}{n} \big). \]
Let $1 < a < a_1$. Denote by $M$ the diameter of $\partial G$ with respect to the visual metric (which is finite due to compactness of $\partial G$), and by $C_1$ --- the multiplier of bi-Lipschitz equivalence between the distance function $d$ and the visual metric.

Let $p, q$ be two distinct elements~of $\partial G$ and let $k \geq 0$ be the minimal natural number such that $d(p, q) > a^{-k}$. Observe that $d(p, q) \leq a^{-(k-1)}$ if $k > 0$, while $d(p, q) \leq MC_1 \leq MC_1 \cdot a^{-(k-1)}$ in the other case, so in general we have:
\begin{align}
\label{eq-bi-lip-lapanie-ujemnych}
 d(p, q) \leq M' \cdot a^{-(k-1)}, \qquad \textrm{ where } \qquad M' = \max(MC_1, 1).
\end{align}

\dzm{As in Definition~\ref{def-konstr-nerwy}, we let~$\pi_n$ denote the projection from~$\liminv K_i$ to $K_n$. Our goal is to estimate $d^M_a(\overline{p}, \overline{q})$, where $\overline{p}$, $\overline{q}$ denote correspondingly the images of $p, q$ under $\varphi$.}

First, we will estimate $d^M_a(\overline{p}, \overline{q})$ from above. Let $l$ be the maximal number not exceeding $k - \log_a E$. We consider two cases:
\begin{itemize}
\item If $l < 0$, then $k < \log_a E$, and then, by Remark~\ref{uwaga-metryka-l1-ogr},
\[ d^M_a(\overline{p}, \overline{q}) = \sum_{t = 0}^\infty a^{-t} \cdot d_{K_t} \big( \pi_t(\overline{p}), \pi_t(\overline{q}) \big) \leq \sum_{t = 0}^\infty a^{-t} \cdot 2 \leq \frac{2a}{a-1} \leq \frac{2a}{(a-1)MC_1} \cdot d(p, q). \]
\item If $l \geq 0$, then by Lemma~\ref{fakt-duze-gwiazdy} there exists $U \in \mathcal{U}_l$ containing both $p$ and $q$.
Then, in the complex $K_l$, the points $\pi_l(\overline{p})$ and~$\pi_l(\overline{q})$ must lie in some (possibly different) simplexes containing the vertex $v_U$. Then, we have
\[ d_{K_l} \big( \pi_l(\overline{p}),v_U \big) \leq 2, \qquad d_{K_l} \big( \pi_l(\overline{q}), v_U \big) \leq 2 \]
by Remark \ref{uwaga-metryka-l1-ogr}, and moreover
\[ d_{K_t} \big( \pi_t(\overline{p}), \pi_t(\overline{q}) \big) \leq 2 \cdot 2 \cdot \big( \tfrac{n}{n+1} \big)^{l-t} \qquad \textrm{ for } \quad 0 \leq t \leq l. \]
by the condition \eqref{eq-konstr-szacowanie-obrazow} from the proof of theorem~\ref{tw-konstr} (which we may use here because we are now working with the same metric in $K_i$ which was used in~\cite{E}). Then, since $\diam K_t \leq 2$ for $t \geq 0$ (by Remark \ref{uwaga-metryka-l1-ogr}) and $\tfrac{an}{n+1} < 1$, we have
\begin{align*} 
d^M_a(\overline{p}, \overline{q}) & = \sum_{t = 0}^\infty a^{-t} \cdot d_{K_t} \big( \pi_t(\overline{p}), \pi_t(\overline{q}) \big) \leq \sum_{t = 0}^l a^{-t} \cdot 4 \cdot \big( \tfrac{n}{n+1} \big)^{l-t} + \sum_{t = l+1}^\infty a^{-t} \cdot 2 \leq \\
 & \leq 4 a^{-l} \cdot \sum_{t = 0}^l \big( \tfrac{an}{n+1} \big)^{l-t} + 2 \cdot \sum_{t = l+1}^\infty a^{-t} \leq C_2 \cdot a^{-l} \leq (C_2Ea) \cdot a^{-k} \leq (C_2Ea) \cdot d(p, q),
\end{align*}
where $C_2$ is some constant depending only on $a$ and~$n$ (and so independent of~$p, q$).
\end{itemize}

The opposite bound will be obtained by Lemma \ref{fakt-pokrycie-male}. Let $C$ denote the constant from that lemma and let $l'$ be the smallest integer greater than $k + \log_a(3C)$. Then, Lemma \ref{fakt-pokrycie-male} ensures that, for any $t \geq l'$ and \dzm{$x \in G$ of length~$t$}, we have
\[ \dzm{ \diam_d U_x \leq \diam_d S_x } \leq 3C \cdot a^{-t} \leq a^{-k} < d(p, q), \]
so the points $p, q$ cannot belong simultaneously to any element of the cover $\mathcal{U}_t$. Then, \dzm{by the definition of $\varphi: \partial G \simeq \liminv K_i$,} the points $\pi_t(\overline{p})$, $\pi_t(\overline{q})$ lie in some two disjoint simplexes in $K_t$, and so, by Lemma \ref{fakt-rozlaczne-symp-daleko}, their distance is equal to $2$. Then, by \eqref{eq-bi-lip-lapanie-ujemnych}, we have:
\[ d^M_a(\overline{p}, \overline{q}) = \sum_{t = 0}^\infty a^{-t} \cdot d_{K_t} \big( \pi_t(\overline{p}), \pi_t(\overline{q}) \big) \geq \sum_{t = l'}^\infty a^{-t} \cdot 2 \geq \frac{2a^{-l'} \cdot a}{a-1} \geq \tfrac{2}{3C(a-1)} \cdot a^{-k} \geq \tfrac{2}{3CM'a(a-1)} \cdot d(p, q). \]
In view of Remark \ref{uwaga-bi-lip-wystarczy-d}, this finishes the proof.
\end{proof}

\section{Markov property}

\label{sec-markow}

The main goal of this section is to prove the following theorem:

\begin{tw}
\label{tw-kompakt-ogolnie}
Let $(\mathcal{U}_n)_{n \geq 0}$ be a quasi-$G$-invariant system of covers of a compact, metric $G$-space-$X$. Let $L_0$ denote the constant given by Proposition \ref{lem-gwiazda} for this system, $L \geq L_0$ and let $(K_n, f_n)$ be the associated inverse system of nerves obtained for the sequence of the covers $(\widetilde{\mathcal{U}}_{nL})_{n \geq 0}$ (see Definition~\ref{def-konstr-nerwy}). 

Then, the system $(K_n, f_n)$ is barycentric, Markov and has the mesh property. 
\end{tw}

The proof of this theorem appears --- after a number of auxiliary definitions and facts --- in Section~\ref{sec-markow-podsum}.

\subsection{Simplex types and translations}

\label{sec-markow-typy}

Below (in Definition \ref{def-typ-sympleksu}) we define simplex \textit{types} which we will use to prove the Markov property of the system $(K_n, f_n)$. Intuitively, we would like the type of a simplex $s = [v_{U_{g_1}}, \ldots, v_{U_{g_k}}]$ to contain the information about types of elements $g_i$ (which seems to be natural), but also about their relative position in $G$ (which, as we will see in Section \ref{sec-markow-synowie}, will significantly help us in controlling the pre-images of the maps $f_n$).

However, this general picture becomes more complicated because we are not guaranteed a unique choice of an element $g$ corresponding to a given set $U_g \in \widetilde{\mathcal{U}}_n$. 
Therefore, in the type of a simplex, we will store information about relative positions of \textit{all} elements of~$G$ representing its vertices.

As an effect of the above considerations, we will obtain a quite complicated definition of type (which will be only rarely directly referred to). An equality of such types for given two simplexes can be conveniently described by existence of a \textit{shift} between them, preserving the simplex structure described above (see Definition~\ref{def-przesuniecie-sympleksu}). This property will be used in a number of proofs in the following sections.

We denote by $Q_n$ the nerve of the cover $\widetilde{\mathcal{U}}_n)$. (Then, $K_n = Q_{nL}$).

\begin{df}
\label{def-graf-typu-sympleksu}
For a simplex $s$ in $Q_n$, we define a directed graph $G_s = (V_s, E_s)$ in the following way:
\begin{itemize}
 \item the vertices in $G_s$ are \textit{all} the elements $g \in G$ for which $v_{U_g}$ is a vertex in~$s$ (and so $|g| = n$ by Remark~\ref{uwaga-quasi-niezm-rozne-poziomy}); thus, $G_s$ may possibly have more vertices than $s$ does;
 \item every vertex $g \in V_s$ is labelled with its type $T(g)$;
 \item the edges in~$G_s$ are all pairs $(g, g')$ for $g, g' \in V_s$, $g \neq g'$;
 \item every edge $(g, g')$ is labelled with the element $g^{-1} g' \in G$.
\end{itemize}
\end{df}

\begin{df}
\label{def-typ-sympleksu}
We call two simplexes $s \in Q_n$ and $s' \in Q_{n'}$ \textit{similar} if there exists an isomorphism of graphs $\varphi: G_s \rightarrow G_{s'}$ preserving all labels of vertices and edges.

The \textit{type} of a simplex $s \in Q_n$ (denoted by $T^\Delta(s)$) is its similarity class. \\
(Hence: two simplexes are similar if and only if they have the same type).
\end{df}

\begin{df}
\label{def-przesuniecie-sympleksu}
A simplex $s' \in Q_{n'}$ will be called the \textit{shift} of a simplex $s \in Q_n$ by an element~$\gamma$ (notation: $s' = \gamma \cdot s$) if the formula $\varphi(g) = \gamma \cdot g$ defines an isomorphism $\varphi$ which satisfies the conditions of Definition \ref{def-typ-sympleksu}.
\end{df}

\begin{fakt}
\label{fakt-przesuniecie-skladane}
Shifting simplexes satisfies the natural properties of a (partial) action of $G$ on a set:
\[ \textrm{ if } \qquad s' = \gamma \cdot s \quad \textrm{ and } \quad s'' = \gamma' \cdot s', \qquad \textrm{ then } \qquad s = \gamma^{-1} \cdot s' \quad \textrm{ and } \quad s'' = (\gamma' \, \gamma) \cdot s. \qedthm \]
\end{fakt}

\begin{fakt}
\label{fakt-przesuniecie-istnieje}
Two simplexes $s \in Q_n$, $s' \in Q_{n'}$ have equal types  $\ \Longleftrightarrow\ $ $s' = \gamma \cdot s$ for some $\gamma \in G$.
\end{fakt}

\begin{proof}
The implication  $(\Leftarrow)$ is obvious. On the other hand, let $\varphi : G_s \rightarrow G_{s'}$ be an isomorphism satisfying the conditions from Definition \ref{def-typ-sympleksu}. We choose arbitrary $g_0 \in V_s$ and define $\gamma = \varphi(g_0) \, g_0^{-1}$. Since $\varphi$ preserves the labels of edges, for any $g \in V_s \setminus \{ g_0 \}$ we have
\[ \varphi(g_0)^{-1} \, \varphi(g) = g_0^{-1} \, g \quad \Rightarrow \quad \varphi(g) \, g^{-1} = \varphi(g_0) \, g_0^{-1} = \gamma \quad \Rightarrow \quad \varphi(g) = \gamma \cdot g. \qedhere \]
\end{proof}

\begin{fakt}
\label{fakt-przesuwanie-symp-zb}
If $s' = \gamma \cdot s$ and $v_{U_x}$ is a vertex in $s$, then $v_{U_{\gamma x}}$ is a vertex in~$s'$ and moreover
\[ U_{\gamma x} = \gamma \cdot U_x, \qquad T(\gamma x) = T(x). \]
In particular, shifting the sets from $\widetilde{\mathcal{U}}$ by~$\gamma$ gives a bijection between the vertices of $s$ and~$s'$.
\end{fakt}

\begin{proof}
This follows from Definitions \ref{def-graf-typu-sympleksu} and~\ref{def-przesuniecie-sympleksu}, and from property~\qhlink{f1}.
\end{proof}

\begin{lem}
The total number of simplex types in all of the complexes $Q_n$ is finite.
\end{lem}

\begin{proof}
Let us consider a simplex $s \in K_n$. If $g, g' \in V_s$, then the vertices $v_{U_g}, v_{U_{g'}}$ belong to $s$, which means by definition that $U_g \cap U_{g'} \neq \emptyset$, and then, by \qhlink{d} and the definition of $V_s$, we have $|g^{-1} g'| \leq D$.

Then, the numbers of vertices in the graphs $G_s$, as well as the number of possible edge labels appearing in all such graphs, are not greater than the cardinality of the ball $B(e, D)$ in the group~$G$. This finishes the proof because the labels of vertices are taken by definition from the finite set of types of elements in~$G$.
\end{proof}

\subsection{The main proposition}

\label{sec-markow-synowie}

\begin{lem}
\label{lem-przesuwanie-dzieci-sympleksow}
Let $s \in K_n$, $s' \in K_{n'}$  be simplexes of the same type and $s' = \gamma \cdot s$ for some $\gamma \in G$.
Then, the maps $I : s \rightarrow s'$ and $J : f_n^{-1}(s) \rightarrow f_n^{-1}(s')$, defined on the vertices of the corresponding subcomplexes by the formulas
 \[ I(v_U) = v_{\gamma \cdot U} \quad \textrm{ for } v_U \in s, \qquad J(v_U) = v_{\gamma \cdot U} \quad \textrm{ for } v_U \in f_n^{-1}(s), \]
and extended affinely to the simplexes in these subcomplexes, have the following properties:
 \begin{itemize}
  \item they are well defined (in particular, $\gamma \cdot U$ is an element of the appropriate cover);
  \item they are isomorphisms of subcomplexes;
  \item they map simplexes to their shifts by $\gamma$ (in particular, they preserve simplex types).
 \end{itemize}
Moreover, the following diagram commutes:
\begin{align}
\label{eq-diagram-do-spr}
\xymatrix@+3ex{
 s \ar[d]_{I} & \ar[l]_{f_n} f_n^{-1}(s) \ar[d]_{J} \\
 s' & \ar[l]_{f_{n'}} f_{n'}^{-1}(s').
}
\end{align}
\end{lem}

\begin{proof}
Let
\begin{align} 
\label{eq-markow-wstep-1}
s = [v_{U_1}, \ldots, v_{U_k}], \qquad U_i = U_{g_i}, \qquad g_i' = \gamma \, g_i, \qquad U_i' = U_{g_i'}. 
\end{align}
Then, from the assumptions (using the definitions and Lemma \ref{fakt-przesuwanie-symp-zb}) we obtain that
\[ s' = [v_{U'_1}, \ldots, v_{U'_k}], \qquad U_i' = \gamma \cdot U_i, \qquad T(g_i) = T(g_i'), \qquad |g_i| = nL, \qquad |g'_i| = n'L. \]
In particular, for every $v_U \in s$ the value $I(v_U)$ is correctly defined and belongs to $s'$; also, $I$ gives a bijection between the vertices of $s$ and $s'$, so it is an isomorphism. Moreover, for any subsimplex $\sigma = [v_{U_{i_1}}, \ldots, v_{U_{i_l}}] \subseteq s$ and $g \in G_s$, by Lemma \ref{fakt-przesuwanie-symp-zb} we have an equivalence
\[ v_{U_g} \in \sigma \quad \Leftrightarrow \quad U_g \in \{ U_{i_j} \,|\, 1 \leq j \leq l \} \quad \Leftrightarrow \quad U_{\gamma g} \in \{ \gamma \cdot U_{i_j} \,|\, 1 \leq j \leq l \} \quad \Leftrightarrow \quad v_{U_{\gamma g}} \in I(\sigma), \]
so the isomorphism $G_s \simeq G_{s'}$ given by $\gamma$ restricts to an isomorphism $G_\sigma \simeq G_{I(\sigma)}$, so $I(\sigma) = \gamma \cdot \sigma$.

It remains to check the desired properties of the map $J$, and commutativity of the diagram \eqref{eq-diagram-do-spr}.

First, we will check that $J$ is correctly defined. Let $v_U$ be a vertex in ~$f_n^{-1}(s)$ and let $U = U_h$ for some $h \in G$ of length $(n + 1)L$. From the definition of $f_n$ we obtain that $U_h \subseteq U_{g_i}$ for some $1 \leq i \leq k$. Then, denoting $h' = \gamma h$ and using~\qhlink{f3}, we have
\begin{align} 
\label{eq-markov-wlasnosci-h'}
U_{h'} = \gamma \cdot U_h \subseteq \gamma \cdot U_{g_i} = U_{g'_i}, \qquad T(h') = T(h), \qquad |h'| = (n' + 1)L, 
\end{align}
so in particular $\gamma \cdot U_h \in \widetilde{\mathcal{U}}_{(n'+1)L}$, and then $J(v_U) = v_{\gamma \cdot U_h}$ is a vertex in~$K_{n' + 1}$.

Now, we will prove that the vertex $J(v_U)$ belongs to $f_{n'}^{-1}(s')$ and that the diagram~\eqref{eq-diagram-do-spr} commutes.
From the definition of maps $f_n$, $f_{n'}$ it follows that, for both these purposes, it is sufficient to prove that
\begin{align} 
\label{eq-markov-zgodnosc-rodzicow-ogolnie}
\big\{ U' \,\big|\, U' \in \mathcal{U}_{n'L}, \, U' \supseteq U_{h'} \big\} = \big\{ \gamma \cdot U \,\big|\, U \in \mathcal{U}_{nL}, \, U \supseteq U_h \big\}. 
\end{align}
Let us check the inclusion $(\supseteq)$. Let $U_g = U \supseteq U_h$ for some $g \in G$ of length~$nL$. Then in particular $U_g \cap U_{g_i} \supseteq U_h \neq \emptyset$, so from property \qhlink{f2} we have
\[ U_{\gamma g} = \gamma \cdot U_g \supseteq \gamma \cdot U_h = U_{h'}, \qquad |\gamma g| = n' L. \]

This proves the inclusion $(\supseteq)$ in~\eqref{eq-markov-zgodnosc-rodzicow-ogolnie}. Since $U_{g_i'} = \gamma \cdot U_{g_i} \supseteq \gamma \cdot U_h = U_{h'}$, the opposite inclusion can be proved in an analogous way, by considering the shift by $\gamma^{-1}$. Hence, we have verified \eqref{eq-markov-zgodnosc-rodzicow-ogolnie}. In particular, we obtain that $J(v_U) \in f_{n'}^{-1}(s')$ for every $v_U \in f_n^{-1}(s)$, and so $J$ is correctly defined on the vertices of the complex $f_n^{-1}(s)$. From~\eqref{eq-markov-zgodnosc-rodzicow-ogolnie} we also deduce the commutativity of the diagram \eqref{eq-diagram-do-spr} when restricted to the vertices of the complexes considered.

Now, let $\sigma = [v_{U_{h_1}}, \ldots, v_{U_{h_l}}]$ be a simplex in~$f_n^{-1}(s)$. Then, we have $\bigcap_{i = 1}^l U_{h_i} \neq \emptyset$, so also $\bigcap_{i=1}^l U_{\gamma h_i} = \gamma \cdot \bigcap_{i=1}^l U_{h_i} \neq \emptyset$. This implies that in $f_{n'}^{-1}(s')$ there is a simplex $[J(v_{U_{h_1}}), \ldots, J(v_{U_{h_l}})]$. This means that $J$ can be affinely extended from vertices to simplexes, leading to a correctly defined map of complexes. The commutativity of the diagram \eqref{eq-diagram-do-spr} is then a result from the (already checked) commutativity for vertices.
By exchanging the roles of $s$ and $s'$, and correspondingly of $g_i$ and $g_i'$, we obtain an exchange of roles between $\gamma$ and $\gamma^{-1}$. Therefore, the map $\widetilde{J} : f_{n'}^{-1}(s') \rightarrow f_n^{-1}(s)$, obtained analogously for such situation, must be inverse to $J$. Hence, $J$ is an isomorphism.

It remains to check the equality $J(\sigma) = \gamma \cdot \sigma$ for any simplex $\sigma$ in~$f_n^{-1}(s)$. For this, we choose any element $h \in V_\sigma$ (i.e. a vertex from the graph $G_\sigma$ from Definition~\ref{def-graf-typu-sympleksu}) and take $\varphi(h) = \gamma h$; this element was previously denoted by $h'$. Then, from \eqref{eq-markov-wlasnosci-h'} and the already checked properties of $J$, we deduce that
\[ v_{U_{\gamma h}} = v_{\gamma \cdot U_h} = J(v_{U_h}) \, \textrm{ is a vertex in } J(\sigma), \qquad T(\gamma h) = T(h). \]
 The first of these facts means that $\varphi(h)$ indeed belongs to $V_{J(\sigma)}$; the second one ensures that $\varphi$ preserves the labels of vertices in the graphs. 
Preserving the labels of edges follows easily from the definition of $\varphi$. Hence, it remains only to check that $\varphi$ gives a bijection between $V_\sigma$ and ~$V_{J(\sigma)}$, which we obtain by repeating the above reasoning for the inverse map $J^{-1}$ (with $s'$, $J(\sigma)$ playing now the roles of $s$, $\sigma$).
\end{proof}

\subsection{Conclusion: $\partial G$ is a Markov compactum}

\label{sec-markow-podsum}

Although we will be able to prove Theorem~\ref{tw-kompakt} in its full strength only at the end of Section~\ref{sec-wymd}, we note that the results already obtained imply the main claim of this theorem, namely that Gromov boundaries of hyperbolic groups are always Markov compacta (up to homeomorphism). Before arguing for that, we will finish the proof of Theorem~\ref{tw-kompakt-ogolnie}.

\begin{proof}[{\normalfont \textbf{Proof of Theorem~\ref{tw-kompakt-ogolnie}}}]
Barycentricity and the mesh property for the system $(K_n, f_n)$ follow from Theorem \ref{tw-konstr}; it remains to check the Markov property for this system. The condition (ii) from Definition \ref{def-kompakt-markowa} follows from the way in which the maps $f_n$ are defined in the claim of Theorem~\ref{tw-konstr}, while (i) is a result of the assumption (ii) in Definition \ref{def-konstr-admissible} (admissibility of $\mathcal{U}$). It remains to check (iii).

The type which we assign to simplexes is the $T^\Delta$-type from Definition \ref{def-typ-sympleksu}. If two simplexes $s \in K_i$, $s' \in K_j$ have the same type, then by Lemma \ref{fakt-przesuniecie-istnieje} we know that $s' = \gamma \cdot s$ for some $\gamma \in G$. Then, by an inductive application of Proposition \ref{lem-przesuwanie-dzieci-sympleksow}, we deduce that for every $k \geq 0$ the simplexes in the pre-image $(f^{j+k}_j)^{-1}(s')$ coincide with the shifts (by $\gamma$) of simplexes in the pre-image $(f^{i+k}_i)^{-1}(s)$, and moreover, by letting
\[ i_k (v_U) = v_{\gamma \cdot U} \qquad \textrm{ for } \quad k \geq 0, \ v_U \in (f^{i+k}_i)^{-1}(s), \]
we obtain correctly defined isomorphisms of subcomplexes which preserve simplex types and make the diagram from Definition \ref{def-kompakt-markowa} commute. This finishes the proof.
\end{proof}

\begin{proof}[{\normalfont \textbf{Proof of the main claim of Theorem~\ref{tw-kompakt}}}]
Let $G$ by a hyperbolic group and let~$\mathcal{S}$ be the system of covers of~$\partial G$ defined in Section~\ref{sec-konstr-pokrycia}. By~Corollary~\ref{wn-spanstary-quasi-niezm}, $\mathcal{S}$ is quasi-$G$-invariant; hence, by Lemma~\ref{fakt-konstr-sp-zal}, there is $L \geq 0$ such that the system $(\widetilde{S}_{nL})_{n \geq 0}$ is admissible. By applying Theorem~\ref{tw-konstr}, we obtain an inverse system $(K_n, f_n)$ whose inverse limit is homeomorphic to~$\partial G$; on the other hand, Theorem~\ref{tw-kompakt-ogolnie} ensures that this system is Markov. This finishes the proof.
\end{proof}

\begin{uwaga}
Theorem~\ref{tw-kompakt-ogolnie} ensures also barycentricity and mesh property for the obtained Markov system.
\end{uwaga}

\section{Strengthenings of type}
\label{sec-abc}

In this section, we will construct new \textit{type functions} in the group $G$ (in the sense of Definition~\ref{def-funkcja-typu}), or in the inverse system $(K_n)$ constructed in Section~\ref{sec-engelking} (in an analogous sense, i.e. we assign to every simplex its \textit{type} taken from a finite set), with the aim of ensuring 
certain regularity conditions of these functions which will be needed in the next sections.

The basic condition of our interest, which will be considered in several flavours, is ``children determinism'': the type of an element (resp. a simplex) should determine the type of its ``children'' (in an appropriate sense), analogously to the properties of the ball type $T^b_N$ described in Proposition \ref{lem-potomkowie-dla-kulowych}. The most important result of this section is the construction of a new type (which we call \textit{$B$-type} and denote by $T^B$) which, apart from being children-deterministic in such sense, returns different values for any pair of \textit{$r$-fellows} in $G$ (see Definition~\ref{def-konstr-towarzysze}) for some fixed value of $r$. (To achieve the goals of this article, we take $r = 16\delta$). This property will be crucial in three places of the remaining part of the paper:

\begin{itemize}
 \item In Section \ref{sec-sk-opis}, we will show that, by including the $B$-type in the input data for Theorem \ref{tw-kompakt-ogolnie}, we can ensure the resulting Markov system has distinct types property (see Definition~\ref{def-kompakt-wlasciwy}), which will in turn guarantee its finite describability (see Remark \ref{uwaga-sk-opis}).
 
 \item In Section \ref{sec-wymd}, this property will allow us a~kind of ``quasi-$G$-invariant control'' of simplex dimensions in the system $(K_n)$. 
 
 \item In Section \ref{sec-sm}, the $B$-type's property of distinguishing fellows will be used to present the boundary $\partial G$ as a semi-Markovian space (see Definition~\ref{def-sm-ps}).
\end{itemize}
Let us note that this property of $B$-type is significantly easier to be achieved in the case of torsion-free groups (see the introduction to Section \ref{sec-sm-abc-b}).

A natural continuation of the topic of this section will also appear in Section \ref{sec-sm-abc-c}, in which we will enrich the $B$-type to obtain a new \textit{$C$-type}, which will serve directly as a basis for the presentation of $\partial G$ as a semi-Markovian space. (We postpone discussing the $C$-type to Section \ref{sec-sm} because it is needed only there, and also because we will be able to list its desired properties only as late as in Section~\ref{sec-sm-zyczenia}).

\subsubsection*{Technical assumptions}

In Sections \ref{sec-abc}--\ref{sec-wymd}, we assume that $N$ and~$L$ are fixed and sufficiently large constants; explicit bounds from below will be chosen while proving consecutive facts. (More precisely, we assume that $N$ satisfies the assumptions of Corollary \ref{wn-sm-kulowy-wyznacza-potomkow} and Proposition \ref{lem-sm-kuzyni}, and that $L \geq \max(N + 4\delta, 14\delta)$ satisfies the assumptions of Lemma \ref{fakt-sm-rodzic-kuzyna}; some of these bounds will be important only in Section \ref{sec-wymd}). Let us note that, under such assumptions, Proposition \ref{lem-potomkowie-dla-kulowych} ensures that the ball type $T^b_N(x)$ determines the values of $T^b_N$ for all descendants of $x$ of length $\geq |x| + L$.

The types which we will construct --- similarly as the ball type $T^b$ --- will depend on the value of a parameter $N$ (discussed in the above paragraph) which, for simplicity, will be omitted in the notation.

Also, we assume that the fixed generating set~$S$ of the group~$G$ (which we are implicitly working with throughout this paper) is closed under taking inverse, and we fix some enumeration $s_1, \ldots, s_Q$ of all its elements (This will be used in Section \ref{sec-abc-pp}).

\subsection{Prioritised ancestors}
\label{sec-abc-pp}

\begin{df}
Let $x, y \in G$. We call $y$ a \textit{descendant} of $x$ if $|y| = |x| + d(x, y)$. (Equivalently: if $y \in xT^c(x)$). In such situation, we say that $x$ is an \textit{ancestor} of~$y$.

If in addition $d(x, y) = 1$, we say that $y$ is a \textit{child} of~$x$ and $x$ is a \textit{parent} of~$y$.
\end{df}

\begin{df}
\label{def-sm-nic-sympleksow}
The \textit{prioritised parent} (or \textit{p-parent}) of an element $y \in G \setminus \{ e \}$ is the element $x \in G$ such that $x$ is a parent of~$y$ and $x = y s_i$ with $i$ least possible. The p-parent of $y$ will be denoted by $y^\uparrow$.
An element $g' \in G$ is a \textit{priority child} (or \textit{p-child}) of $g$ if $g$ is its p-parent.

As already suggested, a given element of $G \setminus \{ e \}$ must have exactly one p-parent but may have many p-children.

The relation of \textit{p-ancestry} (resp. \textit{p-descendance}) is defined as the reflexive-transitive closure of p-parenthood (resp. p-childhood); in particular, for any $g \in G$ and $k \leq |g|$, $g$ has exactly one p-ancestor $g'$ such that~$|g'| = |g| - k$, which we denote by $g^{\uparrow k}$.
\end{df}

\begin{df}
\label{def-sm-p-wnuk}
Let $x, y \in G$. We call $x$ a \textit{p-grandchild} (resp. \textit{p-grandparent}) of $y$ if it is a p-descendant (resp. p-ancestor) of $y$ and $\big| |x| - |y| \big| = L$. For simplicity of notation, we denote $x^\Uparrow = x^{\uparrow L}$ (for $|x| \geq L$) and analogously $x^{\Uparrow k} = x^{\uparrow Lk}$ (for $|x| \geq Lk$).
\end{df}

Let $N_0$ denote the constant coming from Proposition \ref{lem-kulowy-wyznacza-stozkowy}.

\begin{fakt}
\label{fakt-sm-kulowy-wyznacza-dzieci}
Let $x, y, s \in G$ satisfy $T^b_{N_0+2}(x) = T^b_{N_0+2}(y)$ and $|s| = 1$. Then, $(xs)^\uparrow = x$ if and only if $(ys)^\uparrow = y$.
\end{fakt}

\begin{proof}
Since the set $S$ is closed under taking inverse, we have $s = s_i^{-1}$ for some~$i$. Assume that $(xs_i^{-1})^\uparrow = x$ but $(ys_i^{-1})^\uparrow = z \neq y$. Then, $z = ys_i^{-1}s_j$ for some $j < i$. Note that $d(y, z) \leq 2$. Define $z' = xs_i^{-1}s_j$; then by Lemma \ref{fakt-kulowy-duzy-wyznacza-maly} and Proposition \ref{lem-kulowy-wyznacza-stozkowy} we have
\[ T^b_{N_0}(z') = T^b_{N_0}(z), \qquad s_j^{-1} \in T^c(z) = T^c(z'). \]
This means that the element $xs_i^{-1} = z's_j^{-1}$ is a child of $z'$; then, since $j < i$, it cannot be a p-child of $x$.
\end{proof}

\begin{wn}
\label{wn-sm-kulowy-wyznacza-potomkow}
Let $N \geq N_5 := 2N_0 + 8\delta + 2$. Then, if $x, y \in G$ satisfy $T^b_N(x) = T^b_N(y)$, the left translation by $\gamma = yx^{-1}$ gives a bijection between the p-descendants of $x$ and the p-descendants of $y$.
\end{wn}

\begin{proof}
Let $z$ be a p-descendant of $x$; we want to prove that $\gamma z$ is a p-descendant of~$y$. We will do this by induction on the difference $|z| - |x|$.

If $|z| = |x|$, the claim is obvious. Assume that $|z| > |x|$ and denote $w = z^\uparrow$; then $w$ is a p-descendant of $x$ for which we may apply the induction assumption, obtaining that $\gamma w$ is a p-descendant of $y$. To finish the proof, it suffices to check that
\begin{align} 
\label{eq-sm-typy-potomkow}
T^b_{N_0 + 2}(w) = T^b_{N_0 + 2}(\gamma w), 
\end{align}
since then from Lemma \ref{fakt-sm-kulowy-wyznacza-dzieci} it will follow that $\gamma z$ is a p-child of $\gamma w$, and then also a p-descendant of $y$. We consider two cases:
\begin{itemize}
 \item If $d(z, x) \leq N - (N_0 + 2)$, then \eqref{eq-sm-typy-potomkow} holds by equality $T^b_N(x) = T^b_N(y)$ and Lemma \ref{fakt-kulowy-duzy-wyznacza-maly}.
 \item If $|z| - |x| \geq N_0 + 4\delta + 2$, then \eqref{eq-sm-typy-potomkow} follows (in view of the equality $T^b_N(x) = T^b_N(y)$ and the inequality $N \geq N_0 + 8\delta$) from Proposition \ref{lem-potomkowie-dla-kulowych}.
\end{itemize}
Since $d(z, x) = |z| - |x|$ (because $z$ is a descendant of $x$) and $N \geq 2N_0 + 4\delta + 2$, at least one of the two cases must hold. This finishes the proof.
\end{proof}

\subsection{The $A$-type}
\label{sec-sm-abc-a}

Let $\mathcal{T}^b_N$ be the set of values of the type $T^b_N$. As an introductory step, we define the \textit{$Z$-type} $T^Z$ in $G$ so that:
\begin{itemize}
 \item $T^Z$ is compatible with $T^b_N$ for all elements $g \in G$ with length $\geq L$;
 \item $T^Z$ assigns pairwise distinct values, not belonging to $\mathcal{T}^b_N$, to all elements of length $< L$.
\end{itemize}
Let us note that such a strengthening preserves most of the properties of $T^b_N$, in particular those described in Proposition \ref{lem-potomkowie-dla-kulowych} and Corollary \ref{wn-sm-kulowy-wyznacza-potomkow}.

\begin{uwaga}
\label{uwaga-pomijanie-N}
 Although the values of $T^Z$ depend on the value of $N$, for simplicity we omit this fact in notation, assuming $N$ to be fixed. (We are yet not ready to state our assumptions on $N$; this will be done below in Proposition \ref{lem-sm-kuzyni}).
\end{uwaga}

Let $\mathcal{T}^Z$ be the set of all possible values of $T^Z$. For every $\tau \in \mathcal{T}^Z$, choose some \textit{representative} $g_\tau \in G$ of this type. For convenience, we denote by $\gamma_g$ the element of~$G$ which (left-) translates $g$ to the representative (chosen above) of its $Z$-type:
\[ \gamma_g = g_{T^Z(g)} \, g^{-1} \qquad \textrm{ for } \quad g \in G. \]
Let us recall that, by Corollary \ref{wn-sm-kulowy-wyznacza-potomkow}, the left translation by $\gamma_g$ gives a bijection between p-descendants of $g$ and p-descendants of $g_{T^Z(g)}$, and thus also a bijection between the p-grandchildren of these elements.

For every $\tau \in \mathcal{T}$, let us fix an (arbitrary) enumeration of all p-grandchildren of $g_\tau$.

\begin{df}
\label{def-sm-numer-dzieciecy}
The \textit{descendant number} of an element $g \in G$ with $|g| \geq L$ (denoted by~$n_g$) is the number given (in the above enumeration) to the element $\gamma_{g^\Uparrow} \cdot g$ as a p-grandchild of $g_{T^Z(g^\Uparrow)}$. If $|g| < L$, we set $n_g = 0$.
\end{df}

\begin{df}
\label{def-sm-typ-A}
We define the \textit{$A$-type} of an element $g \in G$ as the pair $T^A(g) = (T^Z(g), n_g)$. 
\end{df}
We note that the set $\mathcal{T}^A$ of all possible $A$-types is finite because the number of p-grandchildren of $g$ depends only on $T^Z(g)$.

\begin{fakt}
\label{fakt-sm-istnieje-nss}
For any two distinct elements $g, g' \in G$ of equal length, there is $k \geq 0$ such that $g^{\Uparrow k}$ and $g'^{\Uparrow k}$ exist and have different $A$-types.
\end{fakt}

\begin{proof}
Let $n = |g| = |g'|$ and let $k \geq 0$ be the least number for which $g^{\Uparrow k} \neq g'^{\Uparrow k}$. If $n - kL < L$, then by definition $g^{\Uparrow k}$ and $g'^{\Uparrow k}$ have different $Z$-types. Otherwise, we have $g^{\Uparrow k+1} = g'^{\Uparrow k+1}$, which means that $g^{\Uparrow k}$ and~$g'^{\Uparrow k}$ have different descendant numbers.
\end{proof}

\begin{fakt}
\label{fakt-sm-przen-a}
If $g, h \in G$ have the same $Z$-type, then the left translation by $\gamma = hg^{-1}$ maps p-grandchildren of~$g$ to p-grandchildren of~$h$ and preserves their $A$-types.
\end{fakt}

\begin{proof}
Denote $\tau = T^Z(g) = T^Z(h)$. Let $g'$ be a p-grandchild of $g$; then $\gamma g'$ is a p-grandchild of $h$ by Corollary \ref{wn-sm-kulowy-wyznacza-potomkow}. Moreover, from Proposition \ref{lem-potomkowie-dla-kulowych} we know that $g'$ and $\gamma g'$ have equal ball types $T^b_N$, and since they both have the same length $\geq L$, it follows that they have equal $Z$-types. They also have equal descendant numbers because
\[ \gamma_{(\gamma g')^\Uparrow} \, \gamma g' = g_\tau \, h^{-1} \, \gamma \, g' = g_\tau \, g^{-1} \, g' = \gamma_{g'^\Uparrow} \, g'. \qedhere \]
\end{proof}

\subsection{Cousins and the $B$-type}
\label{sec-sm-abc-b}

The aim of this subsection is to strengthen the type function to distinguish any pair of neighbouring elements in $G$ of the same length. In the case of a torsion-free group, there is nothing new to achieve as the desired property is already satisfied by the ball type $T^b_N$ (by Lemma \ref{fakt-kuzyni-lub-torsje}). In general, the main idea is to remember within the type of $g$ the ``crucial genealogical difference'' between $g$ and every it \textit{cousin} (i.e. a neighbouring element of the same length --- see Definition \ref{def-sm-kuzyni}) --- where, to be more precise, the ``genealogical difference'' between $g$ and $g'$ consists of the $A$-types of their p-ancestors in the oldest generation in which these p-ancestors are still distinct. However, it turns out that preserving all the desired regularity properties (in particular, children determinism) requires remembering not only the genealogical differences between $g$ and its cousins, but also similar differences between any pair of its cousins.

\begin{df}
\label{def-sm-sasiedzi}
Two elements $x, y \in G$ will be called \textit{neighbours} (denotation: $x \leftrightarrow y$) if $|x| = |y|$ and $d(x, y) \leq 8\delta$.
\end{df}

Denote
\begin{align} 
\label{eq-sm-zbiory-torsji}
\begin{split}
 Tor & = \big\{ g \in G \,\big|\, |g| \leq 16\delta, \, g \textrm{ is a torsion element} \big\},  \\
 R & = \max \Big( \big\{ |g^n| \,\big|\, g \in Tor, \, n \in \mathbb{Z} \big\} \cup \{ 16 \delta \} \Big).
\end{split}
\end{align}
Since~$|Tor| < \infty$, it follows that $R < \infty$. 

\begin{df}
\label{def-sm-kuzyni}
 Two elements $g, g' \in G$ will be called \textit{cousins} if $|g| = |g'|$ and $d(g, g') \leq R$. The set of cousins of $g$ will be denoted by $C_g$.  (It is exactly the set $g P_R(g)$ using the notation of Section \ref{sec-konstr-pokrycia}).
\end{df}

Let us note that if $g, h \in G$ are neighbours, then they are cousins too.

\begin{fakt}
\label{fakt-sm-rodzic-kuzyna}
If $L \geq \tfrac{R}{2} + 4\delta$, then for any cousins $g, g' \in G$ their p-grandparents are neighbours (and hence also cousins).
\end{fakt}

\begin{proof}
This is clear by Lemma \ref{fakt-geodezyjne-pozostaja-bliskie}.
\end{proof}

For $g \in G$ of length~$n$ and any two distinct $g', g'' \in C_g$, let $k_{g', g''}$ be the least value $k \geq 0$ for which $g'^{\Uparrow k}$ and $g''^{\Uparrow k}$ have different $A$-types (which is a correct definition by Lemma \ref{fakt-sm-istnieje-nss}). Let the sequence $k^{(g)} = (k^{(g)}_i)_{i = 1}^{M_g}$ come from arranging the elements of the set $\{ k_{g', g''} \,|\, g', g'' \in C_g \} \cup \{ 0 \}$ in a decreasing order.

\begin{df}
\label{def-sm-typ-B}
The \textit{$B$-type} of an element $g$ is the set
\[ T^B(g) = \Big\{ \big( g^{-1} g', \, W_{g, \, g'} \big) \ \Big|\ g' \in C_g \Big\}, \qquad \textrm{ where } \qquad W_{g, \, g'} = \Big( T^A  \big( {g'}^{\Uparrow k^{(g)}_i} \big) \Big)_{i=1}^{M_g}. \]
(We recall that the notation hides the dependence on a fixed parameter $N$, whose value will be chosen in Proposition \ref{lem-sm-kuzyni}).
\end{df}

\begin{uwaga}
\label{uwaga-sm-B-wyznacza-A}
Since $g$ is a cousin of itself and $0$ is the last value in the sequence $(k^{(g)}_i)$, the set $T^B(g)$ contains in particular a pair of the form $\big( e, \, (\ldots, T^A(g)) \big)$, which means that the $B$-type determines the $A$-type (of the same element).
\end{uwaga}

\begin{uwaga}
\label{uwaga-sm-indeksy-z-tabelki}
For any two distinct $g', g'' \in C_g$, let $i^{(g)}_{g', g''}$ be the position on which the value $k_{g', g''}$ appears in sequence~$k^{(g)}$. Then from definition it follows that this index depends only on the sequences $W_{g, g'}$ and~$W_{g, g''}$, and more precisely it is equal to the greatest $i$ for which the $i$-th coordinates of these sequences differ. (At the same time, we observe that the whole sequences $W_{g, g'}$, $W_{g, g''}$ must be different). This fact will be used in the proofs of Proposition \ref{lem-sm-B-dzieci} and~\ref{lem-sm-kuzyni}.
\end{uwaga}

\begin{fakt}
\label{fakt-sm-B-skonczony}
There exist only finitely many possible $B$-types in~$G$.
\end{fakt}

\begin{proof}
The finiteness of $A$-type results from its definition. If $g' \in C_g$, then the element $g^{-1} g'$ belongs to the ball $B(e, R)$ in~$G$, whose size is finite and independent of $g$. Since the number of cousins of $g$ is also globally limited, we obtain a limit on the length of the sequence $(k^{(g)}_i)$, which ensures a finite number of possible sequences $W_{g, g'}$.
\end{proof}

\begin{lem}
\label{lem-sm-B-dzieci}
Let $g_1, h_1 \in G$ have the same $B$-types and $\gamma = h_1g_1^{-1}$. Then, the left translation by $\gamma$ maps p-grandchildren of $g_1$ to p-grandchildren of $h_1$ and preserves their $B$-type.
\end{lem}

\begin{proof}
\textbf{1. }Let $g_2$ be a p-grandchild of $g_1$ and $h_2 = \gamma g_2$. Then, $h_1 = h_2^\Uparrow$ by Remark~\ref{uwaga-sm-B-wyznacza-A} and Lemma~\ref{fakt-sm-przen-a}; it remains to check that $T^B(g_2) = T^B(h_2)$.

Since the left multiplication by $\gamma$ clearly gives a bijection between cousins of $g_2$ and cousins of~$h_2$, it suffices to show that for any $g_2' \in C_{g_2}$ we have
\begin{align} 
\label{eq-sm-B-dzieci-cel}
W_{g_2, \, g_2'} = W_{h_2, \, h_2'}, \qquad \textrm{ where } \quad h_2' = \gamma g_2'. 
\end{align}
 
\textbf{2. }Let $g_2'$, $h_2'$ be as above. Denote $g_1' = {g_2'}^\Uparrow$; by Lemma~\ref{fakt-sm-rodzic-kuzyna}, $g_1'$ is a cousin of $g_1$. Then, the element $h_1' = \gamma g_1'$ is a cousin of $h_1$ and from the equalities $T^B(g_1) = T^B(h_1)$ and $g_1^{-1} g_1' = h_1^{-1} h_1'$ it follows that
\begin{align} 
\label{eq-sm-B-dzieci-znane-war-t1}
W_{g_1, \, g_1'} = W_{h_1, \, h_1'}. 
\end{align}
Since $0$ is the last element in the sequence $(k^{(g_1)}_i)$, as well as in $(k^{(h_1)}_i)$, we obtain in particular that $T^A(g_1') = T^A(h_1')$. By Lemma \ref{fakt-sm-przen-a}, we have:
\begin{align} 
\label{eq-sm-B-dzieci-zachowane-A-t2}
h_1' = h_2'^\Uparrow, \qquad T^A(g_2') = T^A(h_2'). 
\end{align}

\textbf{3. }For arbitrarily chosen $g_2', g_2'' \in C_{g_2}$, we denote:
\[ h_2' = \gamma \, g_2', \qquad h_2'' = \gamma \, g_2'', \qquad \tau' = T^A(g_2') = T^A(h_2'), \qquad \tau'' = T^A(g_2'') = T^A(h_2''). \]
Moreover, we denote by $g_1', g_1'', h_1', h_1''$ the p-grandparents correspondingly of $g_2', g_2'', h_2', h_2''$.
Then, by definition:
\begin{align} 
\label{eq-sm-B-dzieci-nss}
k_{g_2', g_2''} = \begin{cases}
                   k_{g_1', g_1''} + 1, & \textrm{ when } \tau' = \tau'', \\
                   0, & \textrm{ when } \tau' \neq \tau''
                  \end{cases}
\qquad \textrm{ and } \qquad                    
k_{h_2', h_2''} = \begin{cases}
                   k_{h_1', h_1''} + 1, & \textrm{ when } \tau' = \tau'', \\
                   0, & \textrm{ when } \tau' \neq \tau''.
                  \end{cases}
\end{align}

This implies that the sequence $(k^{(g_2)}_j)$ (resp. $(k^{(h_2)}_j)$) is obtained from $(k^{(g_1)}_i)$ (resp. $(k^{(h_1)}_i)$) by removing some elements, increasing the remaining elements by $1$, and appending the value $0$ to its end. Then, for any $g_2' \in C_{g_2}$, the sequences $W_{g_2, g_2'}$, $W_{h_2, h_2'}$ are obtained respectively from $W_{g_1, g_1'}$, $W_{h_1, h_1'}$ by removing the corresponding elements and appending respectively the values $T^A(g_2')$, $T^A(h_2')$; these two appended values must be the same by \eqref{eq-sm-B-dzieci-zachowane-A-t2}. (Note that increasing the values in the sequence $(k^{(g_2)}_j)$ by~$1$ translates to making no change in the sequence $W_{g_2, g_2'}$ due to the equality $g_2'^{\Uparrow k+1} = g_1'^{\Uparrow k}$)

Therefore, to finish the proof (i.e. to show \eqref{eq-sm-B-dzieci-cel}) it is sufficient, by \eqref{eq-sm-B-dzieci-znane-war-t1}, to check that both sequences $(k^{(g_1)}_i)$ and $(k^{(h_1)}_i)$ are subject to removing elements exactly at the same positions.

\textbf{4. }From \eqref{eq-sm-B-dzieci-nss} we know that the value $k^{(g_1)}_i + 1$ appears in the sequence $(k^{(g_2)}_j)$ if and only if there exist $g_2', g_2'' \in C_{g_2}$ such that (with the above notations ):
\begin{align} 
\label{eq-sm-B-dzieci-ii}
T^A(g_2') = T^A(g_2''), \qquad i = i^{(g_1)}_{g_1', g_1''}.
\end{align}
Then, from Remark \ref{uwaga-sm-indeksy-z-tabelki} and \eqref{eq-sm-B-dzieci-znane-war-t1} used for the pairs $(g_1', h_1')$ and~$(g_1'', h_1'')$, we obtain that $i^{(h_1)}_{h_1', h_1''} = i^{(g_1)}_{g_1', g_1''} = i$, while from \eqref{eq-sm-B-dzieci-ii} and~\eqref{eq-sm-B-dzieci-zachowane-A-t2} we have $T^A(h_2') = T^A(h_2'')$, so analogously we prove that the value $k^{(h_1)}_i + 1$ appears in the sequence $(k^{(h_2)}_j)$. The proof of the opposite implication --- if $k^{(h_1)}_i + 1$ appears in~$(k^{(h_2)}_j)$, then $k^{(g_1)}_i + 1$ must appear in~$(k^{(g_2)}_j)$ --- is analogous.

The obtained equivalence finishes the proof.
\end{proof}

\begin{lem}
\label{lem-sm-kuzyni}
If $N$ is sufficiently large, then, for any $g, h \in G$ satisfying $|g| = |h|$ and $d(g, h) \leq 16\delta$, the condition $T^B(g) = T^B(h)$ holds only if $g=h$.
\end{lem}

\begin{proof}
Suppose that $g \neq h$ while $T^B(g) = T^B(h)$. Then, by Remark \ref{uwaga-sm-B-wyznacza-A}, $T^A(g) = T^A(h)$, and so $T^b_N(g) = T^b_N(h)$. Denote $\gamma = g^{-1}h$ and
assume that $N$ is greater than the constant $N_r$ given by Lemma \ref{fakt-kuzyni-lub-torsje} for $r = 16\delta$. Then, the lemma implies that  $\gamma$ is a torsion element, i.e. (using the notation of \eqref{eq-sm-zbiory-torsji}) $\gamma \in Tor$, and so $|\gamma^i| \leq R$ for all $i \in \mathbb{Z}$. Then, the set $A = \{ g \gamma^i \,|\, i \in \mathbb{Z} \}$ has diameter not greater then $R$, so any two of its elements are cousins. Moreover, $A$ contains $g$ and $h$.

We obtain that all elements of the set $K = \{ k_{g', g''} \,|\, g', g'' \in A, \, g' \neq g'' \}$ appear in the sequence $k^{(g)}$ as well as in $k^{(h)}$; moreover, from Remark \ref{uwaga-sm-indeksy-z-tabelki} we obtain that they appear in $k^{(g)}$ exactly at the following set of positions:
\[ I_1 = \Big\{ \max \big\{ j \,\big|\, (W)_j \neq (W')_j \big\} \ \Big| \ (\gamma^i, W), \, (\gamma^{i'}, W') \in T^B(g), \ i, i' \in \mathbb{Z}, \, \gamma^i \neq \gamma^{i'} \Big\}, \]
where by $(W)_j$ we denoted the $j$-th element of the sequence $W$. Similarly, the elements of $K$ appear in $k^{(h)}$ exactly on the positions from the set $I_2$ defined analogously (with $g$ replaced by $h$). However, by assumption we have $T^B(g) = T^B(h)$, and so $I_1 = I_2$. 

Now, consider the pair $(\gamma, W_{g, h})$ which appears in $T^B(g)$; from the equality $T^B(g) = T^B(h)$ we obtain that
\[ (\gamma, W_{g, h}) = (h^{-1} h', W_{h, h'}) \qquad \textrm{ for some } h' \in C_h. \]
Since the sets of indexes $I_1, I_2$ are equal, from $W_{g, h} = W_{h, h'}$ we obtain in particular that
\[ T^A(h^{\Uparrow k}) = T^A({h'}^{\Uparrow k}) \qquad \textrm{ for every } k \in K. \]
However, from $h^{-1} h' = \gamma$ it follows that $h' \in A \setminus \{ h \}$ and then $k_{h, h'} \in K$, which contradicts the above equality. 
\end{proof}

\subsection{Stronger simplex types}
\label{sec-sk-opis}

\begin{tw}
\label{tw-sk-opis}
Let $(\mathcal{U}_n)$ be a quasi-$G$-invariant system of covers of a~space $X$, equipped with a type function stronger than $T^B$ and a neighbourhood constant $D$ not greater than $16\delta$. Let $(K_n, f_n)$ be the inverse system obtained from applying Theorem \ref{tw-kompakt-ogolnie} for $(\mathcal{U}_n)$. Then, the simplex types in the system $(K_n, f_n)$ can be strengthened to ensure the distinct types property for this system, without losing Markov property.
\end{tw}

\begin{proof}[Organisation of the proof]
We will consider the inverse system with a new simplex type $T^{\Delta + A}$, defined in Definition \ref{def-sk-opis-typ-delta-a}. Then, Lemmas \ref{fakt-sk-opis-markow} and~\ref{fakt-sk-opis-bracia} will ensure that this system satisfies the conditions from Definitions \ref{def-kompakt-markowa} and~\ref{def-kompakt-wlasciwy}, respectively. 
\end{proof}

We now start the proof.

Let $T$ denote the type function associated with the system $(\mathcal{U}_n)$, and $T^\Delta$ --- the corresponding simplex type function in the system $(K_n)$, as defined in Section \ref{sec-markow-typy}. 

\begin{fakt}
\label{fakt-sk-opis-podsympleksy-rozne-typy}
For any simplex $s \in K_n$, all subsimplexes $s' \subseteq s$ have pairwise distinct $T^\Delta$-types.
\end{fakt}

\begin{proof}
Let $v = v_{U_x}$, $v' = v_{U_{x'}}$ be two distinct vertices joined by an edge in $K_n$. Then, by using the definition of the complex $K_n$, then the property \qhlink{d} and Proposition \ref{lem-sm-kuzyni}, we have:
\begin{align} 
\label{eq-sm-sk-opis-rozne-typy-wierzch}
U_x \cap U_{x'} \neq \emptyset \quad \Rightarrow \quad d(x, x') \leq D \leq 16\delta \quad \Rightarrow \quad T^B(x) \neq T^B(x') \quad \Rightarrow \quad T(x) \neq T(x'). 
\end{align}
Recall that, by Definition \ref{def-typ-sympleksu}, for any $s = [v_1, \ldots, v_k] \in K_n$ the value $T^\Delta(s)$ determines in particular the set of labels in the graph $G_s$ (defined in Definition \ref{def-graf-typu-sympleksu}). This set can be described by the formula:
\[ A_s = \big\{ T(x) \ \big|\ x \in G, \, |x| = n, \, v_{U_x} = v_i \textrm{ for some } 1 \leq i \leq k \big\} \]
However, it follows from \eqref{eq-sm-sk-opis-rozne-typy-wierzch} that if two simplexes ~$s', s'' \subseteq s$ differ in that some vertex~$v_{U_x}$ belongs only to~$s'$, then the value $T(x)$ belongs to $A_{s'} \setminus A_{s''}$, which means that $T^\Delta(s') \neq T^\Delta(s'')$.
\end{proof}

\begin{fakt}
\label{fakt-sk-opis-przesuniecie-jedyne}
Let $s \in K_n$, $s' \in K_{n'}$ have the same types. Then, there exists a unique $\gamma \in G$ such that $s' = \gamma \cdot s$.
\end{fakt}

\begin{uwaga}
The claim of Lemma~\ref{fakt-sk-opis-przesuniecie-jedyne} is stronger than what was stated in Section \ref{sec-markow-typy} in that $\gamma$ should be unique. The stronger claim follows, as we will show below, from the additional assumption that the type function~$T$ is stronger than~$T^B$.
\end{uwaga}

\begin{proof}[Proof of Lemma~\ref{fakt-sk-opis-przesuniecie-jedyne}]
By Lemma \ref{fakt-przesuniecie-istnieje}, a desired element $\gamma$ exists, it remains to check its uniqueness. Let $s' = \gamma \cdot s = \gamma' \cdot s$; then, by Lemma \ref{fakt-przesuniecie-skladane}, we have~$s = (\gamma^{-1} \gamma') \cdot s$. By Definition \ref{def-przesuniecie-sympleksu}, this means that if $v_{U_x}$ is a vertex in $s$, then setting $x' = \gamma^{-1} \gamma' x$ we have $T(x') = T(x)$ and moreover $v_{U_{x'}}$ is also a vertex in $s$. By reusing the argument from \eqref{eq-sm-sk-opis-rozne-typy-wierzch}, we conclude that $x = x'$ holds, and so $\gamma = \gamma'$.
\end{proof}

We will now define a strengthening of the $T^\Delta$-type, in a way quite analogous to the definition of the $T^A$-type for simplexes. (Some differences will occur in the proofs, and also in Definition ~\ref{def-sk-opis-typ-delta-a}).

\begin{df}[cf.~Definition~\ref{def-sm-nic-sympleksow}]
For $n > 0$, we call the \textit{prioritised parent} of a simplex $s \in K_n$ the minimal simplex in $K_{n-1}$ containing~$f_n(s)$ (which will be denoted by $s^\uparrow$). A Simplex $s$ is a \textit{prioritised child} of a simplex $s'$ if $s' = s^\uparrow$.
\end{df}

\begin{fakt}[cf.~Lemma~\ref{fakt-sm-kulowy-wyznacza-dzieci}, and also Lemma~\ref{fakt-przesuniecie-istnieje}]
\label{fakt-sk-opis-przesuwanie-dzieci}
Let $s \in K_n$, $s' \in K_{n'}$ and $\gamma \in G$ be such that $s' = \gamma \cdot s$ (in the sense of Definition \ref{def-przesuniecie-sympleksu}). Then, the translation by $\gamma$ gives a bijection between prioritised children of $s$ and prioritised children of $s'$.
\end{fakt}

\begin{proof}
This is an easy corollary of Proposition \ref{lem-przesuwanie-dzieci-sympleksow} (and Lemma \ref{fakt-przesuniecie-skladane}). The proposition ensures that the translation by $\gamma$ maps simplexes contained in $f_n^{-1}(s)$ to simplexes contained in $f_{n'}^{-1}(s')$. Moreover, if $\sigma \subseteq s$ and $\gamma \cdot \sigma$ is not a prioritised child of $s'$, then we have $\gamma \cdot \sigma \subseteq f_{n'}^{-1}(s'')$ for some $s'' \subsetneq s'$ and then $\sigma \subseteq f_n^{-1}(\gamma^{-1} \cdot s'')$, which means that $\sigma$ is not a prioritised child of $s$. The reasoning in the opposite direction is analogous because $s = \gamma^{-1} \cdot s'$.
\end{proof}

Let $\mathcal{T}$ be the set of values of $T^\Delta$. For every $\tau \in \mathcal{T}$, choose some \textit{representative} of this type $s_\tau \in K_{n_\tau}$. For any simplex $s \in K_n$ of type~$\tau$, let $\gamma_\tau \in G$ be the unique element such that $\gamma_s \cdot s = s_\tau$ (the uniqueness follows from Lemma \ref{fakt-sk-opis-przesuniecie-jedyne}).

\begin{df}[cf. Definitions~\ref{def-sm-numer-dzieciecy} and~\ref{def-sm-typ-A}]
\label{def-sk-opis-typ-delta-a}
The $T^{\Delta + A}$-type of a simplex $s \in K_n$ is defined by the formula
\begin{align} 
\label{eq-sk-opis-typ-delta-a}
T^{\Delta + A}(s) = \big( T^\Delta(s), \ T^\Delta(s^\uparrow), \ \gamma_{s^\uparrow} \cdot s \big).  
\end{align}
\end{df}

Note that, by Proposition \ref{lem-przesuwanie-dzieci-sympleksow}, the translation $\gamma_{s^\uparrow} \cdot s$ (which plays an analogous role to the descendant number) exists and it is one of the simplexes in the pre-image of $f_{n_\tau}^{-1}(s_\tau)$, where $\tau = T^\Delta(s^\uparrow)$.  This ensures the correctness of the above definition, as well as finiteness of the resulting type $T^{\Delta + A}$.

Also, note that the component $T^\Delta(s^\uparrow)$ in the formula~\eqref{eq-sk-opis-typ-delta-a} has no equivalent in Definition \ref{def-sm-typ-A}. It will be used in the proof of Lemma \ref{fakt-sk-opis-bracia}.

\begin{fakt}
\label{fakt-sk-opis-markow}
The inverse system $(K_n, f_n)$, equipped with the simplex type function $T^{\Delta + A}$, satisfies the conditions from Definition \ref{def-kompakt-markowa}.
\end{fakt}

\begin{proof}
Since the system $(K_n)$ has the Markov property when equipped a type function $T^\Delta$, weaker than $T^{\Delta + A}$, it suffices to check the condition (iii). To achieve this --- by Proposition \ref{lem-przesuwanie-dzieci-sympleksow} --- we need only to check that, if for some $s \in K_n$, $s' \in K_{n'}$, $\gamma \in G$ the equality $s' = \gamma \cdot s$ holds and $s$, $s'$ have the same $T^{\Delta + A}$-type, then the translation by $\gamma$ preserves the values of $T^{\Delta + A}$ for all simplexes contained in $f_n^{-1}(s)$.

Let then $\sigma$ be any such simplex. Then, $\sigma^\uparrow \subseteq s$, so from the claim of Proposition \ref{lem-przesuwanie-dzieci-sympleksow} (more precisely: from the fact that the translation by $\gamma$ preserves $T^\Delta$ and that it commutes with $f_n$ and~$f_{n'}$) we deduce that:
\begin{align} 
\label{eq-sk-opis-delta-typy-zgodne}
T^\Delta \big( (\gamma \cdot \sigma)^\uparrow \big) = T^\Delta \big( \gamma \cdot (\sigma^\uparrow) \big) = T^\Delta( \sigma^\uparrow ). 
\end{align}
Moreover, it is clear that $T^\Delta(\sigma) = T^\Delta(\gamma \cdot \sigma)$, so it only remains to verify the equality of the last coordinates in the types $T^{\Delta + A}(\sigma)$, $T^{\Delta + A}(\gamma \cdot \sigma)$.

Denote by $\tau$ the formula~\eqref{eq-sk-opis-delta-typy-zgodne}. Then, using Lemma \ref{fakt-przesuniecie-skladane}, we have
\[ \gamma_{\sigma^\uparrow} \cdot \sigma^\uparrow = s_\tau = \gamma_{(\gamma \cdot \sigma)^\uparrow} \cdot (\gamma \cdot \sigma)^\uparrow = \gamma_{(\gamma \cdot \sigma)^\uparrow} \cdot \big( \gamma \cdot (\sigma)^\uparrow \big) = ( \gamma_{(\gamma \cdot \sigma)^\uparrow} \, \gamma ) \cdot \sigma^\uparrow, \]
so, by Lemma \ref{fakt-sk-opis-przesuniecie-jedyne}, we obtain $\gamma_{\sigma^\uparrow} = \gamma_{(\gamma \cdot \sigma)^\uparrow} \, \gamma$. This in turn implies that:
\[ \gamma_{\sigma^\uparrow} \cdot \sigma = ( \gamma_{(\gamma \cdot \sigma)^\uparrow} \, \gamma ) \cdot \sigma = \gamma_{(\gamma \cdot \sigma)^\uparrow} \cdot (\gamma \cdot \sigma), \]
which finishes the proof.
\end{proof}

\begin{fakt}
\label{fakt-sk-opis-bracia}
For any simplex $s \in K_n$, all simplexes in the pre-image $f_n^{-1}(s)$ have pairwise distinct $T^{\Delta + A}$-types.
\end{fakt}

\begin{proof}
Let $\sigma, \sigma' \in f_n^{-1}(s)$ satisfy $T^{\Delta + A}$. Then in particular $T^\Delta(\sigma^\uparrow) = T^\Delta({\sigma'}^\uparrow)$, and since $\sigma^\uparrow$, ${\sigma'}^\uparrow$ are subsimplexes of $s$, from Lemma \ref{fakt-sk-opis-podsympleksy-rozne-typy} we obtain that they are equal. Then, the equality of the third coordinates in types $T^{\Delta + A}(\sigma)$, $T^{\Delta + A}(\sigma')$ implies that $\sigma = \sigma'$.
\end{proof}

\section{Markov systems with limited dimension}

\label{sec-wymd}

In this section, we assume that $\dim \partial G \leq k < \infty$, and we discuss how to adjust the construction of a Markov system to ensure that all the complexes in the inverse system also have dimension $\leq k$. Since $\partial G$ is a compact metric space, its dimension can be understood as the covering dimension, or equivalently as the small inductive dimension (cf.~\cite[Theorem~1.7.7]{E}). In the sequel, we denote the space $\partial G$ by~$X$, and the symbol $\partial$ will always mean the topological frontier taken in~$X$ or in some its subset.

The main result of this section is given below. 

\begin{lem}
\label{lem-wym}
Let $k \geq 0$ and let $G$ by a hyperbolic group such that $\dim \partial G \leq k$. Then, there exists a quasi-$G$-invariant system of covers of~$\partial G$ of rank~$\leq k + 1$.
\end{lem}

Since the rank of a cover determines the dimension of its nerve, this result will indeed allow to limit the dimension of the complexes in the Markov system for~$\partial G$ (see Section~\ref{sec-wymd-podsum}).

\begin{uwaga}
Although the proof of Proposition~\ref{lem-wym} given below will involve many technical details, let us underline that --- in its basic sketch --- it resembles an elementary result from dimension theory stating that every open cover $\mathcal{U} = \{ U_i \}_{i = 1}^n$ of a compact metric space~$X$ of dimension~$k$ contains an open subcover of rank~$\leq k + 1$. Below, we present the main steps of a proof of this fact, and pointing out the analogies between these steps and the contents of the rest of this section.
\begin{itemize}

\item[(i)] We proceed by induction on $k$. For convenience, we work with a slightly stronger inductive claim: the cover~$\{ U_i \}$ contains an open subcover~$\{ V_j \}$ such that the closures $\overline{V_j}$ form a family of rank~$\leq k + 1$.

(For the proof of Proposition~\ref{lem-wym}, the inductive reasoning is sketched in more detail in Proposition~\ref{lem-wym-cala-historia}).

\item[(ii)] Using the auxiliary Theorem~\ref{tw-wym-przedzialek} stated below, we choose in each $U_i \in \mathcal{U}$ an open subset $U_i'$ with frontier of dimension $\leq k-1$ so that the sets~$U_i'$ still form a cover of~$X$.

(Similarly we will define the sets~$D_x$ in the proof of Proposition~\ref{lem-wym-bzdziagwy}).

\item[(iii)] We define the sets $U_i''$ by the condition:
\[ x \in U_i'' \qquad \qquad \Longleftrightarrow \qquad \qquad x \in U_i' \quad \textrm{ and } \quad x \notin U'_j \quad \textrm{ for } j < i. \]

(Analogously we will define the sets~$E_x$ in the proof of Proposition~\ref{lem-wym-bzdziagwy})

\item[(iv)] The space~$X$ is now covered by the interiors of the sets $U_i''$ (which are pairwise disjoint) together with the set~$\widetilde{X} = \bigcup_i \partial U_i''$, which is a closed subset of~$X$ of dimension~$\leq k - 1$. In~$\widetilde{X}$, we consider an open cover formed by the sets $\widetilde{U}_i = U_i \cap \widetilde{X}$. By the inductive hypothesis, this cover must contain an open subcover formed by some sets~$\widetilde{V}_j$ ($1 \leq j \leq m$) whose closures form a family of rank~$\leq k$. Also, we may require that $\widetilde{V}_j$ is open in~$\widetilde{X}$ --- but not necessarily in~$X$.

(The sets $\partial U_i''$, $\innt U_i''$ correspond to the sets $F_x$, $G_x$ appearing in the formulation of Proposition~\ref{lem-wym-bzdziagwy}).

\item[(v)] Let $\varepsilon > 0$ be the least distance between any \textit{disjoint pair} of closures $\overline{\widetilde{V}_{j_1}}$, $\overline{\widetilde{V}_{j_2}}$. For $1 \leq j \leq m$, we define $V_j$ as the $\tfrac{\varepsilon}{4}$-neighbourhood (in~$X$) of~$\widetilde{V}_j$. Then, it is easy to verify that the sets $V_j$ are open in~$X$ and cover~$\widetilde{X}$, and moreover the rank of the family $\{\overline{ V_j }\}$ does not exceed the rank of $\{\overline{ \widetilde{V}_j }\}$ which is $\leq k$.

Let $U_i'''$ denote $\innt U_i''$ minus the closed $\tfrac{\varepsilon}{8}$-neighbourhood of $\widetilde{X}$.
Then, the family
\[ \mathcal{V} = \{ V_j \}_{j = 1}^m \cup \{ U_i''' \}_{i = 1}^n \]
is an open cover of~$X$. Moreover, the rank of the family of closures of all elements of $\mathcal{V}$ is at most the sum of ranks of the families $\{ \overline{ V_j } \}$ and $\{ \overline{ U_i''' } \}$, which are respectively $k$ and $1$ (the latter because for every $i \neq i'$ we have $\overline{U_i'''} \cap \overline{U_{i'}'''} \subseteq \overline{U_i''} \cap \overline{U_{i'}''} \subseteq \widetilde{X}$ which is disjoint from both $\overline{U_i'''}$ and~$\overline{U_{i'}'''}$). Hence, $\mathcal{V}$ satisfies all the desired conditions.

(In our proof of Proposition~\ref{lem-wym}, the construction of appropriate neighbourhoods takes place in Proposition~\ref{lem-wym-kolnierzyki}, and the other of the above steps have their counterparts in the proof of Proposition~\ref{lem-wym-cala-historia}).
\end{itemize}
In comparison to the above reasoning, the main difficulty in proving Proposition~\ref{lem-wym} lies in ensuring quasi-$G$-invariance of the adjusted covers, which we need for preserving the Markov property for the system of their nerves (using Theorem~\ref{tw-kompakt-ogolnie}). For this, instead of defining each of the sets $U_i'$ independently, we will first choose a finite number of \textit{model sets}, one for each possible value of type in~$G$, and translate these model sets using Proposition~\ref{lem-potomkowie-dla-kulowych}. The inductive argument will now require special care for preserving quasi-$G$-invariance; nevertheless; the main idea remains unchanged.
\end{uwaga}

We will use the following auxiliary result from dimension theory:

\begin{tw}[{\cite[Theorem~1.5.12]{E}}]
\label{tw-wym-przedzialek}
Let~$Y$ be a separable metric space of dimension~$k$ and $A, B$ be disjoint closed subsets of~$Y$. Then, there exist open subsets $\widetilde{A}, \widetilde{B} \subseteq Y$ such that 
\[ A \subseteq \widetilde{A}, \qquad B \subseteq \widetilde{B}, \qquad \widetilde{A} \cap \widetilde{B} = \emptyset \qquad \textrm{ and } \qquad \dim \big(Y \setminus (\widetilde{A} \cup \widetilde{B})\big) \leq k - 1. \]
\end{tw}

\subsection{$\theta$-weakly invariant subsystems}

\begin{ozn}
For any two families $\mathcal{C} = \{ C_x \}_{x \in G}$, $\mathcal{D} = \{ D_x \}_{x \in G}$, we denote:
\[ \mathcal{C} \sqcup \mathcal{D} = \{ C_x \cup D_x \}_{x \in G}. \]
\end{ozn}

\begin{df}
A system $\mathcal{C} = \{ C_x \}_{x \in G}$ of subsets of~$X$ will be called:
\begin{itemize}
 \item \textit{of dimension $\leq k$} if $|\mathcal{C}|_n$ is of dimension $\leq k$ for all $n \geq 0$;
 \item \textit{of rank $\leq k$} if, for every $n \geq 0$, the family $\mathcal{C}_n$ is of rank $\leq k$ (i.e. if the intersection of any $k+1$ pairwise distinct members of $\mathcal{C}_n$ must be empty); in particular, \textit{disjoint} if it is of rank~$\leq 1$.
\end{itemize}
\end{df}

\begin{fakt}
\label{fakt-wym-suma-pokryc}
If two systems of subsets $\mathcal{C}$, $\mathcal{D}$ are correspondingly of rank~$\leq a$ and~$\leq b$, then $\mathcal{C} \sqcup \mathcal{D}$ is of rank $\leq a+b$. \qed
\end{fakt}

\begin{df}
Let $X$ be a topological space and $A, B \subseteq X$. We will say that~$A$ is a \textit{separated} subset of~$B$ (denotation: $A \Subset B$) if $\overline{A} \subseteq \innt B$.
\end{df}

\begin{df}
Let $\mathcal{C} = \{ C_x \}$, $\mathcal{D} = \{ D_x \}$ be two quasi-$G$-invariant systems of subsets of~$X$. We will say that:
\begin{itemize}
 \item $\mathcal{C}$ is an \textit{(open, closed) subsystem} in~$\mathcal{D}$ if, for every $n \geq 0$ and $C_x \in \mathcal{C}_n$, $C_x$ is an (open, closed) subset in~$|\mathcal{D}|_n$ \\ (recall from \ref{ozn-suma-rodz} that $|\mathcal{D}|_n$ denotes $\bigcup_{C \in \{C_x \,|\, x \in G, \, |x|=n\}} C$ )
 \item $\mathcal{C}$ is a \textit{semi-closed subsystem} in~$\mathcal{D}$ if $|\mathcal{C}|_n$ is a closed subset in~$|\mathcal{D}|_n$ for $n \geq 0$;
 \item $\mathcal{C}$ \textit{covers} $\mathcal{D}$ if $|\mathcal{C}|_n \supseteq |\mathcal{D}|_n$ for $n \geq 0$.
\end{itemize}
A system $\mathcal{C}$ will be called \textit{semi-closed} if $|\mathcal{C}|_n$ is a closed subset in~$X$ for $n \geq 0$.
\end{df}

\begin{df}
For any integer $\theta \geq 0$, we define the type function $T^B_\theta$ in~$G$ as the extension (in the sense of Definition~\ref{def-sm-typ-plus}) of the type function $T^B$ (defined in Definition~\ref{def-sm-typ-B}) by~$r = \theta \cdot 12\delta$:
\[ T^B_\theta = (T^B)^{+ \theta \cdot 12\delta} \]
\end{df}

\begin{fakt}
\label{fakt-wym-poglebianie-B+}
Let $g, x, y \in G$ and~$\theta \geq 0$, $k > 0$ satisfy
\[ y \in xT^c(x), \qquad T^B_\theta(x) = T^B_\theta(gx), \qquad |y| = |x| + kL, \]
where $L$ denotes the constant from Section~\ref{sec-sm} (defined in Section~\ref{sec-sm-nici}).
Then:
\[ T^B_{\theta + 1}(y) = T^B_{\theta + 1}(gy), \qquad |gy| = |gx| + kL. \]
\end{fakt}

\begin{proof}
It suffices to prove the claim for $k = 1$; for greater values of~$k$, it will then easily follow by induction.

Let $z \in yP_{(\theta + 1) \cdot 12\delta}(y)$. Denote $w = z^\Uparrow$ (see Definition~\ref{def-sm-p-wnuk}). Since $L \geq 14\delta$, Lemma~\ref{fakt-geodezyjne-pozostaja-bliskie} implies that
\[ d(x, w) \leq \max \big( (\theta+1) \cdot 12\delta + 16\delta - 2L, \ 8\delta \big) \leq \theta \cdot 12\delta, \]
so $w \in xP_{12\delta}(x)$, and then $T^B(w) = T^B(gw)$ and $|gx| = |gw|$. Therefore, by Proposition~\ref{lem-sm-B-dzieci} we know that $gw = (gz)^\Uparrow$ and $T^B(z) = T^B(gz)$. The first of these equalities implies in particular that $gz$ is a descendant of $gw$, that is,
\[ |gz| = |gw| + d(gz, gw) = |gx| + d(z, w) = |gx| + L. \]
in particular, setting $z = y$ we obtain that $|gy| = |y| + L$. Considering again an arbitrary~$z$, we deduce that $|gz| = |gy|$, so $P_{(\theta+1) \cdot 12\delta}(y) \subseteq P_{(\theta+1) \cdot 12\delta}(gy)$; the opposite inclusion can be proved analogously (by exchanging the roles between $g$ and $g^{-1}$). In this situation, the equality $T^B(z) = T^B(gz)$ for an arbitrary~$z$ implies that $T^B_{\theta+1}(y) = T^B_{\theta+1}(gy)$.
\end{proof}

\begin{df}
A family of subsets $\mathcal{C} = \{ C_x \}$ will be called a \textit{$\theta$-weakly invariant system} if:
\begin{itemize}
 \item for every $x \in G$, $C_x$ is a subset of the set $S_x$ defined in Section~\ref{sec-konstr-pokrycia};
 \item the family $\mathcal{C}$, together with the type function $T^B_\theta$, satisfies the condition~\qhlink{f1} of Definition~\ref{def-quasi-niezm}.
\end{itemize}
\end{df}

Although we only require the condition~\qhlink{f1} to hold, we will show in Proposition~\ref{lem-wym-slabe-sa-qi} that this suffices to force the other conditions of Definition~\ref{def-quasi-niezm} to hold under some appropriate assumptions.

Let us observe that the system $\mathcal{S}$ described in Section~\ref{sec-konstr-pokrycia} is $0$-weakly (and then also $\theta$-weakly for $\theta \geq 0$) invariant. This is because the values of $T^B_\theta$ determine uniquely the values of $T^b_N$ (by Remark~\ref{uwaga-sm-B-wyznacza-A} and Definition~\ref{def-sm-typ-A}), while the system $\mathcal{S}$ equipped with the latter type function is quasi-$G$-invariant by Corollary~\ref{wn-spanstary-quasi-niezm}. 

\begin{fakt}
\label{fakt-wym-typ-wyznacza-sasiadow}
Let $\mathcal{C} = \{ C_x \}_{x \in G}$ be a $\Theta$-weakly invariant system for some $\Theta \geq 0$. Let $n \geq 0$, $\theta \geq 1$ and $x, y \in G$ be of length~$n$, and suppose that $C_x \cap C_y \neq \emptyset$. Then:
\begin{itemize}
 \item[\textbf{(a)}] $T^B_\theta(x) \neq T^B_\theta(y)$;
 \item[\textbf{(b)}] For every $g \in G$, the equality $T^B_\theta(x) = T^B_\theta(gx)$ implies that $T^B_{\theta-1}(y) = T^B_{\theta-1}(gy)$ and $|gx| = |gy|$.
\end{itemize}
\end{fakt}

\begin{proof}
Since $\emptyset \neq C_x \cap C_y \subseteq S_x \cap S_y$, it follows from Lemma~\ref{fakt-sasiedzi-blisko} that $d(x, y) \leq 12\delta$. Then, Proposition~\ref{lem-sm-kuzyni} implies that $T^B(x) \neq T^B(y)$, so $T^B_\theta(x) \neq T^B_\theta(y)$. Moreover, since $|x| = |y|$, we have $x^{-1} y \in P_{\theta \cdot 12\delta}(x) = P_{\theta \cdot 12\delta}(gx)$, and so $|gx| = |gy|$. 
Note that the triangle inequality gives:
\[ P_{(\theta-1) \cdot 12\delta}(y) = \big( y^{-1} x \, P_{\theta \cdot 12\delta}(x) \big) \cap B \big( e, (\theta - 1) \cdot 12\delta \big) \]
and analogously for~$gx, gy$, which shows that $P_{(\theta-1) \cdot 12\delta}(y) = P_{(\theta-1)\cdot 12\delta}(gy)$. Moreover, for any $z \in yP_{(\theta-1) \cdot 12\delta}(y)$ the above equality implies that $x^{-1}z \in P_{\theta \cdot 12\delta}(x) = P_{\theta \cdot 12\delta}(gx)$, and so $T^B(z) = T^B(gz)$. Since~$z$ was arbitrary, we deduce that $T^B_{\theta-1}(y) = T^B_{\theta-1}(gy)$.
\end{proof}

\begin{lem}
\label{lem-wym-slabe-sa-qi}
Let $\theta \geq 0$ and $\mathcal{C}$ be a $\theta$-weakly invariant system of open covers of~$X$. Then, $\mathcal{C}$ is quasi-$G$-invariant when equipped with the type function~$T^B_{\theta+1}$.
\end{lem}

\begin{proof}
Let $\mathcal{C} = \{ C_x \}_{x \in G}$. The conditions~\qhlink{c}, \qhlink{d} for $\mathcal{C}$ follow directly from the same conditions for~$\mathcal{S}$. Also, \qhlink{f2} can be easily translated: whenever we have
\[ T^B_{\theta+1}(x) = T^B_{\theta+1}(gx), \qquad |x| = |y|, \qquad C_x \cap C_y \neq \emptyset, \]
then by Lemma~\ref{fakt-wym-typ-wyznacza-sasiadow}b it follows that $|gx| = |gy|$ and $T^B_{\theta}(y) = T^B_{\theta}(gy)$, so --- by the property~\qhlink{f1} for~$\mathcal{C}$ --- we have $C_{gy} = g \cdot C_y$.

We will now verify the property \qhlink{f3}. 
Let $L$ denote the constant coming from Lemma~\ref{fakt-wym-poglebianie-B+}.
Suppose that
\[ T^B_{\theta+1}(x) = T^B_{\theta+1}(gx), \qquad |y| = |x| + L, \qquad \emptyset \neq C_y \subseteq C_x. \]
Let $\alpha$ be a geodesic joining~$e$ with~$y$ and let $z = \alpha(|x|)$. Then, by Lemma~\ref{fakt-wlasnosc-gwiazdy-bez-gwiazdy}c we have $C_y \subseteq S_y \subseteq S_z$; on the other hand, $C_y \subseteq S_x$, so $S_x \cap S_z \neq \emptyset$, and then by Lemma~\ref{fakt-wym-typ-wyznacza-sasiadow} we have
\[ T^B_\theta(z) = T^B_\theta(gz), \qquad |gx| = |gz|. \]
Since $y \in zT^c(z)$ and $|y| = |x| + L = |z| + L$, using Lemma~\ref{fakt-wym-poglebianie-B+} and then the property~\qhlink{f1} for~$\mathcal{C}$ we obtain that
\[ T^B_{\theta+1}(y) = T^B_{\theta+1}(gy), \qquad |gy| = |gz| + L = |gx| + L, \qquad C_{gy} = g \cdot C_y. \]
By Remark~\ref{uwaga-quasi-niezm-jeden-skok}, this means that $\mathcal{C}$ satisfies~\qhlink{f3} for the jump constant~$L$.

It remains to verify the property \qhlink{e}. Let $\mathcal{T}$ be the set of all possible values of~$T^B_{\theta + 1}$. Choose $N > 0$ such that, for every $\tau \in \mathcal{T}$, there is $x \in G$ of length less than~$N$ such that $T^B_{\theta + 1}(x) = \tau$. Let $\varepsilon > 0$ be the minimum of all Lebesgue numbers for the covers $\mathcal{C}_1, \ldots, \mathcal{C}_N$, and choose $J' > N$ so that $\max_{S \in \mathcal{S}_n} \diam S_n < \varepsilon$ for every $n \geq J'$. We will show that the system~$\mathcal{C}$ together with $J'$ satisfies~\qhlink{e}. By Remark~\ref{uwaga-quasi-niezm-jeden-skok}, it suffices to verify this in the case when $k = 1$.

Let $x \in G$ satisfy $|x| \geq J'$. By Lemma~\ref{fakt-wlasnosc-gwiazdy-bez-gwiazdy}, we have $S_x \subseteq S_y$ for some $y \in G$ of length~$|x| - J'$. Let $y' \in G$ be an element such that $T^B_{\theta+1}(y') = T^B_{\theta+1}(y)$ and $|y'| < N$. Denote $\gamma = y'y^{-1}$ and $x' = \gamma x$; then, by~\qhlink{f3}, we have
\[ S_{x'} = \gamma \cdot S_x, \qquad T^B_{\theta+1}(x') = T^B_{\theta+1}(x), \qquad |x'| \, = \, |y'| + |x| - |y| \, = \, |y'| + J'. \]
The last equality implies that $\diam S_{x'} < \varepsilon$, which means (since $|y'| < N$) that $S_{x'}$ must be contained in~$C_{z'}$ for some $z' \in G$ such that $|z'| = |y'|$. Denote $z = \gamma^{-1} z'$. Then, using the property~\qhlink{f2} and then Lemma~\ref{fakt-wym-typ-wyznacza-sasiadow}b, we obtain:
\[ S_z = \gamma^{-1} \cdot S_{z'}, \qquad T^B_\theta(z) = T^B_\theta(z'), \qquad |z| = |y| = |x| - J', \]
and so, since $\mathcal{C}$ is $\theta$-weakly invariant, it follows that:
\[ C_x \subseteq S_x = \gamma^{-1} \cdot S_{x'} \subseteq \gamma^{-1} \cdot C_{z'} = C_z. \]

Altogether, we obtain that $\mathcal{C}$ is quasi-$G$-invariant with $\gcd(L, J')$ as the jump constant.
\end{proof}

\begin{fakt}
\label{fakt-wym-homeo-na-calej-kuli}
Let $\theta \geq 0$, $\mathcal{C}$ be a $\theta$-weakly invariant system, and $x, y \in G$ with $T^B_{\theta + 1}(x) = T^B_{\theta + 1}(y)$. Then, we have
\[ |\mathcal{C}|_{|y|} \cap S_y = yx^{-1} \cdot \big( |\mathcal{C}|_{|x|} \cap S_x \big). \]
In particular, the left translation by $yx^{-1}$, when applied to separated subsets of~$S_x$, preserves the interiors and closures taken in $|\mathcal{C}|_{|x|}$ and respectively $|\mathcal{C}|_{|y|}$.
\end{fakt}

\begin{proof}
Denote $n = |x|$, $m = |y|$ and $\gamma = yx^{-1}$.
By symmetry, it suffices to show one inclusion. Let~$p \in |\mathcal{C}|_n \cap S_x$; then~$p$ belongs to some $C_{x'} \in \mathcal{C}_n$. In particular, we have:
\[ p \ \in \ S_x \cap C_{x'} \ \subseteq \ S_x \cap S_{x'}. \]
Denote $y' = \gamma x'$.
Then, we have $T^B_\theta(y') = T^B_\theta(x')$ by Lemma~\ref{fakt-wym-typ-wyznacza-sasiadow}b, so~\qhlink{f1} implies that $C_{y'} = \gamma \cdot C_{x'}$, as well as $S_y = \gamma \cdot S_x$. Therefore,
\[ \gamma \cdot p \ \in \ \gamma \cdot (C_{x'} \cap S_x) \ = \ C_{y'} \cap S_y \ \subseteq \ |\mathcal{C}|_m \cap S_y. \qedhere \]
\end{proof}

\subsection{Disjoint, weakly invariant nearly-covers}
\label{sec-wym-wyscig}

Before proceeding with the construction, we will introduce notations and conventions used below.

In the proofs of Propositions~\ref{lem-wym-bzdziagwy} and \ref{lem-wym-kolnierzyki}, we will use the following notations, dependent on the value of a~parameter~$\theta$ (which is a~part of the input data in both propositions). Let $\tau_1, \ldots, \tau_K$ be an enumeration of all possible $T^B_{\theta+1}$-types. For simplicity, we identify the value $\tau_i$ with the natural number~$i$. For every~$1 \leq i \leq K$, we fix an arbitrary $x_i \in G$ such that $T^B_{\theta+1}(x_i) = i$, and we set $S_i = _S{x_i}$, $n_i = |x_i|$.
Similarly, let $1, \ldots, \widetilde{K}$ be an enumeration of all possible~$T^B_{\theta+2}$-types, and for every $1 \leq \widetilde{\imath} \leq \widetilde{K}$ let $\widetilde{x}_{\widetilde{\imath}} \in G$ be a~fixed element such that $T^B_{\theta+2}(\widetilde{x}_{\widetilde{\imath}}) = \widetilde{\imath}$. We denote also $M = \max_{{\widetilde{\imath}}=1}^{\widetilde{K}} |\widetilde{x}_{\widetilde{\imath}}|$.

In the remaining part of Section~\ref{sec-wymd}, we will usually consider sub-systems of a given semi-closed system in~$X$ (which is given the name~$\mathcal{C}$ in Propositions~\ref{lem-wym-bzdziagwy} and~\ref{lem-wym-kolnierzyki}), and more generally --- subsets of~$X$ known to be contained in~$|\mathcal{C}|_n$ for some~$n$ (known from the context). Unless explicitly stated otherwise, the basic topological operators for such sets (closure, interior, frontier) will be performed within the space~$|\mathcal{C}|_n$ (for the appropriate value of~$n$). This will not influence closures, since~$|\mathcal{C}|_n$ is a closed subset of~$X$, but will matter for interiors and frontiers.

This subsection contains the proof of the following result.

\begin{lem}
\label{lem-wym-bzdziagwy}
Let $k \geq 0$, $\theta \geq 1$ and $\mathcal{C} = \{ C_x \}$ be a semi-closed, $\theta$-weakly invariant system of dimension $\leq k$ in~$X$. Then, there exist:
\begin{itemize}
 \item a disjoint, open, $(\theta + 2)$-weakly invariant subsystem $\mathcal{G} = \{ G_x \}$ in $\mathcal{C}$;
 \item a closed, $(\theta + 2)$-weakly invariant subsystem $\mathcal{F} = \{ F_x \}$ in $\mathcal{C}$ of dimension $\leq k - 1$
\end{itemize}
such that $\mathcal{F} \sqcup \mathcal{G}$ covers $\mathcal{C}$ and $\partial G_x \subseteq |\mathcal{F}|_n$ for every $n \geq 0$ and~$G_x \in \mathcal{G}_n$.
\end{lem}

\begin{proof}
Let $\varepsilon > 0$ be the minimum of all Lebesgue numbers for the covers $\mathcal{S}_1, \ldots, \mathcal{S}_M$. Define:
\[ I_x = \{ p \in S_x \,|\, B(p, \varepsilon) \subseteq S_x \} \qquad \qquad \textrm{ for } x \in G, \, |x| \leq M \] 
and
\begin{align*} 
G_i = \bigcup_{|x| \leq M, \, T^B_{\theta+1}(x) = i} x_i x^{-1} \cdot I_x, \qquad H_i = X \setminus S_i \qquad \qquad \textrm{ for } 1 \leq i \leq K.
\end{align*}

Fix some $1 \leq i \leq K$. We observe that, for every~$x \in G$, we have $d(I_x, X \setminus S_x) \geq \varepsilon$, which implies that~$G_i$ is a finite union of separated subsets of~$S_i$; then,
$\overline{G_i} \cap \overline{H_i} = \emptyset$ (in~$X$). Then, the intersections $\overline{G_i} \cap |\mathcal{C}|_{n_i}$, $\overline{H_i} \cap |\mathcal{C}|_{n_i}$ are disjoint closed subsets of~$|\mathcal{C}|_{n_i}$, so by Theorem~\ref{tw-wym-przedzialek} there are open subsets of~$|\mathcal{C}|_{n_i}$:
\[ \widetilde{G}_i \supseteq \overline{G_i} \cap |\mathcal{C}|_{n_i}, \qquad \widetilde{H}_i \supseteq \overline{H_i} \cap |\mathcal{C}|_{n_i}, \qquad \qquad (\textrm{ closures of } G_i, H_i \textrm{ taken in } X) \]
which cover $|\mathcal{C}|_{n_i}$ except for some subset of dimension~$\leq k - 1$ (which must contain $\partial \widetilde{G}_i$). 

For any $x \in G$, we denote $|x| = n$ and $T^B_{\theta+1}(x) = i$, and then we define:
\begin{align*}
 D_x & = x x_i^{-1} \cdot \widetilde{G}_i, \\
 E_x & = D_x \setminus \, \bigcup_{y \in G, \, |y| = n, \, T^B_{\theta+1}(y) < T^B_{\theta+1}(x)} D_y, \\
 F_x & = \partial E_x, \\
 G_x & = \intr E_x.
\end{align*}
Note that $\partial G_x = \overline{G_x} \setminus G_x \subseteq \overline{E_x} \setminus \intr E_x = \partial E_x = F_x$, as desired in the claim.

The remaining part of the proof consists of verifying the following claims (of which \textbf{(a)} and \textbf{(b)} are auxiliary):
\begin{itemize}[nolistsep]
 \item[\textbf{(a)}] $\mathcal{D} = \{ D_x \}_{x \in G}$ covers~$\mathcal{C}$;
 \item[\textbf{(b)}] $\mathcal{E} = \{ E_x \}_{x \in G}$ is disjoint and covers~$\mathcal{C}$;
 \item[\textbf{(c)}] $\mathcal{F} = \{ F_x \}_{x \in G}$ is of dimension~$\leq k - 1$;
 \item[\textbf{(d)}] $\mathcal{G} = \{ G_x \}_{x \in G}$ is disjoint;
 \item[\textbf{(e)}] $F_x \subseteq S_x$ for every~$x \in G$;
 \item[\textbf{(f)}] $\mathcal{F} \sqcup \mathcal{G}$ covers~$\mathcal{C}$;
 \item[\textbf{(g)}] $\mathcal{F}$ and~$\mathcal{G}$ are $(\theta+2)$-weakly invariant.
\end{itemize}

\textbf{(a)} 
To verify that $\mathcal{D}$ covers $\mathcal{C}$, choose an arbitrary $p \in |\mathcal{C}|_n$; then $p$ lies in some $S_x \in \mathcal{S}_n$. Denote 
\[ \widetilde{\imath} = T^B_{\theta + 2}(x), \qquad \gamma = x\widetilde{x}_{\widetilde{\imath}}^{-1}, \qquad p' = \gamma^{-1} \cdot p. \]
Then, $p' \in S_{\widetilde{x}_{\widetilde{\imath}}}$. By the definition of $\varepsilon$, we have $B(p', \varepsilon) \subseteq S_y$ for some $S_y \in \mathcal{S}_{|\widetilde{x}_{\widetilde{\imath}}|}$; then $p'$ lies in~$I_y$. Now, let
\[ 
j = T^B_{\theta+1}(y), \qquad \beta = yx_j^{-1}, \qquad p'' = \beta^{-1} \cdot p'. \]
By definition, $p''$ lies in $G_j$; moreover, by applying Lemma~\ref{fakt-wym-homeo-na-calej-kuli} twice we obtain that $p'' \in |\mathcal{C}|_{|x_j|}$. On the other hand, since $S_y$ intersects non-trivially with $S_{x_i}$ and $T^B_{\theta+2}(\gamma x_i) = T^B_{\theta+2}(x_i)$, by~\qhlink{f2} and Lemma~\ref{fakt-wym-typ-wyznacza-sasiadow}b we have:
\[ S_{\gamma \beta x_j} \ = \ S_{\gamma y} \ = \ \gamma \cdot S_y \ \ni \ \gamma \cdot p' \ = \ p, \qquad T^B_{\theta+1}(\gamma \beta x_j) = T^B_{\theta+1}(\gamma y) = T^B_{\theta+1}(y) = j, \]
which implies that 
\[ p \ = \ \gamma \beta \cdot p'' \ \in \ \gamma \beta \cdot (G_j \cap |\mathcal{C}|_{|x_j|}) \ \subseteq \ \gamma \beta \cdot \widetilde{G}_j \ \subseteq \ D_{\gamma \beta x_j}. \]

\textbf{(b)} 
Suppose that $p \in E_x \cap E_y$ for some $x \neq y$. Then, $S_x \cap S_y \neq \emptyset$, so by Lemma~\ref{fakt-wym-typ-wyznacza-sasiadow}a we have $T^B_{\theta+1}(x) \neq T^B_{\theta+1}(y)$; assume w.l.o.g. that $T^B_{\theta+1}(x)$ is the smaller one. Then, by definition, $p \in E_x \subseteq D_x$ cannot belong to $E_y$. This proves that $\mathcal{E}$ is disjoint.

Now, let~$p \in |\mathcal{C}|_n$ and let $x \in G$ of length~$n$ be chosen so that $p \in D_x$ and $T^B_{\theta+1}(x)$ is the lowest possible. Then, by definition, $p \in E_x$. This means that $\mathcal{E}$ covers $\mathcal{C}$.

\textbf{(c)} 
First, note that, if $x \in G$ and $T^B_{\theta+1}(x) = i$, then, by Lemma~\ref{fakt-wym-homeo-na-calej-kuli} and the definitions of~$D_x$ and~$\widetilde{G}_i$, we have:
\[ \dim \partial D_x = \dim \partial \big( xx_i^{-1} \cdot \widetilde{G}_i \big) = \dim \big( xx_i^{-1} \cdot \partial \widetilde{G}_i \big) = \dim \partial \widetilde{G}_i \leq k - 1. \]

Fix some $n \geq 0$. Recall that, for any subsets $Y, Z$ in any topological space, we have $\partial(Y \setminus Z) \subseteq \partial Y \cup \partial Z$. By applying this fact finitely many times in the definition of every $E_x$ with $|x| = n$, we obtain that $|\mathcal{F}|_n = \bigcup_{|x| = n} \partial E_x$ is contained in $\bigcup_{|x| = n} \partial D_x$, i.e. in a~finite union of closed subsets of~$X$ of dimension~$\leq k - 1$. By Theorem 1.5.3 in~\cite{E}, such union must have dimension $\leq k-1$, which proves that $\mathcal{F}$ is of dimension~$\leq k - 1$.

\textbf{(d)} follows immediately from~\textbf{(b)}. 

\textbf{(e)} Note first that, for every $1 \leq i \leq K$, we have
\[ \overline{D_{x_i}} = \overline{\widetilde{G}_i} \subseteq |\mathcal{C}|_{n_i} \setminus \widetilde{H}_i \subseteq X \setminus \overline{H_i} = \intr S_{x_i}. \]
Now, let $x \in G$ and denote $T^B_{\theta+1}(x) = i$ and $\gamma = xx_i^{-1}$. The left translation by~$\gamma$ is clearly a~homeomorphism mapping $S_{x_i}$ to~$S_x$ and $D_{x_i}$ to $D_x$; moreover, by Lemma~\ref{fakt-wym-homeo-na-calej-kuli}, it preserves interiors and closures computed within the appropriate spaces $|\mathcal{C}|_n$. Hence, we have:
\[ F_x \subseteq \overline{E_x} \subseteq \overline{D_x} = \gamma \cdot \overline{D_{x_i}} \subseteq \gamma \cdot \intr S_{x_i} = \intr S_x. \]

\textbf{(f)}
follows easily from \textbf{(b)}:
\[ |\mathcal{C}|_n \setminus |\mathcal{G}|_n = |\mathcal{C}|_n \setminus \bigcup_{|x| = n} G_x \subseteq \bigcup_{|x| = n} \big( E_x \setminus G_x \big) \subseteq \bigcup_{|x| = n} \partial E_x = |\mathcal{F}|_n.  \]

\textbf{(g)}
Let $T^B_{\theta+2}(x) = T^B_{\theta+2}(y)$ for some $x, y \in G$ and denote $\gamma = yx^{-1}$. By Lemma~\ref{fakt-wym-homeo-na-calej-kuli}, it is sufficient to check that $E_y = \gamma \cdot E_x$. We will show that $\gamma \cdot E_x \subseteq E_y$; the other inclusion is analogous.

Let $i = T^B_{\theta+1}(x) = T^B_{\theta+1}(y)$ and let $p \in E_x$. Then, in particular, $p \in D_x$, so
\[ \gamma \cdot p = yx_i^{-1} \cdot ( x_ix^{-1} \cdot p ) \in D_y. \]
Suppose that $\gamma \cdot p \notin E_y$; then, there must be some $y' \in G$ such that
\[ |y'| = |y|, \qquad T^B_{\theta+1}(y') < i, \qquad \gamma \cdot p \in D_{y'}. \]
In particular, we have $\emptyset \neq D_y \cap D_{y'} \subseteq S_y \cap S_{y'}$. By~Lemma~\ref{fakt-wym-typ-wyznacza-sasiadow}b, setting $x' = \gamma^{-1} y'$ we obtain that
\[ |x'| = |x|, \qquad T^B_{\theta+1}(x') = T^B_{\theta+1}(y'). \]
In particular, since $T^B_{\theta+1}(x') = T^B_{\theta+1}(y')$ and $x' = \gamma^{-1} y'$, we must have $D_{x'} = \gamma^{-1} \cdot D_{y'} \ni p$. This contradicts the assumption that $p \in E_x$ because $T^B_{\theta+1}(x') < T^B_{\theta+1}(x)$.
\end{proof}

\subsection{Weakly invariant neighbourhoods}
\label{sec-wym-kolnierzyki}

\begin{lem}
\label{lem-wym-kolnierzyki}
Let $k \geq 0$, $\theta \geq 0$ and suppose that:
\begin{itemize}
 \item $\mathcal{C} = \{ C_x \}_{x \in G}$ is a semi-closed, $\theta$-weakly invariant system;
 \item $\mathcal{D} = \{ D_x \}_{x \in G}$ is a closed, $(\theta+1)$-weakly invariant subsystem in~$\mathcal{C}$ of rank~$\leq k$.
\end{itemize}
Then, there exists an open, $(\theta+1)$-weakly invariant subsystem $\mathcal{G} = \{ G_x \}_{x \in G}$ in~$\mathcal{C}$ such that $\mathcal{G}$ covers $\mathcal{D}$ and moreover the system of closures $\overline{\mathcal{G}} = \{ \overline{G_x} \}_{x \in G}$ is $(\theta+1)$-weakly invariant and of rank~$\leq k$.
\end{lem}

\begin{uwaga}
Since $\mathcal{G}$ itself is claimed to be $(\theta+1)$-weakly invariant, the condition that the system of closures $\overline{\mathcal{G}}$ is $(\theta+1)$-weakly invariant reduces to the condition that $\overline{G_x} \subseteq S_x$ for every $x \in G$.
\end{uwaga}

\begin{proof}[Proof of Proposition~\ref{lem-wym-kolnierzyki}]
In the proof, we use the notations and conventions introduced in the beginning of~Section~\ref{sec-wym-wyscig}.
Also, we will frequently (and implicitly) use Lemma~\ref{fakt-wym-homeo-na-calej-kuli} to control the images of interiors/closures (taken ``in~$\mathcal{C}$'') under translations by elements of~$G$.

\textbf{1. }We choose by induction, for $1 \leq i \leq K$, an open subset $G_i \Subset S_{x_i}$ containing $D_{x_i}$ such that:
\begin{gather}
\begin{split}
\label{eq-wym-war-laty}
\textrm{for every } x, y \in G \textrm{ with } |x| = |y| \leq M, \, T^B_{\theta+1}(x) = i, \, T^B_{\theta+1}(y) = j \textrm{ and } D_x \cap D_y = \emptyset, \textrm{ we have:} \\
(xx_i^{-1} \cdot \overline{G_i}) \cap D_y = \emptyset, \textrm{ and moreover } (xx_i^{-1} \cdot \overline{G_i}) \cap (yx_j^{-1} \cdot \overline{G_j}) = \emptyset \textrm{ if } j < i. \qquad \ 
\end{split}
\end{gather}

Such choice is possible because we only require $G_i$ to be an open neighbourhood of~$D_{x_i}$ such that the closure $\overline{G_i}$ is disjoint with the union of sets of the following form:
\[ |\mathcal{C}|_{n_i} \setminus S_{x_i}, \qquad x_i x^{-1} \cdot D_y, \qquad x_i x^{-1} y x_j^{-1} \cdot \overline{G_j}. \]
Since this is a~finite union of closed sets (as we assume $|x|, |y| \leq M$), and we work in a~metric space, it suffices to check that each of these sets is disjoint from~$D_{x_i}$. 

In the case of~$|\mathcal{C}_{n_i}| \setminus S_{x_i}$, this is clear.
If we had $D_{x_i} \cap (x_i x^{-1} \cdot D_y) \neq \emptyset$ for some $x, y$ as specified above, it would follow that
\[ D_x \cap D_y = (x x_i^{-1} \cdot D_{x_i}) \cap D_y \neq \emptyset, \]
contradicting one of the assumptions in~\eqref{eq-wym-war-laty}. Similarly, if we had $D_{x_i} \cap (x_i x^{-1} y x_j^{-1} \cdot \overline{G_j}) \neq \emptyset$ for some $j < i$, then it would follow that $D_x \cap (y x_j^{-1} \cdot \overline{G_j}) \neq \emptyset$, which contradicts the assumption that we have (earlier) chosen $G_j$ to satisfy~\eqref{eq-wym-war-laty}.

\textbf{2. }Now, let
\[ G_x = x x_i^{-1} \cdot G_i \qquad \textrm{ for } \quad x \in G, \, T^B_{\theta+1}(x) = i. \]
Then, the system $\mathcal{G} = \{ G_x \}_{x \in G}$ is obviously open, $(\theta+1)$-weakly invariant and covers~$\mathcal{D}$. The fact that $\overline{\mathcal{G}}$ is also $(\theta+1)$-weakly invariant follows then from Lemma~\ref{fakt-wym-homeo-na-calej-kuli} (because $G_i \Subset S_{x_i}$, and hence $G_x \Subset S_x$ for $x \in G$). It remains to check that $\overline{\mathcal{G}}$ is of rank~$\leq k$. 

\textbf{3. }Let $x, y \in G$ be such that $|x| = |y|$ and $\overline{G_x} \cap \overline{G_y} \neq \emptyset$. Denote $T^B_{\theta+1}(x) = i$, $T^B_{\theta+1}(y) = j$. Let $x'$ be such that $|x'| \leq M$ and $T^B_{\theta+2}(x') = T^B_{\theta+2}(x)$. Since $\overline{G_x} \cap \overline{G_y} \neq \emptyset$ implies $S_x \cap S_y \neq \emptyset$, by Lemma~\ref{fakt-wym-typ-wyznacza-sasiadow}b (combined with~\qhlink{f1} for $\mathcal{G}$) we obtain that 
\[ T^B_{\theta+1}(y') = j, \qquad |y'| = |x'| \leq M, \qquad \overline{G_{x'}} \cap \overline{G_{y'}} = x'x^{-1} \cdot (\overline{G_x} \cap \overline{G_y}) \neq \emptyset, \qquad \quad \textrm{ where } \quad y' = x'x^{-1}y. \]
Then it follows that
\[ \emptyset \neq \overline{G_{x'}} \cap \overline{G_{y'}} = (x'x_i^{-1} \cdot \overline{G_i}) \cap (y'x_j^{-1} \cdot \overline{G_j}), \]
so the condition~\eqref{eq-wym-war-laty} implies that $D_{x'} \cap D_{y'} \neq \emptyset$. Then, since $\mathcal{D}$ is $(\theta+1)$-weakly invariant, we have
\[ D_x \cap D_y = xx'^{-1} \cdot (D_{x'} \cap D_{y'}) \neq \emptyset. \]
This means that, for $x, y \in G$ of equal length, $\overline{G_x}$ and $\overline{G_y}$ can intersect non-trivially only if $D_x$ and $D_y$ do so, whence it follows that the rank of~$\overline{\mathcal{G}}$ is not greater than that of~$\mathcal{D}$. This finishes the proof.
\end{proof}

\subsection{The overall construction}

The following proposition describes the whole inductive construction --- analogous to the one presented in the introduction to this section --- of a cover satisfying the conditions from Proposition~\ref{lem-wym}.

\begin{lem}
\label{lem-wym-cala-historia}
Let $k \geq -1$, $\theta \geq 1$ and let $\mathcal{C}$ be a semi-closed, $\theta$-weakly invariant system of dimension $\leq k$. 
Then, there exist $(\theta + 3(k+1))$-weakly invariant subsystems $\mathcal{D}$, $\mathcal{E}$ in~$\mathcal{C}$ of rank $\leq k+1$ which both cover~$\mathcal{C}$. Moreover, $\mathcal{D}$ is closed and $\mathcal{E}$ is open.
\end{lem}

\begin{proof}
We proceed by induction on~$k$. If $k = -1$, the system $\mathcal{C}$ must consist of empty sets, so we can set~$\mathcal{D} = \mathcal{E} = \mathcal{C}$.

Now, let $k > -1$. Denote $\Theta = \theta + 3(k+1)$. We perform the following steps:

\textbf{1. }By applying Proposition~\ref{lem-wym-bzdziagwy} to the system $\mathcal{C}$, we obtain some $(\theta + 2)$-weakly invariant systems $\mathcal{G} = \{ G_x \}_{x \in G}$ and $\mathcal{F}$ with additional properties described in the claim of the proposition.

\textbf{2. }Since $\mathcal{F}$ is closed (and then also semi-closed), $(\theta+2)$-weakly invariant and has dimension~$\leq k - 1$, it satisfies the assumptions of the current proposition (with parameters $k-1$ and $\theta+2$). Therefore, by the inductive hypothesis, there exists a $(\Theta-1)$-weakly invariant closed system $\mathcal{D}'$ of rank $\leq k$ which covers $\mathcal{F}$.

\textbf{3. }Since $\mathcal{C}$ is $\theta$-weakly invariant, it is also $(\Theta-2)$-weakly invariant, which means that the systems $\mathcal{C}$, $\mathcal{D}'$ satisfy the assumptions of Proposition~\ref{lem-wym-kolnierzyki} (with parameters $k$ and $\Theta-2$). Then, there exists an open, $(\Theta-1)$-weakly invariant subsystem $\mathcal{G}' = \{ G'_x \}_{x \in G}$ in~$\mathcal{C}$ which covers~$\mathcal{D}'$ and such that the system of closures $\mathcal{F}' = \{ \overline{G'_x} \}_{x \in G}$ is $(\Theta-1)$-weakly invariant and of rank~$\leq k$.

\textbf{4. }Now, we define two subsystems $\mathcal{D} = \{ D_x \}_{x \in G}$ and $\mathcal{E} = \{ E_x \}_{x \in G}$ as follows:
\begin{align} 
\label{eq-wym-konstr-wycinanka}
D_x = (G_x \setminus |\mathcal{G}'|_n) \cup F'_x, \qquad E_x = G_x \cup G'_x \qquad \textrm{ for } \quad x \in G, \, |x| = n. 
\end{align}
Observe that $G_x \setminus |\mathcal{G}'|_n$ is closed because the claim of Proposition~\ref{lem-wym-bzdziagwy} implies that $\partial G_x \subseteq |\mathcal{F}|_n \subseteq |\mathcal{D}'|_n \subseteq |\mathcal{G}'|_n$, so $G_x \setminus |\mathcal{G}'|_n = \overline{G_x} \setminus |\mathcal{G}'|_n$ is a difference of a closed and an open subset (in~$|\mathcal{C}|_n$). Hence, $D_x$ is closed. On the other hand, $E_x$ is clearly open in~$|\mathcal{C}|_n$.

Since $\mathcal{G}$, $\mathcal{F}'$ and~$\mathcal{G}'$ are all $(\Theta-1)$-weakly invariant, $\mathcal{D}$ and~$\mathcal{E}$ must both be $\Theta$-weakly invariant (more precisely: $\mathcal{E}$ is obviously $(\Theta-1)$-weakly and then also $\Theta$-weakly invariant, while for~$\mathcal{D}$ we apply Lemma~\ref{fakt-wym-homeo-na-calej-kuli}). It is also easy to see that
\[ |\mathcal{E}_n| \ = \ |\mathcal{G}|_n \cup |\mathcal{G}'|_n \ \supseteq \ |\mathcal{G}|_n \cup |\mathcal{F}|_n \ = \ |\mathcal{C}|_n, 
\qquad 
|\mathcal{D}|_n \ = \ \big( |\mathcal{G}|_n \setminus |\mathcal{G}'|_n \big) \cup |\mathcal{F}'|_n \ = \ |\mathcal{G}|_n \cup |\mathcal{F}'|_n \ \supseteq \ |\mathcal{E}|_n, \]
so $\mathcal{D}$ and~$\mathcal{E}$ both cover~$\mathcal{C}$. Finally, since $\mathcal{G}$ is disjoint and $\mathcal{F}'$ (and so also $\mathcal{G}'$) is of rank $\leq k$, it follows that $\mathcal{D}$ and~$\mathcal{E}$ must be of rank~$\leq k +1$ by Lemma~\ref{fakt-wym-suma-pokryc}.
\end{proof}

\subsection{Conclusion: The complete proof of Theorem~\ref{tw-kompakt}}

\label{sec-wymd-podsum}

\begin{proof}[{\normalfont \textbf{Proof of Proposition~\ref{lem-wym}}}]
The claim follows from applying Proposition~\ref{lem-wym-cala-historia} to the system~$\mathcal{S}$ (defined in Section~\ref{sec-konstr-pokrycia}). This is a semi-closed and~$0$-weakly (and hence also $1$-weakly) invariant system, so the proposition ensures that there exists an open $(3k+4)$-weakly invariant subsystem $\mathcal{E}$ of rank~$\leq k + 1$ which covers~$\mathcal{S}$. 

Since $\mathcal{S}$ is a system of covers, while~$\mathcal{E}$ is open and covers~$\mathcal{S}$, it follows that $\mathcal{E}$ is also a system of covers. Then, it follows from Proposition~\ref{lem-wym-slabe-sa-qi} that $\mathcal{E}$ is quasi-$G$-invariant (with the type function~$T^B_{3k+5}$). This means that $\mathcal{E}$ has all the desired properties.
\end{proof}

\begin{proof}[{\normalfont \textbf{Proof of Theorem~\ref{tw-kompakt}}}]
We use the quasi-$G$-invariant system of covers~$\mathcal{E}$ obtained in the proof of Proposition~\ref{lem-wym}. By Lemma~\ref{fakt-konstr-sp-zal}, there is $L \geq 0$ such that the system $(\widetilde{\mathcal{E}}_{Ln})_{n \geq 0}$ (where $\widetilde{\mathcal{E}}_n$ denotes $\mathcal{E}_n$ with empty members removed) is admissible. Then, by Theorems~\ref{tw-konstr} and~\ref{tw-kompakt-ogolnie}, the corresponding system of nerves $(K_n, f_n)$ is Markov, barycentric and has mesh property. 

Since the type function $T^B_{3k+5}$ associated with this system is stronger than $T^B$, Theorem~\ref{tw-sk-opis} ensures that the simplex types used in the system $(K_n, f_n)$ can be strengthened to make this system simultaneously Markov and has the distinct types property. (Barycentricity and mesh property are clearly preserved as the system itself does not change). Moreover, for every $n \geq 0$ we have
\[ \dim K_n = \rank \widetilde{\mathcal{E}}_{Ln} - 1 = \rank \mathcal{E}_{Ln} - 1 \leq \dim \partial G, \]
where the last inequality follows from the property of~$\mathcal{E}$ claimed by Proposition~\ref{lem-wym}. Finally, since $\mathcal{E}$ is $(3k+4)$-weakly invariant, it is in particular inscribed into~$\mathcal{S}$, which means in view of Theorem~\ref{tw-bi-lip} that the homeomorphism $\varphi : \partial G \simeq \liminv K_n$ obtained from Theorem~\ref{tw-konstr} is in fact a bi-Lipschitz equivalence (in the sense specified by Theorem~\ref{tw-bi-lip}).

This shows that the system~$(K_n, f_n)_{n \geq 0}$ has all the properties listed in Theorem~\ref{tw-kompakt}, which finishes the proof.
\end{proof}

\section{$\partial G$ as a semi-Markovian space}
\label{sec-sm}

The aim of this section is to show that the boundary $\partial G$ of a hyperbolic group~$G$ is a \textit{semi-Markovian space} (see Definition \ref{def-sm-ps}). In Section \ref{sec-sm-def}, we introduce notions needed to formulate the main result, which appears at its end as Theorem \ref{tw-semi-markow-0}. The remaining part of the section contains the proof of this theorem.

\subsection{Semi-Markovian sets and spaces}
\label{sec-sm-def}

Let $\Sigma$ be a finite alphabet and $\Sigma^\mathbb{N}$ denote the set of infinite words over $\Sigma$.

In the set $\Sigma$ we define the operations of \textit{shift} $S : \Sigma^\mathbb{N} \rightarrow \Sigma^\mathbb{N}$ and \textit{projection} $\pi_F : \Sigma^\mathbb{N} \rightarrow \Sigma^F$ (where $F \subseteq \mathbb{N}$) by the formulas:
\[ S \big( (a_0, a_1, \ldots) \big) = (a_1, a_2, \ldots), \qquad \pi_F \big( (a_0, a_1, \ldots) \big) = (a_n)_{n \in F}. \]

\begin{df}[{\cite[Chapter~2.3]{zolta}}]
\label{def-sm-cylinder}
A subset $C \subseteq \Sigma^\mathbb{N}$ is called a \textit{cylinder} if $C = \pi_F^{-1}(A)$ for some finite $F \subseteq \mathbb{N}$ and for some $A \subseteq \Sigma^F$. 

(Intuitively: the set $C$ can be described by conditions involving only a finite, fixed set of positions in the sequence $(a_n)_{n \geq 0} \in \Sigma^\mathbb{N}$).
\end{df}

\begin{df}[{\cite[Definition~6.1.1]{zolta}}]
\label{def-sm-zb}
A subset $M \subseteq \Sigma^\mathbb{N}$ is called a \textit{semi-Markovian set} if there exist cylinders $C_1$, $C_2$ in~$\Sigma^\mathbb{N}$ such that $M = C_1 \cap \bigcap_{n \geq 0}^\infty S^{-n}(C_2)$. 
\end{df}

\begin{uwaga}
\label{uwaga-sm-proste-zbiory}
In particular, for any subset $\Sigma_0 \subseteq \Sigma$ and binary relation $\rightarrow$ in $\Sigma$, the following set is semi-Markovian:
\[ M(\Sigma_0, \, \rightarrow) = \big\{ (a_n)_{n \geq 0} \,\big|\, a_0 \in \Sigma_0, \, a_n \rightarrow a_{n+1} \textrm{ for } n \geq 0 \big\}.  \]
\end{uwaga}

We consider the space of words $\Sigma^\mathbb{N}$ with the natural Cantor product topology (generated by the base of cylinders). In this topology, all semi-Markovian sets are closed subsets of $\Sigma^\mathbb{N}$.

Before formulating the next definition, we introduce a natural identification of pairs of words and words of pairs of symbols:
\[ J : \quad \Sigma^\mathbb{N} \times \Sigma^\mathbb{N} \quad \ni \quad \Big( \big( (a_n)_{n \geq 0}, \, (b_n)_{n \geq 0} \big) \Big) \qquad \mapsto \qquad \big( (a_n, \, b_n) \big)_{n \geq 0} \quad \in \quad (\Sigma \times \Sigma)^\mathbb{N}. \]

\begin{df}
\label{def-sm-rel}
A binary relation $R \subseteq \Sigma^\mathbb{N} \times \Sigma^\mathbb{N}$ will be called a \textit{semi-Markovian relation} if its image under the above identification $J(R) \subseteq (\Sigma \times \Sigma)^\mathbb{N}$ is a semi-Markovian set (over the product alphabet $\Sigma \times \Sigma$).
\end{df}

\begin{df}[{\cite[Definition~6.1.5]{zolta}}]
\label{def-sm-ps}
A topological Hausdorff space $\Omega$ is called a \textit{semi-Markovian space} if it is the topological quotient of a semi-Markovian space (with the Cantor product topology) by a semi-Markovian equivalence relation.
\end{df}

We can now re-state the main result of this section:

\begin{twsm}
The boundary of any hyperbolic group $G$ is a semi-Markovian space.
\end{twsm}

The proof of Theorem \ref{tw-semi-markow-0} --- preceded by a number of auxiliary facts --- is given at the end of this section. Roughly, it will be obtained by applying Corollary~\ref{wn-sm-kryt} to the \textit{$C$-type} function which will be defined in Section~\ref{sec-sm-abc-c}.

\begin{uwaga}
 Theorem~\ref{tw-semi-markow-0} has been proved (in~\cite{zolta}) under an additional assumption that $G$ is torsion-free.
We present a proof which does not require this assumption; the price for it is that our reasoning (including the results from Section~\ref{sec-abc}, which will play an important role here) is altogether significantly more complicated.

However, in the case of torsion-free groups, these complications mostly trivialise (particularly, so does the construction of $B$-type) --- and the remaining basic structure of the reasoning (summarised in Lemma \ref{fakt-sm-kryt-zbior}) is analogous to that in the proof from~\cite{zolta}. Within this analogy, a key role in our proof is played by Proposition~\ref{lem-sm-kuzyni}, corresponding to Lemma 7.3.1 in~\cite{zolta} which particularly requires $G$ to be torsion-free. More concrete remarks about certain problems related with the proof from \cite{zolta} will be stated later in Remark~\ref{uwaga-zolta}.
\end{uwaga}

\begin{uwaga}
Theorem \ref{tw-semi-markow-0} can be perceived as somewhat analogous to a known result stating automaticity of hyperbolic groups (described for example in \cite[Theorem~12.7.1]{CDP}).
The relation between those theorems seems even closer if we notice that --- although the classical automaticity theorem involves the Cayley graph of a group --- it can be easily translated to an analogous description of the Gromov boundary. Namely, the boundary is the quotient of some ``regular'' set of infinite words by some ``regular'' equivalence relation (in a sense analogous to Definition \ref{def-sm-rel}) where ``regularity'' of a set $\Phi \subseteq \Sigma^\mathbb{N}$ means that there is a finite automaton $A$ such that any infinite word $(a_n)_{n \geq 0}$ belongs to $\Phi$ if and only if $A$ accepts all its finite prefixes.

However, such regularity condition is weaker than the condition from Definition \ref{def-sm-zb}, as the following example shows:
\[ \Phi = \big\{ (x_n)_{n \geq 0} \in \{ a, b, c \}^\mathbb{N} \ \big|\ \forall_n \ (x_n = b \ \Rightarrow \ \exists_{i < n} \ x_i = a) \big\}, \qquad \quad A: \ \ 
\raisebox{-3.5ex}{
\begin{tikzpicture}[scale=0.15]
 \node[draw, circle, double] (s0) at (0, 0) {};
 \node[draw, circle, double] (s1) at (10, 0) {};
 \node[draw, circle] (s2) at (20, 0) {};
 \draw[->] (-3,0) -- (s0);
 \draw[->] (s0) -- (s1) node [midway, above, draw=none, inner sep=2] {\footnotesize $a$};
 \draw (s0) edge [in=110,out=70,loop] node [midway, above, inner sep=2] {\footnotesize $c$} ();
 \draw (s1) edge [in=110,out=70,loop] node [midway, above, inner sep=2] {\footnotesize $a,b,c$} ();
 \draw (s2) edge [in=110,out=70,loop] node [midway, above, inner sep=2] {\footnotesize $a,b,c$} ();
 \draw[->] (s0) edge[bend right=20] node [midway, sloped, below, inner sep=2] {\footnotesize $b$} (s2);
\end{tikzpicture}
}
\]
It is easy to check that the set $\Phi$ corresponds to the automaton $A$ in the sense described above, while it is not semi-Markovian.
The latter claim can be shown as follows. Assume that there exists a presentation  $\Phi = C_1 \cap \bigcap_{n \geq 0} S^{-n} (C_2)$ as required by Definition \ref{def-sm-zb}, and let the cylinder $C_1$ have form $\pi_F^{-1}(A)$, according to Definition \ref{def-sm-cylinder}. Denote by $N$ the maximal element of $F$. Then, the word $\alpha_1 = \underbrace{cc\ldots{}cc}_{N+1}\underbrace{aa\ldots{}aa}_{N+1}\underbrace{cc\ldots{}cc}_{N+1}bb\ldots$ belongs to $\Phi$, while $\alpha_2 = \underbrace{cc\ldots{}cc}_{N+1}bb\ldots$ does not belong to $\Phi$. However, these words have a common prefix of length $N+1$ (so $\alpha_2$ cannot be rejected by $C_1$) and moreover $\alpha_2$ is a suffix of $\alpha_1$ (so $\alpha_2$ cannot be rejected by $C_2$).
\end{uwaga}

\subsection{Compatible sequences}
\label{sec-sm-nici}

In the remainder of Section \ref{sec-sm}, we work under the assumptions formulated in the introduction to Section \ref{sec-abc}.

\begin{ozn}
For any $n \geq 0$, we denote
\[ G_n = \big\{ x \in G \,\big|\, |x| = n \big\}. \]
\end{ozn}

\begin{df}
An infinite sequence $(g_n)_{n \geq 0}$ in~$G$ will be called \textit{compatible} if, for all $n \geq 0$, we have $g_n \in G_{Ln}$ and $g_n = g_{n+1}^\Uparrow$. We denote the set of all such sequences by $\mathcal{N}$.
\end{df}

Note that to any compatible sequence $(g_n)_{n \geq 0}$ we can naturally assign a geodesic $\alpha$ in~$G$, defined by the formula
\[ \alpha(m) = g_n^{\uparrow Ln-m} \qquad \textrm{ for } \quad m, n \geq 0, \  Ln \geq m. \]  
(It is easy to check that the value $g_n^{\uparrow Ln-m}$ does not depend on the choice of $n$). This means that every compatible sequence $(g_n)$ has a \textit{limit}: it must converge in~$G \cup \partial G$ to the element $\lim_{m \rightarrow \infty} \alpha(m) = [\alpha] \in \partial G$.

\begin{fakt}
\label{fakt-sm-nici-a-punkty}
Let the map $I : \mathcal{N} \rightarrow \partial G$ assign to a compatible sequence $(g_n)_{n \geq 0}$ its limit in $\partial G$. Then:
\begin{itemize}
\item[\textbf{(a)}] $I$ is surjective;
\item[\textbf{(b)}] for every $(g_n), (h_n) \in \mathcal{N}$, we have
\[ I \big( (g_n) \big) = I \big( (h_n) \big) \qquad \Longleftrightarrow \qquad g_n \leftrightarrow h_n \quad \textrm{ for every } n \geq 0. \]
\end{itemize}
\end{fakt}

\begin{proof}

\textbf{(a)} Let $x \in \partial G$ and let $\alpha$ be any infinite geodesic going from $e$ towards $x$. For $k \geq 0$, we define a geodesic $\alpha_k$ in~$G$ by the formula
\[ \alpha_k(n) = \begin{cases}
                  \alpha(k)^{\uparrow k-n} & \textrm{ for } n \leq k, \\
                  \alpha(n) & \textrm{ for } n \geq k.
                 \end{cases}
 \]
(For $n = k$, both branches give the same result). Then, for any $n \geq 0$, we have $|\alpha_k(n)| = n$ and moreover $d \big( \alpha_k(n), \alpha_k(n+1) \big) = 1$, which proves that $\alpha_k$ is a geodesic. Moreover, we have $\alpha_k(0) = e$ and $\lim_{n \rightarrow \infty} \alpha_k(n) = x$ because $\alpha_k$ ultimately coincides with $\alpha$. 

Using Lemma~\ref{fakt-geodezyjne-przekatniowo} to the sequence $(\alpha_k)$, we obtain some subsequence $(\alpha_{k_i})$ and a geodesic $\alpha_\infty$ such that $\alpha_\infty$ coincides with~$\alpha_{k_i}$ on the segment~$[0, i]$. By choosing a subsequence if necessary, we can assume that $k_i \geq i$; then we have $\alpha_\infty(i-1) = \alpha_\infty(i)^\uparrow$ for every $i \geq 1$, so the sequence $\big( \alpha_\infty(Ln) \big)_{n \geq 0}$ is compatible. On the other hand, Lemma \ref{fakt-geodezyjne-przekatniowo} also ensures that $I \big( (g_n) \big) = [\alpha_\infty] = \lim_{i \rightarrow \infty} [\alpha_{k_i}] = x$, which proves the claim. 

\textbf{(b)} The implication $(\Rightarrow)$ follows directly from the inequality (1.3.4.1) in~\cite{zolta}. On the other hand, if $g_n \leftrightarrow h_n$ for every $n \geq 0$, and if $\alpha, \beta$ are the geodesics corresponding to the compatible sequences $(g_n)$ and~$(h_n)$, then we have
\[ d \big( \alpha(Ln), \beta(Ln) \big) = d(g_n, h_n) \leq 8\delta \qquad \textrm{ for } \quad n \geq 0, \]
so from the triangle inequality we deduce that $d \big( \alpha(m), \beta(m) \big) \leq 2L + 8\delta$ for all $m \geq 0$, and so in $\partial G$ we have $[\alpha] = [\beta]$.
\end{proof}

\subsection{Desired properties of the type function}
\label{sec-sm-zyczenia}

The presentation of $\partial G$ as a semi-Markovian space will be based on an appropriate type function (see the introduction to Section \ref{sec-abc}). Since the ball type $T^b_N$ used in the previous sections has too weak properties for our needs, we will use some its strengthening. In this section, we state (in Lemma \ref{fakt-sm-kryt-zbior}) a list of properties of a type function which are sufficient (as we will prove in Corollary \ref{wn-sm-kryt}) to give a semi-Markovian structure on $\partial G$. The construction of a particular function $T^C$ satisfying these conditions will be given in Section \ref{sec-sm-abc-c}.

\begin{df}
Let $T$ be any type function in $G$ with values in a~finite set $\mathcal{T}$. For a compatible sequence $\nu = (g_n)_{n \geq 0} \in \mathcal{N}$, we call its \textit{type} $T^*(\nu)$ the sequence $\big( T(g_n) \big)_{n \geq 0} $.
\end{df}

Then, using the definition of a semi-Markovian space, it is easy to show the following lemma.

\begin{fakt}
\label{fakt-sm-kryt-zbior}
Let $T$ be a type function in $G$ with values in $\mathcal{T}$. Then:
\begin{itemize}
 \item[\textbf{(a)}] If, for every element of~$G$, all its p-grandchildren have pairwise distinct types, then the function $T^* : \mathcal{N} \rightarrow \mathcal{T}^{\mathbb{N}}$ is injective;
 \item[\textbf{(b)}] If the set of p-grandchildren of $g \in G$ depends only on the type of $g$, then the image of $T^*$ is a semi-Markovian set over $\mathcal{T}$;
 \item[\textbf{(c)}] Under the assumptions of parts \textbf{(a)} and \textbf{(b)}, if for any $g, g' \in G_{L(n+1)}$, $h, h' \in G_{L(m+1)}$ the conditions
 \begin{gather*} 
 T(g) = T(h), \qquad T(g') = T(h'), \qquad T(g^\Uparrow) = T(h^\Uparrow), \qquad T({g'}^\Uparrow) = T({h'}^\Uparrow), \\
 g^\Uparrow \leftrightarrow g'^\Uparrow, \qquad g \leftrightarrow g', \qquad h^\Uparrow \leftrightarrow {h'}^\Uparrow,
 \end{gather*}
 imply that $h \leftrightarrow h'$, then the equivalence relation $\sim$ in the set $T^*(\mathcal{N})$, given by the formula \linebreak \mbox{$T^*(\nu) \sim T^*(\nu') \ \Leftrightarrow \ I(\nu) = I(\nu')$}, is a semi-Markovian relation.
\end{itemize}
\end{fakt}

\begin{proof}
Part \textbf{(a)} is clear. If the assumption of part~\textbf{(b)} holds, it is easy to check that
$T^*(\mathcal{N}) = M(\Sigma_0, \rightarrow)$, where $\Sigma_0 = \{ T(e) \}$ and $\tau \rightarrow \tau'$ if and only if $\tau = T(g^\Uparrow)$ and $\tau' = T(g)$ for some $n \geq 0$ and $g \in G_{L(n+1)}$.

Analogously, it is easy to check that, under the assumptions of part~\textbf{(c)}, the relation $\sim$ has the form $M(A_0, \leadsto)$, where
\begin{gather*}
 A_0 = \big\{ \big( T(e), T(e) \big) \}, \\
 \big( T(g^\Uparrow), T({g'}^\Uparrow) \big) \ \leadsto \ \big( T(g), T(g') \big) \qquad \textrm{ for } \quad g, g' \in G_{L(n+1)}, \ g^\Uparrow \leftrightarrow g'^\Uparrow, \, g \leftrightarrow g', \ n \geq 0.
\end{gather*}
Indeed: the containment $\sim \subseteq M(A_0, \leadsto)$ results from Lemma \ref{fakt-sm-nici-a-punkty}b. On the other hand, if a sequence $\big( (\tau_n, \tau_n') \big)_{n \geq 0}$ belongs to $M(A_0, \leadsto)$, then the sequences $(\tau_n)_{n \geq 0}$ and $(\tau'_n)_{n \geq 0}$ belong to the set $M(\Sigma_0, \rightarrow)$
defined in the previous paragraph, so they are types of some compatible sequences $(h_n)_{n \geq 0}$ and correspondingly $(h'_n)_{n \geq 0}$. Moreover, it is easy to check by induction that $h_n \leftrightarrow h_n'$ for every $n \geq 0$: for $n = 0$ this holds since $h_0 = h'_0 = e$, and for $n > 0$ one can use the relation $h_{n-1} \leftrightarrow h'_{n-1}$, the condition $(\tau_{n-1}, \tau'_{n-1}) \leadsto (\tau_n, \tau'_n)$ and the assumptions of part~\textbf{(c)}. Therefore, we obtain that $h_n \leftrightarrow h_n'$ for $n \geq 0$, and so by Lemma~\ref{fakt-sm-nici-a-punkty}b we deduce that $(\tau_n) \sim (\tau_n')$.
\end{proof}

\begin{wn}
\label{wn-sm-kryt}
Under the assumptions of parts~\textbf{(a-c)} in Lemma~\ref{fakt-sm-kryt-zbior}, $\partial G$ is a semi-Markovian space.
\end{wn}

\begin{proof}
Since the map $I \circ (T^*)^{-1} : T^*(\mathcal{N}) \rightarrow \partial G$ is surjective by Lemma \ref{fakt-sm-nici-a-punkty}a, to verify that it is a homeomorphism we need only to check its continuity. Let $(\tau_n^{(i)}) \mathop{\longrightarrow}\limits_{i \rightarrow \infty} (\tau_n)$ in the space $T^*(\mathcal{N})$; this means that there exists a sequence $n_i \rightarrow \infty$ such that for every $i \geq 0$ the sequences $(\tau_n^{(i)})$ and $(\tau_n)$ coincide on the first $n_i$ positions. Then, the assumptions of part~\textbf{(a)} imply that the corresponding compatible sequences $(g_n^{(i)})$, $(g_n)$ also coincide on the first $n_i$ positions; in particular, $g_{n_i}^{(i)} = g_{n_i}$. Then, also the geodesics $\alpha^{(i)}$ corresponding to the sequences $(g_n^{(i)})$ are increasingly coincident with the geodesic $\alpha$ corresponding to the sequence $(g_n)$, which means by the definition of $\partial G$ that $I \big( (g_n^{(i)}) \big) = [\alpha^{(i)}] \rightarrow [\alpha] = I \big( (g_n) \big)$ in $\partial G$.
\end{proof}

\begin{uwaga}
\label{uwaga-zolta}
The main ``skeleton'' of the proof of Theorem \ref{tw-semi-markow-0} presented in Lemma \ref{fakt-sm-kryt-zbior} is taken from \cite{zolta}. The proof given there uses the ball type $T^b_N$ (defined in Section~\ref{sec-typy-kulowe}) as the type function, and $1$ as the value of~$L$. However, in the case of a torsion group, this type function does not have to satisfy the assumptions of part~\textbf{(a)} in Lemma \ref{fakt-sm-kryt-zbior}. Moreover, even in the torsion-free case, the verification of the assumptions of part~\textbf{(b)} ---  given in \cite{zolta} on page 125 (Chapter~7, proof of Proposition~2.4) --- contains a~defect in line 13. More precisely, it is claimed there that if one takes $N$ sufficiently large, $L = 1$ and $x', y' \in G$ such that $y'$ ``follows'' (in our terms: is a~child of) $x'$, then every element of the set $B_N(y') \setminus B_N(x')$ ``can be considered as belonging to the tree $T_{geo,x'}$'' (in our terms: to $x' T^c(x')$). 
Our approach avoids this problem basically by choosing $L$ and $N$ so large that the analogous claim must indeed hold, as one could deduce e.g. from Lemma~\ref{fakt-kulowy-duzy-wyznacza-maly} combined with the proof of Proposition~\ref{lem-potomkowie-dla-kulowych}.
\end{uwaga}

\subsection{Extended types and the $C$-type}

\label{sec-sm-abc-c}

\begin{df}
\label{def-sm-typ-plus}
Let $T$ be an arbitrary type function in $G$ with values in a finite set $\mathcal{T}$ and let $r \geq 0$. Let $g \in G$, and let $P_r(g)$ denote the set of $r$-fellows of $g$ (see Definition \ref{def-konstr-towarzysze}). We define the \textit{extended type} of element $g$ as the function $T^{+r}(g) : P_r(g) \rightarrow \mathcal{T}$, defined by the formula:
\[ \big( T^{+r}(g) \big)(h) = T(gh) \qquad \textrm{ for } \quad h \in P_r(g). \]
\end{df}

Since $P_r(g)$ is contained in a bounded ball $B(e, r)$, the extended type function $T^{+r}$ has, in an obvious way, finitely many possible values.

\begin{df}
\label{def-sm-typ-C}
We define the \textit{$C$-type} of an element $g \in G$ as its $B$-type extended by $8\delta$:
\[ T^C(g) = (T^B)^{+8\delta}(g) \qquad \textrm{ for } g \in G. \]
\end{df}

Note that, by comparing Definitions~\ref{def-konstr-towarzysze} and~\ref{def-sm-sasiedzi}, we obtain that the set $P_{8\delta}(g)$ contains exactly these $h \in G$ for which $g \leftrightarrow gh$. This means that the $C$-type of $g$ consists of the $B$-type of~$g$ and of $B$-types of its neighbours (together with the knowledge about their relative location).

\begin{fakt}
\label{fakt-sm-C-bracia}
For every $g \in G$, all p-grandchildren of $g$ have pairwise distinct $C$-types.
\end{fakt}

\begin{proof}
By definition, the $C$-type of an element $h \in G$ contains its $B$-type, which in turn contains its $A$-type and finally its descendant number $n_h$, which by definition distinguishes all the p-grandchildren of a fixed element $g \in G$.
\end{proof}

\begin{lem}
\label{lem-sm-C-dzieci}
The set of $C$-types of all p-grandchildren of a given element $g \in G$ depends only on $T^C(g)$.
\end{lem}

\begin{proof}
Let $g_1, h_1 \in G$ satisfy $T^C(g_1) = T^C(h_1)$; denote $\gamma = h_1g_1^{-1}$. By Lemma \ref{fakt-sm-przen-a} we know that the left translation by $\gamma$
gives a bijection between p-grandchildren of $g_1$ and p-grandchildren of $h_1$. Let $g_2$ be a p-grandchild of $g_1$ and $h_2 = \gamma g_2$; our goal is to prove that $T^C(g_2) = T^C(h_2)$. For this, choose any $g_2' \in G$ such that $g_2 \leftrightarrow g_2'$; we need to prove that $T^B(g_2') = T^B(h_2')$, where $h_2' = \gamma g_2'$.

Denote $g_1' = g_2'^\Uparrow$ and $h_1' = \gamma g_1'$. Since $g_2 \leftrightarrow g_2'$, by Lemma~\ref{fakt-sm-rodzic-kuzyna} we have $g_1 \leftrightarrow g_1'$; then from the equality $T^C(g_1) = T^C(h_1)$ we obtain that $T^B(g_1') = T^B(h_1')$. In this situation, Proposition \ref{lem-sm-B-dzieci} ensures that $T^B(g_2') = T^B(h_2')$, q.e.d.
\end{proof}

\begin{lem}
\label{lem-sm-C-klejenie}
The type function $T^C$ satisfies the condition stated in part~\textbf{(c)} of Lemma \ref{fakt-sm-kryt-zbior}.
\end{lem}

\begin{proof}
Let $g, g', h, h'$ be as in part~\textbf{(c)} of Lemma \ref{fakt-sm-kryt-zbior}. In particular, we assume that  $T^C(g^\Uparrow) = T^C(h^\Uparrow)$. By the definition of $C$-type, this means that the left translation by $\gamma = h^{-1}g$ neighbours of $g$ to neighbours of $h$, preserving their $B$-type, which in turn implies by Proposition \ref{lem-sm-B-dzieci} that this shift preserves the children of these neighbours, together with their $B$-types. In particular:
\begin{itemize} 
 \item the element ${g'}^\Uparrow$ must be mapped to ${h'}^\Uparrow$ since by assumption we have $T^B({g'}^\Uparrow) = T^B({h'}^\Uparrow)$, and moreover ${h'}^\Uparrow$ is the only neighbour of $h^\Uparrow$ with the appropriate $B$-type (by Proposition~\ref{lem-sm-kuzyni});
 \item the elements $g$, $g'$ must be mapped correspondingly to $h$, $h'$ since by assumption the corresponding $B$-types coincide, and moreover $h$, $h'$ are the only p-grandchildren of $h^\Uparrow$, ${h'}^\Uparrow$ with the appropriate $B$-types (because by Remark~\ref{uwaga-sm-B-wyznacza-A} the $B$-type determines the $A$-type, which in turn distinguishes all the p-grandchildren of a given element).
\end{itemize}
Therefore, we have $d(h, h') = d(\gamma g, \gamma g') = d(g, g') \leq 8\delta$, q.e.d.
\end{proof}

\begin{proof}[{\normalfont \textbf{Proof of Theorem~\ref{tw-semi-markow-0}}}]
By Corollary~\ref{wn-sm-kryt} it suffices to ensure that the type~$T^C$ (which is finitely valued by Lemma~\ref{fakt-sm-B-skonczony} and Definition~\ref{def-sm-typ-C}) satisfies the conditions \textbf{(a-c)} from Lemma~\ref{fakt-sm-kryt-zbior}. The conditions of parts~\textbf{(b)} and~\textbf{(c)} follow correspondingly from Propositions~\ref{lem-sm-C-dzieci} and~\ref{lem-sm-C-klejenie}, while the condition of part~\textbf{(a)} follows from the fact that, by definition, the value $T^C(x)$ for a given $x \in G$ determines $T^B(x)$ and further $T^A(x)$, while the $A$-types of all p-grandchildren of a given element are pairwise distinct by definition. This finishes the proof.
\end{proof}

\bibliographystyle{plain}
\bibliography{markow-nowy3.bib}

\end{document}